\documentclass[a4paper]{article}
\title{Statistical Learning Theory for Neural Operators}
\author{Jakob Zech, Sven Wang, Niklas Reinhardt}

\date{\today}
\usepackage[ngerman,english]{babel} 
\usepackage{textgreek}
\usepackage{leftidx}
\usepackage{enumitem}
\usepackage{amssymb}
\usepackage{amsmath}
\usepackage{csquotes}
\usepackage{eucal}
\usepackage{authblk}
\usepackage{bbold} 
\usepackage[toc,page]{appendix}
\usepackage{a4wide}
\usepackage[textsize=footnotesize]{todonotes}
\usepackage[normalem]{ulem}

\usepackage{comment}

\usepackage{epigraph}
\usepackage{biblatex}
\usepackage{xcolor}

\usepackage[colorlinks,        
   linkcolor=black,  
   filecolor=black,
   citecolor=black]{hyperref}

\usepackage{xurl}
\hypersetup{breaklinks=true}

\usepackage{pdflscape}
\usepackage{afterpage}
\usepackage{subcaption}
\usepackage{graphicx}
\usepackage{verbatim}

\usepackage{stackengine} 
\usepackage{mathtools}
\mathtoolsset{showonlyrefs=true}
\usepackage{amsthm}

\newtheorem{assumption}{Assumption}
\newtheorem{theorem}{Theorem}[section]
\newtheorem{proposition}[theorem]{Proposition}
\newtheorem{corollary}[theorem]{Corollary}
\newtheorem{lemma}[theorem]{Lemma}
\newtheorem{definition}[theorem]{Definition}
\newtheorem{remark}[theorem]{Remark}
\newtheorem{example}[theorem]{Example}

\def\biggg#1{{\hbox{$\left#1\vbox to20.5pt{}\right.$}}}
\def\bigggl{\mathopen\biggg}
\def\bigggr{\mathclose\biggg}

\DeclareMathOperator*{\argmin}{arg\,min}
\DeclareMathOperator\supp{supp}
\DeclareMathOperator\Id{Id}

\DeclareMathOperator\range{range}
\numberwithin{equation}{section}
\newcommand{\GNNSP}[1][\sigma_q]{\pmb{G}_{{\rm FN}}^{{\rm sp}}(#1,N)}

\newcommand{\GNNFC}[1][\sigma_q]{\pmb{G}_{{\rm FN}}^{{\rm full}}(#1,N)}
\newcommand{\CCh}{C_{\rm{Ch}}}

\newcommand{\ltg}{{L^2(\gamma)}}
\newcommand{\Fnorm}{\lVert F\rVert}
\newcommand{\Fsjnorm}{\lVert F_{s,j}\rVert}
\newcommand{\Fsjxinorm}{\lVert F_{s,j}(x_i)\rVert}
\newcommand{\Akp}{A_k^+}
\newcommand{\Akm}{A_k^-}
\newcommand{\sigmabk}{\sigma^{B_k}}
\newcommand{\sumin}{\sum\limits_{i=1}^n}
\newcommand{\hatgn}{\hat{G}_n}
\newcommand{\pgnstar}{\pG_n^*}
\newcommand{\gstar}{G^*}

\newcommand{\tildedeltan}{\tilde{\delta}_n}

\newcommand{\X}{\mathcal{X}}
\newcommand{\Y}{\mathcal{Y}}
\newcommand{\cZ}{\mathcal{Z}}
\newcommand{\HH}{\mathcal{H}}
\newcommand{\G}{G}
\newcommand{\dd}{\;\mathrm{d}}
\newcommand{\E}{\mathcal{E}}
\newcommand{\D}{\mathcal{D}}
\newcommand{\F}{\mathcal{F}}
\newcommand{\NN}{\mathcal{N}}
\newcommand{\N}{\mathbb{N}}

\newcommand{\B}{\mathcal{B}}

\renewcommand{\O}{\mathcal{O}}

\newcommand{\IR}{\mathbb{R}}
\newcommand{\IN}{\mathbb{N}}

\newcommand{\IE}{\mathbb{E}}
\newcommand{\IP}{\mathbb{P}}
\newcommand{\IC}{\mathbb{C}}

\newcommand{\cX}{\mathcal{X}}

\newcommand{\cY}{\mathcal{Y}}

\newcommand{\cE}{\mathcal{E}}
\newcommand{\cD}{\mathcal{D}}

\newcommand{\norm}[2][]{\|#2\|_{#1}}

\newcommand{\set}[2]{\{#1 : #2\}}

\renewcommand{\O}{\mathcal{O}}

\newcommand{\eps}{\varepsilon}

\newcommand{\pnu}{\pmb{\nu}}
\newcommand{\pmu}{\pmb{\mu}}

\newcommand{\py}{\pmb{y}}
\newcommand{\px}{\pmb{x}}
\newcommand{\pg}{\pmb{g}}
\newcommand{\pG}{\pmb{G}}
\newcommand{\pF}{\pmb{F}}

\newcommand{\Finf}{F_{\infty}}

\newcommand{\bsj}{{\boldsymbol{j}}}

\newcommand{\bsPsi}{{\boldsymbol{\Psi}}}

\newcommand{\bbE}{{\mathbb{E}}}

\newcommand{\bbP}{{\mathbb{P}}}

\newcommand{\bbT}{{\mathbb{T}}}

\newcommand{\R}{{\mathbb{R}}} %

\DeclareSymbolFont{bbold}{U}{bbold}{m}{n}
\DeclareSymbolFontAlphabet{\mathbbold}{bbold}

\newcommand{\trm}[1]{\textrm{#1}}

\newcommand\sbullet[1][.5]{\mathbin{\vcenter{\hbox{\scalebox{#1}{$\bullet$}}}}}

\DeclareMathOperator*{\essinf}{ess\,inf}
\usepackage{enumitem}

\begin{document}
\author[1]{Niklas Reinhardt}
\author[2]{Sven Wang}
\author[1]{Jakob Zech}

\affil[1]{\small Interdisziplin\"ares Zentrum f\"ur wissenschaftliches Rechnen, Universit\"at Heidelberg, Im Neuenheimer Feld 205, 69120 Heidelberg, Germany
  
  \texttt{niklas.reinhardt@iwr.uni-heidelberg.de}\qquad
    \texttt{jakob.zech@uni-heidelberg.de}}

\affil[2]{\small Institut für Mathematik,
Humboldt-Universit\"at zu Berlin, Rudower Chaussee 25, 12489 Berlin, Germany
  
  \texttt{sven.wang@hu-berlin.de}}

\maketitle
\begin{abstract}
  We present statistical convergence results for the learning of (possibly) non-linear mappings in infinite-dimen\-sional spaces. Specifically, given a 
  map $G_0:\mathcal X\to\mathcal Y$ between two separable Hilbert spaces, we analyze the problem of recovering $G_0$ from $n\in\mathbb N$ noisy input-output pairs $(x_i, y_i)_{i=1}^n$ with $y_i = G_0 (x_i)+\varepsilon_i$; here the $x_i\in\mathcal X$ represent randomly drawn ``design'' points, and the $\varepsilon_i$ are assumed to be either i.i.d.\ white noise processes or subgaussian random variables in $\mathcal{Y}$.
  We provide general convergence results for least-squares-type empirical risk minimizers over compact regression classes $\mathbf G\subseteq L^\infty(X,Y)$, in terms of their approximation properties and metric entropy bounds, which are derived using empirical process techniques. This generalizes classical results from finite-dimensional nonparametric regression to an infinite-dimensional setting.
  As a concrete application, we study an encoder-decoder based neural operator architecture termed FrameNet.
  Assuming $G_0$ to be holomorphic, we prove algebraic (in the sample size $n$) convergence rates in this setting, thereby overcoming the curse of dimensionality.
  To illustrate the wide applicability, as a prototypical example we discuss the learning of the non-linear solution operator to a parametric elliptic partial differential equation.
\end{abstract}

\newcommand{\By}{{\boldsymbol{y}}}
\newcommand{\Bnu}{{\boldsymbol{\nu}}}
\tableofcontents

\section{Introduction}
Learning non-linear relationships of high- and infinite-dimensional data is a fundamental problem in modern statistics and machine learning. In recent years, ``Operator Learning'' has emerged as a powerful tool for analyzing and approximating mappings $G_0$ between \emph{infinite-dimensional} spaces \cite{li2020fourier,HESTHAVEN201855,BhattModRedNNPDE,Lu2021,raonic2023convolutional,anandkumar2019neural,owhadi2019kernel,nelsen2024operator,kovachki2024operator}.
The primary motivation for considering truly infinite-dimensional data stems from applications in the natural sciences, where inputs and outputs of operators are elements in function spaces. For instance, $G_0$ could be the 
operator relating an initial condition $x$ of a dynamical system to the state $G_0(x)$ of the system after a certain time, or a coefficient-to-solution map of a parametric partial differential equation (PDE).

For finite-dimensional inputs and outputs,
nonparametric regression is the standard framework for inferring general, non-linear relationships. There, one aims to reconstruct some ``ground truth''
$G_0:\IR^d\to\IR^m$, $d$, $m\in\IN$, from noisy data $(x_i,y_i)\in \IR^d\times\IR^m$, $i=1,\dots, n$, generated via $y_i=G_0(x_i)+\eps_i$, where $x_i$ are called the ``design points'' and $\eps_i$ are typically independent and identically distributed (i.i.d.) noise variables. In the framework of empirical risk minimization (ERM), one chooses a suitable function class $\pG$ of mappings from $\IR^d$ to $\IR^m$ and some loss function $L:\R^m\times\R^m\to\R$ measuring the discrepancy between predictions $G(x_i)$ and the data $y_i$. Statistical estimation is achieved by minimizing 
\begin{equation}\label{eq:finiteregression}
\hat{G}_n\in \argmin_{G\in\pG}\,J_n(G),\quad J_n(G):=\frac{1}{n}\sum_{i=1}^n L(G(x_i),y_i),
\end{equation}
assuming that minimizers exist. In the finite-dimensional setting, statistically optimal convergence rates for such estimators were established for least-squares, maximum likelihood, and more generally ``minimum contrast'' estimators, in \cite{VDG00, BBM99, BM93}; see also \cite{SH20} where such results are 
shown for ERMs over neural network classes.
However, as is well-known, both---approximation rates \cite{DeVore1989,devore1993constructive} as well as statistical convergence rates \cite{VDG00,Gine.op.2016} over 
classical smoothness classes---deterioriate exponentially in terms of the dimension $d$. 
This renders computations practically infeasible for large $d$. This phenomenon is referred to as the \emph{curse of dimensionality}, see also Section \ref{sec:finitedimensional} ahead.

The framework for operator learning 
considered in this paper can be viewed as a direct extension of \eqref{eq:finiteregression} to the infinite-dimensional case. Given Hilbert spaces $\X$ and $\Y$, and a mapping $G_0:\X\to\Y$, the goal is to reconstruct $G_0$ from ``training data'' $(x_i,y_i)\in\X\times\Y$ with 
\[y_i=G_0(x_i)+\eps_i,\]
where the regression class $\pG$ is a suitable set of measurable mappings between $\X$ and $\Y$ and $\eps_i$ are centered noise variables,
see Section \ref{sec:framework} for details. This ``supervised learning'' setting underlies popular 
methods such as the PCA-Net \cite{HESTHAVEN201855,BhattModRedNNPDE}.

We also mention the framework of ``physics-informed learning'' which is common in operator learning (relevant e.g.~for the DeepONet \cite{Lu2021}), but which is not considered in the present manuscript.
Here, information on the ground truth $G_0$ is not 
known in the form of input-output pairs, but instead is implicitly
described via
\begin{equation*}
  \mathcal{N}(G_0(x),y)=0,
\end{equation*}
where $\mathcal{N}:\X\times\Y\to \cZ$ for a third vector space $\cZ$.
Typically, $\mathcal{N}(x,\cdot)$ encodes a family of differential operators parametrized by $x$ which represent the underlying physical model. In this case, the loss to minimize 
is a residual of the form $\sum_{i=1}^n \|\mathcal{N}(x_i,G(x_i))\|_{\cZ}^2$, thus leading to an ``unsupervised learning problem''. The two cases, supervised and unsupervised learning, can also be combined. Various different (neural network based) architectures (i.e.\ regression classes $\pG$) have been proposed in recent years for the purpose of supervised or unsupervised operator learning. 

\subsection{Outline and Contributions}
In this paper we provide statistical convergence results for operator learning which do not suffer from the curse of dimensionality, and which can be applied to prototypical problems in the PDE literature. We first develop our theory in an abstract setting for ERMs over classes of 
mappings between separable Hilbert spaces, and later apply our theory to concrete examples. In doing so, we build upon and synthesize influential proof techniques from nonparametric statistics, in particular M-estimation \cite{VDG01, NVW18}, approximation theory for parametric PDEs \cite{CDS11,CohenDeVore,Herrmann.2024}, and empirical process theory \cite{Talagrand.2005, Dirksen.2015}.

To illustrate the scope of our contributions,
we
start by stating a convergence result for the elliptic ``Darcy flow'' problem on the $d$-dimensional torus.
This is a standard example in PDE driven forward  and inverse problems, e.g.\ 
\cite{MR2646806,MR3298364,MR2835612,MR2903278,S10,N23,NVW18}.
We aim for an informal exposition here, with full details given in
Section \ref{sec:applications}: Denote by $\bbT^d$ the $d$-dimensional torus, fix a smooth source function 
$f:\bbT^d\to (0,\infty)$ and let $a_{\min}>0$. 
For a sufficiently smooth and uniformly positive conductivity $a: \bbT^d\to \mathbb R$, denote by $G_0(a)$ the unique solution of the elliptic PDE
\begin{equation}\label{darcy-intro}
    -\nabla\cdot (a\nabla u) = f~~\text{on}~\bbT^d ~~~ \text{and}~~~ \int_{\bbT^d}u(x)dx =0.
\end{equation}
Now let $\gamma$ be some probability distribution on $L^2(\bbT^d)$ 
such that
\begin{equation}
    \text{supp}(\gamma)\subseteq \big\{ a \in H^\mathfrak{s}(\bbT^d): \inf_{x\in \bbT^d} a(x)\ge a_{\min},~\|a\|_{H^\mathfrak{s}(\bbT^d)}\le R\big\},\label{supp-gamma}
\end{equation}
for some $R>0$, $\mathfrak{s}>2d+1$. Suppose we observe noisy input-output pairs
$(a_i,y_i)_{i=1}^n$ given by $y_i=G_0(a_i)+\eps_i$, where the $\eps_i$ are  
independent $L^2(\bbT^d)$-Gaussian white noise processes 
(Section \ref{sec:framework}).
The operator $G_0$ can then be learned from this data as stated in the next theorem (Section \ref{sec:DiffEqTorus});
it regards empirical risk minimizers $\hat G_n$ over the so-called FrameNet $\pG_{\rm FN}$,
which corresponds to a neural network based class of measurable mappings from $\X\to \Y$
(Section \ref{sec:networkarchitecture}).
Formally, $\hat G_n$ is defined as 
a minimizer of the least squares 
objective
\begin{equation}\label{eq:hatGnIntro}
  \hat G_n \in \argmin_{G\in \pG_{\rm FN}} \frac 1n\sum_{i=1}^n \|y_i-G(a_i)\|^2_{L^2(\bbT^d)},
\end{equation}
although a suitable modification
is required to make this mathematically rigorous
(Section \ref{sec:setting}).
\begin{theorem}[Informal]\label{thm:intro}
  Consider the operator $G_0$ from the Darcy problem on the $d$-dimensional torus $\bbT^d$ ($d\ge 2$), and suppose that $\gamma$ satisfies \eqref{supp-gamma} for some $\mathfrak{s} > 3d/2+1 $ and $a_{min}>0$. Fix $\tau>0$ (arbitrarily small).

  Then there exists a constant $C$ such that for each $n\in\N$
    there exists a FrameNet class $\pG_{\rm FN}(n)$
    and any empirical risk minimizer $\hat G_n$ 
    in \eqref{eq:hatGnIntro} satisfies\footnote{Here and in the following $\bbE_{G_0}$ denotes the expectation w.r.t.\ the random data $(x_i,y_i)_i$ generated by the ground truth $G_0$. Similarly, we write $\bbP_{G_0}$ for corresponding probabilities.}
    \begin{equation}\label{intro-rate}
      \mathbb E_{G_0} \left[ \int \|\hat G_n(a)-G_0(a)\|_{L^2(\bbT^d)}^2 d\gamma(a) \right]\le Cn^{-\frac{2\mathfrak{s}+2-3d}{2\mathfrak{s}+2-d}+\tau}.
    \end{equation}
  \end{theorem}

  The most significant feature of the above statement is that the convergence rate in \eqref{intro-rate} is algebraic in $n$, and thus circumvents the curse of dimensionality. The classes $\pG_{\rm FN}$, whose existence is postulated by the theorem, can be precisely 
characterized in terms of the sparsity, depth, width and other network class parameters, which are chosen in terms of the statistical sample size $n$. We also note that the regularity assumption $\mathfrak{s}>3d/2+1$ was made here for convenience and can be weakened to $\mathfrak{s} >3d/2$, see Theorem \ref{thm:torus} and Remark \ref{rmk:torus} below. 

To achieve Theorem \ref{thm:intro} and several other related results, we build our theory in multiple steps. In Section \ref{sec:framework}, a general regression framework for mappings between Hilbert spaces is considered.
Our first main result, Theorem \ref{th:empiricalerror}, gives a non-asyptotic concentration upper bound on the empirical risk between $\hat G_n$ and $G_0$,
with respect to the design points $x_i$. The upper bound is quantified in terms of the metric entropy of the ``regression class'' $\pG$ and the best approximation of $G_0$ from $\pG$. Theorem \ref{th:samplecomplexitywhitnoise} strengthens this statement to $L^2(\gamma)$-loss, for the case of random design points $x_i\sim \gamma$. These results provide an operator learning analogue to classical convergence rates in nonparametric regression. The proofs rely on probabilistic generic chaining techniques \cite{T14,Dirksen.2015} and ``slicing'' arguments as introduced in \cite{VDG01}, which we generalize to the current setting. We also note our proofs contrast existing nonparametric statistical analyses of neural networks \cite{SH20} for real-valued regression, where generic chaining techniques were not required for obtaining optimal rates (up to $\log$-factors).

In the second part of this work, we apply our statistical results to
the specific deep operator network class $\pG_{\trm{FN}}$ termed ``FrameNet'' 
and introduced in
\cite{Herrmann.2024}.
Together with their underlying decoder-encoder structure and feedforward neural network structure, {FrameNet} classes are defined in Section \ref{sec:networkarchitecture}.
These classes are known to satisfy good approximation properties for holomorphic operators, a property which is fulfilled for the Darcy problem \eqref{darcy-intro} and more broadly a wide range of PDE based problems
\cite{CDS10,CDS11,CohenDeVore,JSZ17,harbrecht2016analysis,henriquez2021shape,Hiptmair2018,spence2023wavenumber,CSZ18_2319}.
In Section \ref{sec:networkarchitecture}, we identify such operator holomorphy as a key regularity property which allows to derive ``dimension-free'' statistical convergence rates.
By extending approximation theoretic results from \cite{Herrmann.2024}, as well as establishing metric entropy bounds for $\pG_{\rm FN}$ based on \cite{SH20}, we obtain algebraic convergence results for ERMs over {FrameNet} classes,  
for reconstructing holomorphic operators $G_0$.
Specifically, Theorem \ref{th:RgnG0boundWN} bounds the $L^2$-risk 
$\IE[\lVert\hatgn-G_0\rVert_{L^2(\gamma)}^2]\lesssim n^{-\kappa/(\kappa+1)}$, where $\kappa>0$ denotes the \emph{approximation} rate established in \cite{Herrmann.2024}.
We treat the case of ReLU and RePU \cite{Li.2020} activation functions for both sparse and fully-connected architectures.

In Section \ref{sec:applications}, we illustrate the usefulness of our general theory in two concrete settings. First, we show how our theory recovers well-known minimax-optimal convergence rates for real-valued regression (i.e., $\Y=\R$) on $d$-dimensional domains. This proves that our abstract results from Section \ref{sec:framework} cannot be improved \textit{in general}, although matching lower bounds are yet unknown in the infinite-dimensional setting. Thereafter, Section \ref{sec:DiffEqTorus} demonstrates how our theory can be used to yield the first algebraic convergence rates for a non-linear operator arising from PDEs -- see in particular Theorem \ref{thm:torus} and Remark \ref{rmk:torus}, which underlie Theorem \ref{thm:intro}.

\subsection{Existing Results}
The approximation of mappings between infinite-dimensional spaces has been studied extensively in the context of Uncertainty Quantification, where $G_0$ corresponds to the solution operator of a parameter dependent PDE. Various methodologies have been proposed and analyzed for this task, including for example compressed sensing \cite{DOOSTAN20113015,RS16_1344}, sparse-grid interpolation \cite{MR3800593,NobilerandomPDE}, least-squares \cite{leastsquaresI,leastsquaresII}, and reduced basis methods \cite{MR3379913,MR3408061}. Recently, neural network approaches have become increasingly popular for this task as they provide a highly expressive and fast to evaluate parametrization of high-dimensional functions. These attributes make them particularly useful for learning surrogates in scientific applications, e.g.\ \cite{HESTHAVEN201855,Lu2021,li2020fourier,BhattModRedNNPDE,anandkumar2019neural,OLEARYROSEBERRY2022114199,OLEARYROSEBERRY2024112555,becker2024learning,cicci2022deep,dal2020data,kropfl2022operator}.

\paragraph*{Approximation Theory for Neural Operators}
First theoretical results on operator learning focused on the approximation error, establishing the existence of neural network architectures capable of approximating $G_0$ up to a certain accuracy, with the error decreasing algebraic in terms of the number of learnable network parameters. For example, \cite{Schwab2019Deep,MR4376560,Schwab2023Deep} showed that neural networks have sufficient expressivity to efficiently approximate certain (holomorphic) mappings $G_0$. Such results are based on the observation that the smoothness of $G_0$ implies the image of this operator to have moderate $n$-widths,
i.e.\ to be well approximated in moderate-dimensional linear subspaces. See for example \cite{Dung2023Gaussian,CDS10,CDS11,CohenDeVore,Hoang2014,MR3601011}.
Specifically for DeepONets \cite{Lu2021}, such a result was obtained in \cite{Lanthaler.2022}.
\paragraph*{Statistical Theory for Neural Operators} The analysis of sample complexity has received less attention so far. In \cite{Lanthaler2023PCA}, the authors analysed in particular the error of PCA encoders and decoders used for PCA-Net, but did not analyse the statistical error for the full operator.

The paper \cite{deHoopLinOp} provides such a result for the estimation of linear and diagonalizable mappings from noisy data; for lower bounds see, e.g., \cite{chagny2022adaptive}. For other work on ``functional regression'', see, e.g., \cite{greven2017general,morris2006wavelet}. An analysis for nonparametric regression of nonlinear mappings from noisy data in infinite dimensions was provided in \cite{Liu.20220101}. 
There, the authors considered Lipschitz continuous mappings $G_0$, and proved consistency in the large data limit. Additionally they give convergence results, which in general suffer from the curse of dimension however. This is due to their very general assumption on the smoothness of $G_0$: Mhaskar and Hahm \cite{Mhaskar.1997} showed very early, that the nonlinear $n$-width of Lipschitz operators in $L^2$ decays only logarithmically, i.e.\ the number of (exact) data points needed for the reconstruction of the functional is exponential in the desired accuracy. Recently, \cite{Kovachki.2024} generalized these results and showed a generic curse of dimensionality for the reconstruction of Lipschitz operators and $C^k$-operators from exact data. Moreover, the authors show that under the existence of some intrinsic low-dimensionality allowing for fast approximation, also the dependence on the data complexity improves. Concerning the case of noisy and holomorphic operators, we also refer to the recent works \cite{ADCOCK2025106761,Adcock.2024b} who consider a setup similar to ours, but contrary to us also treat the more general case of Banach space valued functions. The authors derive upper bounds and concentration inequalities for the $L^2$-error, and lower bounds for the approximation error, in terms of a neural network based architecture. Key differences to our work include in particular that \cite{ADCOCK2025106761,Adcock.2024b} consider network architectures that are linear in the trainable parameters, in the noisy case their analysis does in general not imply convergence in the large data limit $n\to\infty$, and they do not provide convergence rates for concrete PDE models.

\paragraph*{M-Estimation in Nonparametric Regression}
Convergence theory of M-estimators and (penalized) empirical risk minimizers was investigated around the 2000s in foundational works by van de Geer \cite{VDG00, VDG01}, Birge, Massart and co-authors \cite{BM93, BBM99}. These works build on concentration inequalities for empirical processes using ``chaining'' techniques which date back to seminal contributions of Talagrand and others, see, e.g.,~\cite{T14,GN16} and references therein. These techniques are known to produce minimax-optimal rates $n^{-{2\mathfrak{s}/(2\mathfrak{s}+d)}}$ for ERMs over $\mathfrak{s}$-smooth Sobolev, H\"older and more generally, Besov smoothness classes of real-valued functions 
on bounded $d$-dimensional Euclidean domains. The analysis of neural-network based ERMs was initiated by the work \cite{SH20}, which considered regression over (compositional) H\"older classes on finite-dimensional domains, and was followed by several other works such as 
\cite{S19}. We also mention \cite{NVW18, NW20, AW21} which analyse ERMs in non-linear elliptic PDE-based inverse problems such as the ``Darcy'' flow problem studied here. The present setting falls outside the scope of such classical theory for real-valued functions. However, the derivation of our concentration inequalities for ERMs does build upon the same probabilistic empirical process machinery laid out above \cite{T14, Dirksen.2015}.
\subsection{Notation}
We write $\IN=\{1,2\dots\}$ and $\IN_0=\{0,1,2,\dots\}$. 
We write $a_n\lesssim b_n, a_n\gtrsim b_n$ for real sequences $(a_n)_{n\in\IN}$, $(b_n)_{n\in\IN}$ if $a_n$ is respectively upper or lower bounded by a positive multiplicative constant which does not depend on $n$ (but may well depend on other ambient parameters which we make explicit whenever confusion may arise). By $a_n\simeq b_n$, we mean that both $a_n\lesssim b_n$ and $a_n\gtrsim b_n$.

For a pseudometric space $(T,d)$ and any $\delta>0$, let $N(T,d,\delta)$ be the $\delta$-covering number of $T$, i.e.~the minimal number of open $\delta$-balls in $d$ needed to cover $T$. We denote the metric entropy of $T$ by
\begin{equation}\label{eq:entropy}
  H(T,d,\delta)= \log N(T,d,\delta).
\end{equation}
Given a Borel probability measure $\gamma$ on $\X$ and a subset $D\subseteq\X$, we define the norms
\begin{align*}
    &\|G\|_{L^2(\X,\gamma;\Y)}^2\coloneqq\int_{\X} \|G(x)\|_{\Y}^2\dd\gamma(x),\\
    &\|G\|_{\infty,D}\coloneq \sup_{x\in D}\lVert G(x)\rVert_{\Y}
\end{align*}
and also write $\|\cdot\|_{L^2(\gamma)}$ and $\|\cdot\|_{\infty}$ if the underlying spaces are clear from context. The space of real-valued, square summable sequences indexed over $\IN$ is denoted by $\ell^2(\IN)$. The complexification of a real Hilbert space $H$ is denoted by $H_{\IC}$, see \cite{padraig,munoz99}.

\section{Regression in Hilbert Spaces}\label{sec:framework}
\subsection{Problem Formulation}\label{sec:setting}
Throughout, let $\X$ and $\Y$ denote two separable (real) Hilbert spaces with respective inner products $\langle \cdot,\cdot\rangle_{\X}$, $\langle \cdot,\cdot\rangle_{\Y}$ and suppose
\[\G_0:\mathcal \X\to\Y\] is some non-linear (Borel measurable) operator which we aim to reconstruct. The observed data are assumed to be noisy ``input-output pairs''
$(x_i,y_i)_{i=1}^{n}\in (\X\times\Y)^n$ given by
 
\begin{subequations}\label{eq:dataxy}
 \noeqref{eq:datax,eq:datay}
\begin{equation}\label{eq:datax}
    x_i\overset{\rm iid}{\sim}\gamma\qquad i=1,\dots,n,
\end{equation}
and
\begin{equation}\label{eq:datay}
  y_i=\G_0(x_i)+\sigma\eps_i\qquad i=1,\dots,n,
\end{equation}
\end{subequations}
where $\sigma>0$ denotes a scalar ``noise level'',
$\eps_i$ are independent random noise variables and $\gamma$ is a probability distribution on $\X$. The
$x_i\in \mathcal X$ are also referred to as the ``design points'',
and we write $\px = (x_1,..., x_n)\in\X^n$. We will both derive results which are \emph{conditional} on the design $\px$, as well as results for \textit{random design}.
To avoid confusions, we will use the notations $P_{G_0}^{\px}$, $\mathbb E_{G_0}^{\px}$ to denote probabilities and expectations under the distribution (\ref{eq:dataxy}) with fixed design $\px$, and we use $\mathbb P_{G_0}$, $\mathbb E_{G_0}$ to denote probabilities and expectations with random design $x_i\sim \gamma$.
\begin{remark}\label{rmk:subsetV}
In practice, we will often deal with scenarios in which $\G_0$ is only defined on some measurable subset $\mathcal V\subset\X$, see e.g.~the solution operator in the Darcy flow example in Section \ref{sec:torusG}. In this case, our results can be applied to any measurable extension of $G_0$ on $\X$.
\end{remark}

\paragraph*{\hypertarget{WN}{White Noise Model}}
In this article we consider two assumptions on the noise, the first being that the $(\eps_i)_{i=1}^n$ in \eqref{eq:dataxy} are independent copies of a $\Y$-white noise process. Recall that for any given separable Hilbert space $\Y$, the $\Y$-Gaussian white noise process is defined as the mean-zero Gaussian process $\mathbb W_{\Y}=(\mathbb W_{\Y}(y):y\in\Y)$ 
indexed by $\Y$ with ``iso-normal'' covariance structure
\[ \mathbb W_{\Y}(y)\sim \mathcal N(0,\|y\|_{\Y}^2),~~~ \text{Cov}(\mathbb W_{\Y}(y),\mathbb W_{\Y}(y'))= \langle y,y'\rangle_{\mathcal Y},~~~ \text{for all}~y,y'\in\Y.\]
It is well-known that $\mathbb{W}_\Y$ does not take values in $\Y$ unless $\dim(\Y)<\infty$, but is interpreted as a stochastic process indexed by $\Y$, see \cite[p.19]{Gine.op.2016} for details. Nevertheless, we slightly abuse notation and use the common notation $\langle \mathbb W_{\Y}, y \rangle_{\mathcal Y}:=\mathbb W_{\Y}(y)$.

Under this assumption, conditionally on $x_i$ we interpret each observation $y_i$ in (\ref{eq:dataxy}) as a realisation of a Gaussian process $(y_i(f):f\in\mathcal Y)$ with
\[ \mathbb E[y_i(f)] = \langle G_0(x_i), f\rangle_{\mathcal Y},~~~ \text{Cov}(y_i(f),y_i(f'))= \langle v,v'\rangle_{\mathcal Y},\]
and we shall again use the notation $\langle y_i, f\rangle_{\Y}$ to denote $y_i(f)$ (see also \cite{T09, GN16, NVW18} where this common viewpoint is explained in detail).

\begin{example}
Let $\mathcal O\subseteq \R^d$ be a bounded, smooth domain. Then, for $\Y=L^2(\O)$, one can show that draws of an $L^2(\mathcal O)$-white noise process a.s.~take values in negative Sobolev spaces $H^{-\kappa}$ for $\kappa>d/2$, see, e.g., \cite{Nickl.2020, Castillo.2013}.
\end{example}

\paragraph*{\hypertarget{SGN}{Sub-Gaussian Noise Model}}
The second setting we consider is that of sub-Gaussian noise. We say that a random vector $X$ taking values in $\Y$ is sub-Gaussian with parameter $\eta>0$ if $\IE[X]=0$ and 
\[\IP(\lVert X\rVert_{\Y} \geq t)\leq2\exp\left(-\frac{t^2}{2\eta^2}\right),~~\text{for all} ~t\geq 0.
\]
In the sub-Gaussian noise model, we assume that $(\eps_i)_{i=1}^n$ in \eqref{eq:dataxy} are independent sub-Gaussian variables in $\mathcal Y$ with parameter $\eta =1$.

\subsubsection{Empirical Risk Minimization}
Let $\pG$ be a class of (measurable) operators $\pG \ni G:\X\to\Y$. We would like to study classical empirical risk minimizers of least-squares type over $\pG$. Specifically, given regression data $(x_i,y_i)_{i=1}^n$, consider the empirical risk
\begin{equation}
\label{eq:empiricalrisminimizersubGaussian}
    \tilde I_n(G) := \frac{1}{n}\sum_{i=1}^n\bigl\lVert y_i- G(x_i)\bigr\rVert_{\Y}^2, ~~~ \tilde I_n: \pG \to [0,\infty].
\end{equation}
However, this functional takes finite values almost surely only in the \hyperlink{SGN}{sub-Gaussian noise} model. In the \hyperlink{WN}{white noise model}, since $y_i\notin \Y$, it holds $\tilde{I}_n(G)=\infty$ almost surely -- we thus consider a modified definition of least-squares type estimators which is common in the literature on regression with white noise \cite{NVW18, GN16}. Instead of (\ref{eq:empiricalrisminimizersubGaussian}), we consider
\begin{equation}
    \label{eq:empiricalrisminimizer}
    I_n(G)=\frac{1}{n}\sum_{i=1}^n-2\langle G(x_i),y_i\rangle_{\Y}+\bigl\lVert G(x_i)\bigr\rVert_{\Y}^2, ~~~ 
    I_n:\pG\to \R,
\end{equation}
which takes finite values a.s.~also in the \hyperlink{WN}{white noise model}. Note that the latter objective function can be obtained from (\ref{eq:empiricalrisminimizersubGaussian}) by formally subtracting  the term $n^{-1}\sum_{i=1}^n\|y_i\|_{\mathcal Y}^2$ which exhibits no dependency on $G$. Therefore in the \hyperlink{SGN}{sub-Gaussian noise model} the minimization of $\tilde I_n$ and $I_n$ are equivalent which is why we consider (\ref{eq:empiricalrisminimizer}) in the following. We will denote minimizers of $I_n(G)$ by $\hat G_n$.

Our assumptions on the class $\pG$ in the ensuing theorems will ensure that a \textit{measurable choice} of minimizers $\hat G_n$ of $I_n$ exists, see Theorem \ref{th:empiricalerror} \ref{item:minimizer}. However, the ERM $\hatgn$ will in general not be unique, since we do not impose convexity on $\pG$. The reason is that our main application, the NN-based FrameNet class $\pG_{\rm FN}$, is non-convex.

\begin{remark}[Connection to maximum likelihood]
    In the \hyperlink{WN}{white noise model}, it follows from the Cameron-Martin theorem (see, e.g., Theorem 2.6.13 in \cite{GN16}) that $-nI_n(G)/(2\sigma^2)$ constitutes the negative log-likelihood of the (dominated) statistical model arising from (\ref{eq:dataxy}) with white noise. In this case $\hat G_n$ can also be interpreted as a (nonparametric) maximum likelihood estimator over the class $\pG$.
\end{remark}

\begin{remark}
Consider nonparametric regression of an unknown function $f:\mathcal O\to\R$ for some bounded, smooth domain $\mathcal O$. Here, it is well-known that the observation white noise error model, where data is given by $Y=f+\sigma \mathbb W$ (with $\mathbb W$ a $L^2(\mathcal O)$-white noise process) is asymptotically equivalent in a Le Cam-sense to an observation model with $m$ ``equally spaced'' (random or deterministic) observation points throughout $\mathcal O$,
\[ Y_i= f(z_i)+\eta_i,~~~~i=1,...,m, \]
with i.i.d.~$N(0,1)$ errors, where the equivalence holds for $\sigma\asymp1/\sqrt{m}$, see \cite{R08}. Therefore, our observation model (\ref{eq:dataxy}) may be viewed as a simplified proxy.
\end{remark}

\subsection{Main Results}\label{sec:secmainresult}
Let $\pG$ be a class of operators mapping from $\X$ to $\Y$.
For any fixed $\px=(x_1,...,x_n)\in \mathcal X^n$ and (Borel) measurable map $G:\X\to\Y$, we denote the empirical seminorm induced by $\px$ with
\begin{equation}\label{eq:empiricalnorm}
    \|G\|_{n}^2=\frac 1n\sum_{i=1}^n \lVert G(x_i)\rVert_{\Y}^2. 
\end{equation}
For any element $G^*\in\pG$ and $\delta>0$, define the localized classes\[\pG_n^*(\delta)=\bigl\{G\in\pG:\,\lVert G-G^*\rVert_n\leq\delta\bigr\},\] and denote its metric entropy integral by
\begin{equation}\label{eq:Jintegral} 
  J(\delta) =  J(\pG_n^*(\delta),\|\cdot\|_n):= \int_0^\delta H^{\frac{1}{2}}(\pG_n^*(\delta),\|\cdot\|_n,\rho) d\rho.
\end{equation}

The following result provides a general convergence theorem for empirical risk minimizers with high probability, which relates the empirical risk of ERMs over some operator class $\pG$ to the metric entropy of $\pG$. It can be viewed as a generalisation of classical convergence results for sieved M-estimators \cite{vanGeer.2009} to Hilbert space valued functions. The proof can be found in Appendix \ref{app:proofempiricaltot}.

\begin{theorem}\label{th:empiricalerror}
For some measurable $G_0:\mathcal X\to\mathcal Y$, let the data $(x_i,y_i)_{i=1}^n$ arise from (\ref{eq:dataxy}) either with \hyperlink{WN}{white noise} or with \hyperlink{SGN}{sub-Gaussian noise}.
Let $\pG$ be a class of measurable maps from $\X\to\Y$, let $G^*\in \pG$, and let $\px=(x_1,...,x_n)\in \X^n$ be such that the following holds.
  \begin{enumerate}[label=(\alph*)]
  \item\label{item:Gassump1} There exists a constant $C>0$ s.t.~$\pG$ is compact with respect to some norm $\|\cdot\|$
    satisfying $\|\cdot\|_{n}\le C \|\cdot\|$.
  \item\label{item:Gassump2}
There exists $\Psi_n:(0,\infty)\to [0,\infty)$ s.t.\
    $\Psi_n(\delta)\ge J(\pG_n^*(\delta),\|\cdot\|_{n})$ for all $\delta>0$ and
  \begin{equation*}
    \delta \mapsto \frac{\Psi_n(\delta)}{\delta^2} \qquad\text{is non-increasing for }\delta\in (0,\infty).
  \end{equation*}
\end{enumerate}
 Then the following holds.
  \begin{enumerate}
    \item\label{item:minimizer}
Minimizers $\hat G_n$ of the empirical risk 
(\ref{eq:empiricalrisminimizer}) exist, and there is a measurable selection (with respect to the data $(x_i,y_i)_{i=1}^n$) of such a minimizer.
 
\item\label{item:conditionaltail} Fix any measurable selection $\hat G_n$ from part (i). Then there exists a 
  universal constant $C_{\rm Ch}>0$ (see Lemma \ref{lem:chaining1}) such that for any $G^*\in\pG$ as above, any positive sequence $(\delta_n)_{n\in\IN}$ satisfying
\begin{align}\label{eq:deltan}
    \sqrt{n}\delta_n^2\geq 32 C_{ch} \sigma\Psi_n(\delta_n),
\end{align}
and for any 
\begin{equation}\label{R-def}
    R\geq\max\Big\{\delta_n,\frac{4C_{ch}\sigma}{\sqrt{n}},\sqrt{2}\lVert G^*-G_0\rVert_n\Big\},
\end{equation}it holds that
\begin{align}
\label{eq:empiricalconcentration}
\IP_{G_0}^{\px}\bigl(\lVert\hat{G}_n-\G_0\rVert_n\geq R\bigr)\leq 2\exp\biggl(-\frac{nR^2}{16C_{ch}^2\sigma^2}\biggr).
\end{align}
\end{enumerate}
\end{theorem}
The lower bound for $R$ in \ref{R-def}, which determines the convergence rate of $\|\hat G_n-G_0\|_n$, is typically optimized by balancing the ``stochastic term'' $\delta_n$ and the empirical approximation error $\|G^*-G_0\|_n$, see, e.g., Theorem \ref{th:RgnG0boundWN} below. 
Note that Theorem \ref{th:empiricalerror} gives a convergence rate with respect to the \textit{empirical norm} $\|\cdot\|_n$. Therefore, when the design points $\px$ are random, the norm itself is also random, and the assumptions \ref{item:Gassump1} and \ref{item:Gassump2} in Theorem \ref{th:empiricalerror} have to be understood conditional on possible realizations of $\px$. Theorem \ref{th:samplecomplexitywhitnoise} below will give a corresponding concentration inequality on the $\|\cdot\|_{L^2(\gamma)}$-error under the assumption of i.i.d.~random design $x_i\sim \gamma$. In this setting, a sufficient condition for \ref{item:Gassump2} to be satisfied almost surely is to take $\Psi_n$ as an upper bound for the entropy integral with uniform entropy $H(\pG,\|\cdot\|_{\infty,\supp(\gamma)},\delta)$. 

The compactness of $\pG$ in part \ref{item:Gassump1} is needed for the existence of the ERM $\hatgn$ in \eqref{eq:empiricalrisminimizer}. The existence result for $\hatgn$ (see e.g.~\cite[Proposition 5]{Nickl.2007}) requires compact \emph{metric} spaces, which
is why we assume copmactness w.r.t.\ a \emph{norm}
$\|\cdot\|$ stronger than the empirical \emph{seminorm} $\|\cdot\|_n$.

 In particular, compactness implies separability of $\pG$ with respect to~$\|\cdot\|_{n}$, which in turn is needed to guarantee the measurability of certain suprema of empirical processes ranging over $\pG$. See the proof of Theorem \ref{th:empiricalerror} below, in particular \eqref{eq:proofempirical3} and \eqref{eq:proofempirical4}. The technical growth restriction in \ref{item:Gassump2} on the function $\Psi_n(\delta)$ (i.e., our upper bound for the entropy integral) is required for the ``peeling device'' in \eqref{eq:proofempirical3}. In particular, this assumption is satisfied in case that $H(\pG,\|\cdot\|_{\infty},\delta)\lesssim \delta^{-\mathfrak{s}}$ for any $0<\alpha<2$, see Corollary \ref{cor:MSE} below for details.

Theorem \ref{th:empiricalerror} provides a concentration inequality for the empirical norm $\|\hat G_n-G_0\|_{n}$. Under the assumption of randomly chosen design points $x_i\sim \gamma$, $x_i\in\mathcal X$, this statement can be extended to a convergence result for the $L^2(\gamma)$-norm. To this end, we also need slightly stronger technical assumptions on the class $\pG$ with respect to the $\|\cdot\|_{\infty,\supp(\gamma)}$-norm.

\begin{assumption}\label{assump:g0assump}
 For some probability measure $\gamma$ on $\X$, assume $\px=(x_1,\dots, x_n)$ arises from i.i.d.~draws $x_i\sim \gamma$. Let $\pG$ be a class of measurable maps $\X\to\Y$, $G^*\in \pG$ and $\|G_0\|_{\infty,\supp(\gamma)}<\infty$. Suppose
  \begin{enumerate}[label=(\alph*)]
  \item\label{item:Gassump11} $\pG$ is compact with respect to $\|\cdot\|_{{\infty,\supp(\gamma)}}$,
  \item\label{item:Gassump21}
There exists a (deterministic) upper bound $\Psi_n:(0,\infty)\to [0,\infty)$ such that for a.e.~$\px\sim \gamma^n$, it holds 
    $\Psi_n(\delta)\ge J(\pG_n^*(\delta),\|\cdot\|_{n})$ for all $\delta>0$ and
  \begin{equation}\label{eq:independentofx}
    \delta \mapsto \frac{\Psi_n(\delta)}{\delta^2} \qquad\text{is non-increasing for }\delta\in (0,\infty).
  \end{equation}
\end{enumerate}
\end{assumption}
Note that Assumption \ref{assump:g0assump} is strictly stronger than assumptions \ref{item:Gassump1} and \ref{item:Gassump2} in Theorem \ref{th:empiricalerror}, since $\|\cdot\|_{{\infty,\supp(\gamma)}}$ is stronger than $\|\cdot\|_n$ and 
  $\Psi_n$
  in \eqref{eq:independentofx} does not depend on $\px$.
  In particular, Assumption \ref{assump:g0assump} implies the existence of a measurable ERM $\hatgn$ by Theorem \ref{th:empiricalerror} \ref{item:concrete1}.

The uniform $\|\cdot\|_{\infty,\supp(\gamma)}$ assumptions are used to control the concentration of the empirical norm $\|\cdot\|_n$ around the ``population norm'' $\|\cdot\|_{{\infty,\supp(\gamma)}}$, see Lemma \ref{lem:variance} and also Lemma \ref{lem:sub-Gaussianepsilon} below. Let us write $\pF=\set{G-\G_0}{G\in\pG}$. As an immediate consequence of Assumption \ref{assump:g0assump}, there exists some $F_\infty<\infty$ such that
\begin{equation}\label{eq:F-bound}
\sup_{F\in\pF}\|F\|_{{\infty,\supp(\gamma)}}=\sup_{G\in\pG}\,\lVert G-\G_0\rVert_{{\infty,\supp(\gamma)}}\leq F_{\infty}.
\end{equation} 
We can now state our main concentration inequality for the convergence of $\lVert \hat{G}_n-\G_0\rVert_{L^2(\gamma)}$, which in particular provides a bound for the mean squared error
 $\IE_{\G_0}[\lVert \hat{G}_n-\G_0\rVert_{L^2(\gamma)}^2]$ as well. The proof of Theorem \ref{th:samplecomplexitywhitnoise} can be found in Appendix \ref{app:proofsamplecomplexity}.

\begin{theorem}[$L^2(\gamma)$-Concentration under Random Design]
\label{th:samplecomplexitywhitnoise}
    Consider the nonparametric regression model \eqref{eq:dataxy} either with \hyperlink{WN}{white noise}  or with \hyperlink{SGN}{sub-Gaussian noise} and any measurable empirical risk minimizer $\hat G_n$ from \eqref{eq:empiricalrisminimizer}. Suppose that $G_0$, $\pG$, $G^*$, $\gamma$ and $\Psi_n(\cdot)$ are such that Assumption \ref{assump:g0assump} holds. Then there exists some universal constant $C>0$ such that for any positive sequences $(\delta_n)_{n\in\IN}$ and $(\tildedeltan)_{n\in\IN}$ with 
      \begin{equation}\label{eq:deltandeltanhat}
        \sqrt{n}\delta_n^2\geq C\sigma\Psi_n(\delta_n)\quad
        \text{and}\quad
        n\tildedeltan^2\geq C\Finf^2 H(\pG,\norm[{\infty,\supp(\gamma)}]{\cdot},\tildedeltan),
      \end{equation}
      all $G^*\in\pG$ as above and all
\begin{align}\label{eq:Rinequality}
          R\geq C\max\left\{\delta_n,\tildedeltan,\lVert G^*-G_0\rVert_{{\infty,\supp(\gamma)}},\frac{\sigma+F_\infty}{\sqrt{n}}\right\}
      \end{align}
      we have, 
      \begin{align}\label{eq:L2concentration}
          \IP_{G_0}\left(\lVert\hatgn-G_0\rVert_\ltg\geq R\right)\leq 2\exp\left(-\frac{nR^2}{C^2(\sigma^2+F_\infty^2)}\right).
      \end{align}
\end{theorem}
To keep the presentation simple, we have left the numerical constants in the preceding theorem implicit. However, they can be made explicit, see the proof for details.
The following bound on the mean squared error is obtained upon integration of the concentration inequalites from Theorem \ref{th:empiricalerror} and Lemma \ref{lem:variance}. Note that directly integrating the $L^2(\gamma)$-concentration, cf.~\eqref{eq:L2concentration}, gives an approximation term in the uniform norm $\|\cdot\|_{\infty,\supp(\gamma)}$ (following \eqref{eq:Rinequality}), which has weaker convergence properties in general, see Theorem \ref{th:approximationNN}.
For the proof of Corollary \ref{cor:MSE}, see Appendix \ref{app:samplecomplexityexpectation}.
\begin{corollary}[$L^2(\gamma)$-Mean Squared Error]\label{cor:MSE}
    Consider the setting of Theorem \ref{th:samplecomplexitywhitnoise} and assume in addition that Assumption \ref{assump:g0assump} is fulfilled for all $G^*\in\pG$ (with the same $\Psi_n$). Then, for some universal constant $C>0$ and all $n\in\IN$,
    \begin{align}
    \label{eq:WNconvergence}
           \IE_{\G_0}\,\Big[\lVert \hat{G}_n-\G_0\rVert_{L^2(\gamma)}^2\Big]
      \leq C\Big( \delta_n^2+\tilde{\delta}_n^2+\frac{\sigma^2+\Finf^2}{n}\Big)+8\inf_{\gstar\in \pG}\lVert \gstar-\G_0\rVert_{L^2(\gamma)}^2.
    \end{align}
\end{corollary}

To demonstrate the typical use-cases of our abstract results, we summarize in the following corollary the rates which can be obtained under algebraic approximation properties of $\pG$ and two concrete scalings of the metric entropy of $\pG$. These correspond to the typical entropy bounds satisfied by (i) some fixed $n$-independent, infinite-dimensional regression class, and (ii) an $N$-dimensional approximation class, where $N$ is chosen in terms of $n$. In the following corollary, $\lesssim$ refers to an inequality involving a constant independent of $N$ and $\delta$. For a proof of Corollary \ref{cor:explicitentropy}, see Appendix \ref{app:prooconcreteentropyWN}.

\begin{corollary}\label{cor:explicitentropy}
Consider the setting of Corollary \ref{cor:MSE}. Let $\pG=\pG(N)$, $N\in \N$, be a sequence of regression classes\footnote{We may think of $N$ as the number of
parameters of $\pG$, e.g.~the size of the FrameNet class $\pG_{\rm FN}$ in Section \ref{sec:networkarchitecture} below.} such that
$\inf_{G^*\in\pG(N)}\lVert \gstar-\G_0\rVert_{L^2(\gamma)}^2\lesssim N^{-\beta}$ for some $\beta>0$ and all $N\in\N$.
Denote the entropy by $H(N,\delta)\coloneqq H(\pG(N),\|\cdot\|_{\infty,\supp(\gamma)},\delta)$.
\begin{enumerate}
    \item\label{item:concrete1} If $H(N,\delta)\lesssim \delta^{-\alpha}$ for some $0<\alpha<2$, then 
      \[  \IE_{\G_0}\,\Big[\lVert \hat{G}_n-\G_0\rVert_{L^2(\gamma)}^2\Big]\lesssim n^{-\frac{2}{2+\alpha}}.\]
    \item\label{item:concrete2} If $H(N,\delta)\lesssim N\log(\delta^{-1})$, then for all $\tau>0$
    \[  \IE_{\G_0}\,\Big[\lVert \hat{G}_n-\G_0\rVert_{L^2(\gamma)}^2\Big]\lesssim n^{-\frac{\beta}{\beta+1}+\tau},\]
    where the constant in $\lesssim$ in general depends on $\tau$.
\end{enumerate}
\end{corollary}
\begin{remark}[Effective smoothness]
In classical nonparametric regression over $s$-smooth function classes on $[0,1]^d$, the entropy assumption $H(N,\delta)\lesssim \delta^{-\alpha}$ from part (i) is fulfilled for $\alpha=d/\mathfrak{s}$, which yields the minimax-optimal rate $n^{-\frac{2\mathfrak{s}}{2\mathfrak{s}+d}}$, see Section \ref{sec:finitedimensional} for details. Since the rate only depends on $\alpha$ (or equivalently $\alpha^{-1}$), we can think of $\alpha^{-1}=\mathfrak{s}/d$ as the ``effective smoothness'' of the statistical model at hand.
\end{remark}

The sub-Gaussian noise model poses a more restrictive regularity assumption on $\eps_i$ than the assumption of white noise. It is 
possible to get $L^2(\gamma)$-convergence for sub-Gaussian noise in the case the entropy integral $J(\delta)$ in \eqref{eq:Jintegral} is not finite, such that Assumption \ref{assump:g0assump} \ref{item:Gassump21} does not hold. The details are shown in the next theorem, its proof is deferred to Appendix \ref{app:subgaussnoJ}. 
\begin{theorem}
\label{th:samplecomplexitysub-Gaussiannoise}
   Consider the nonparametric regression model \eqref{eq:dataxy} with sub-Gaussian noise and the empirical risk from \eqref{eq:empiricalrisminimizer}. 
Let Assumption \ref{assump:g0assump} \ref{item:Gassump11} hold, i.e.~suppose that $\|G_0\|_{\infty,\supp(\gamma)}<\infty$ and $\pG$ is compact with respect to $\|\cdot\|_{{\infty,\supp(\gamma)}}$.
 Then there exists some universal constant $C_1>0$, such that for any positive sequences $(\delta_n)_{n\in\IN}$ and $(\tildedeltan)_{n\in\IN}$ with 
      \begin{equation}\label{eq:deltandeltanhatSBGN}
           n\delta_n^4\geq C_1^2\sigma^2\Finf^2H\left(\frac{\delta_n^2}{8\sigma^2+\delta_n^2}\right)\quad
        \text{and}\quad
        n\tildedeltan^2\geq 6\Finf^2 H(\pG,\norm[{\infty,\supp(\gamma)}]{\cdot},\tildedeltan),
      \end{equation}
      all $G^*\in\pG$ and all
\begin{align}\label{eq:RinequalitySBGN}
          R\geq\max\left\{\delta_n,\tildedeltan,\lVert G^*-G_0\rVert_{{\infty,\supp(\gamma)}},\frac{\Finf}{\sqrt{n}}\right\},
      \end{align}
      we have
      \begin{align}\label{eq:L2concentrationSBGN}
          \IP_{G_0}\left(\lVert\hatgn-G_0\rVert_\ltg\geq R\right)\leq 4\exp\left(-\frac{ nR^4}{C_1^2\sigma^2(1+F_\infty^2)}\right)+2\exp\left(-\frac{nR^2}{C_1^2\Finf^2}\right).
      \end{align}  
    Furthermore, for all $n\in\IN$ there exists a universal constant $C_2>0$ such that
   \begin{align}\label{eq:SBGNconvergence2}
            \IE_{\G_0}\,\Big[\lVert \hat{G}_n-\G_0\rVert_{L^2(\gamma)}^2\Big]
          \leq C_2\Big(\delta_n^2+\tildedeltan^2+\frac{(1+\sigma)(1+\Finf^2)}{\sqrt{n}}+\inf_{\gstar\in \pG}\lVert \gstar-\G_0\rVert_{L^2(\gamma)}^2\Big).
    \end{align}
\end{theorem}

\begin{remark}\label{rmk:SBGNconcreteentopy}
Consider $\alpha\ge 2$ in Corollary \ref{cor:explicitentropy} \ref{item:concrete1}. Then $J(\delta)$ in \eqref{eq:Jintegral} is not necessarily finite, and hence Assumption \ref{assump:g0assump} \ref{item:Gassump21} need not be satisfied. Therefore Theorem \ref{th:samplecomplexitywhitnoise} cannot be applied. In the sub-Gaussian noise case, we may still use the $L^2(\gamma)$-bound in \eqref{eq:SBGNconvergence2} however. Similar as in the proof of Corollary \ref{cor:explicitentropy}, it can then be shown that $\delta_n^2\lesssim n^{-\frac{1}{2+\alpha}}$ and $\tildedeltan^2\lesssim n^{-\frac{2}{2+\alpha}}$ satisfy \eqref{eq:deltandeltanhatSBGN}. This yields
\begin{align*}
\IE_{\G_0}\,\Big[\lVert \hat{G}_n-\G_0\rVert_{L^2(\gamma)}^2\Big]
  \lesssim n^{-\frac{1}{2+\alpha}},
\end{align*}
i.e.~half the convergence rate of the ``chaining regime'' considered in Corollary \ref{cor:explicitentropy}.
\end{remark}  

\section{Learning Holomorphic Operators with FrameNet}\label{sec:networkarchitecture}
In this section we first recall the NN-based operator class FrameNet from
\cite[Section 2]{Herrmann.2024}, see Sections
\ref{sec:representationssystems} and \ref{sec:framNet}.
This will provide the regression class
  $\pG_{{\rm FN}}$ over which to estimate $G_0$. Similar to, e.g.,
  PCA-Net \cite{HESTHAVEN201855}, FrameNet consists of mappings
\begin{equation}
\label{eq:Gg}
    G=\D_{\Y}\circ g\circ E_{\X},
  \end{equation}
  for a linear encoder $\E_{\X}:\X\to \ell^2(\IN)$, a
  linear decoder $\D_{\Y}: \ell^2(\IN)\to\Y$,
  and a coefficient map $\ell^2(\IN)\to\ell^2(\IN)$.
  The encoder maps an $x\in\X$ to its coefficients in some representation system of $\X$. Conversely, the decoder builds a $y\in\Y$ out of a coefficient series in $\ell^2(\IN)$. These representation systems consist of apriorily fixed frames. The coefficient map $g$ is represented by a feedforward neural network, that will be trained by ERM.
  
  Subsequently, in Sections \ref{sec:framenetapprox}-\ref{sec:Framenetstatistic}, we apply our analysis
  to the learning of \emph{holomorphic} operators $G_0:\X\to\Y$. For such
  mappings, the FrameNet architecture was shown in \cite{Herrmann.2024}
  to be capable of overcoming the curse of dimensionality in terms of
  the \emph{approximation error}. We generalize this property to the case of bounded network parameters in subsection \ref{sec:framenetapprox} and furthermore show metric entropy bounds.
  This allows us to 
  prove that FrameNet can overcome the curse of dimensionality in the learning of holomorphic
  operators, both in terms the \emph{approximation capability} and in terms of 
  \emph{sample complexity}.
  
\subsection{Representation Systems}\label{sec:representationssystems}
We briefly recall basic definitions and properties of frames.
For more details see for instance \cite{Christensen.2016}.
\subsubsection{Frames}\label{sec:frames}
\begin{definition}
    \label{def:frame}
    A family $\pmb{\Psi}=\{\psi_j:j\in\IN\}\subset \X$ is called a frame of $\X$, if the \emph{analysis operator} 
    \begin{equation*}
        F:\X\to \ell^2(\IN),\qquad v\mapsto\bigl(\langle v,\psi_j\rangle_{\X}\bigr)_{j\in\IN} 
    \end{equation*}
    is bounded and boundedly invertible between $\X$ and $\range(F)\subset \ell^2(\IN)$.
  \end{definition}
  Every orthonormal basis of $\X$ is trivially a frame. Since Definition \ref{def:frame} merely requires bounded invertibility on the range of $F$, $F$ need not be surjective, and in particular $\pmb{\psi}$ need not consist of linearly independent vectors. The \emph{frame bounds} of $\pmb{\psi}$ are defined as
\begin{equation}\label{eq:FrameConstant}
    \Lambda_{\pmb{\Psi}}\coloneqq\lVert F\rVert_{\X\to \ell^2}=\sup_{0\neq v\in\X}\frac{\lVert F v\rVert_{\ell^2}}{\lVert v\rVert_{\X}}, \quad\lambda_{\pmb{\Psi}}\coloneqq\inf_{0\neq v\in\X}\frac{\lVert Fv\rVert_{\ell^2}}{\lVert v\rVert_{\X}},
  \end{equation}
the \emph{synthesis operator} $F'$ as
\begin{equation*}
    F':\ell^2(\IN)\to\X,\quad\bigl(v_i\bigr)_{i\in\IN}\mapsto\pmb{v}^T\pmb{\psi}\coloneqq\sum_{i\in\IN}v_i\psi_i,
\end{equation*}
and finally the \emph{frame operator} as $T\coloneqq F'F:\, \X\to\X$. The following lemma gives a characterization of $T$. For a proof, see \cite[Lemma 5.1.5]{Christensen.2016}.
\begin{lemma}
\label{ref:frameoperator}
The frame operator $T$ is boundedly invertible, self-adjoint and positive. Furthermore, it holds that $\lVert T\rVert_{\X\to\X}=\Lambda_{\pmb{\Psi}}^2$ and $\lVert T{^{-1}}\rVert_{\X\to\X}=\lambda_{\pmb{\Psi}}^{-2}$. The family $\tilde{\pmb{\Psi}}\coloneqq T^{-1}\pmb{\Psi}$ is a frame of $\X$, called the \emph{(canonical) dual frame} of $\X$. The analysis operator of the dual frame is $\tilde{F}\coloneqq F(F'F)^{-1}$ and its frame bounds are $\lambda_{\pmb{\Psi}}^{-1}$ and $\Lambda_{\psi}^{-1}$.
\end{lemma}
\begin{definition}
    \label{def:Rieszbasis}
    A family $\pmb{\Psi}=\{\psi_j:\,j\in\IN\}\subset\X$ is called a Riesz basis of $\X$ if there exists a bounded, bijective operator $A:\,\X\to\X$ and an orthonormal basis $(e_j)_{j\in\IN}$ with $\psi_j=Ae_j$ for all $j\in\IN$.
  \end{definition}
A Riesz basis is a frame $\pmb{\Psi}$ which is also a basis. Equivalently, a Riesz basis is a frame with $\trm{ker}(F')=0$ and therefore $\trm{range}(F)=\ell^2(\IN)$.
Moreover, the dual frame $\tilde{\pmb{\Psi}}$ of a Riesz basis is also a Riesz basis,
e.g.\ \cite[Section 5]{Christensen.2016}.
\subsubsection{Encoder and Decoder}\label{sec:decoderencoder}
Throughout the rest of this paper we fix frames and their duals on $\X$, $\Y$ and denote them by
\begin{equation}
\label{eq:framesXY}
    \pmb{\Psi}_{\X}=(\psi_j)_{j\in\IN},\quad\tilde{\pmb{\Psi}}_{\X}=(\tilde{\psi}_j)_{j\in\IN}, \quad  \pmb{\Psi}_{\Y}=(\eta_j)_{j\in\IN},\quad\tilde{\pmb{\Psi}}_{\Y}=(\tilde{\eta}_j)_{j\in\IN}.
\end{equation}
The corresponding analysis operators are $F_{\X}$, $\tilde{F}_{\X}$, $F_{\Y}$ and $\tilde{F}_{\Y}$.  
We then introduce encoder and decoder maps
via
\begin{align}
    \label{eq:definiitonDE}
    &\E_{\X}\coloneqq\tilde{F}_{\X}=\begin{cases}
        \X\to \ell^2(\IN),\\
        x\mapsto(\langle x,\tilde{\psi}_j\rangle_{\X})_{j\in\IN},
    \end{cases}\qquad
        \D_{\Y}\coloneqq F_{\Y}'=\begin{cases}
        \ell^2(\IN)\to\Y,\\
        (y_j)_{j\in\IN}\mapsto\sum_{j\in\IN}y_j\eta_j.
    \end{cases}
\end{align}
In case $\pmb{\Psi}_{\X}$ and $\pmb{\Psi}_{\Y}$ are Riesz bases, the mappings in \eqref{eq:definiitonDE} are boundedly invertible.

\subsubsection{Smoothness Scales}\label{sec:smoothness}
The encoder mapping $\E_\X:\X\to\ell^2(\IN)$ maps an element to its coefficients in the frame representation. For computational purposes, this coefficient sequence must be truncated, as only finitely many coefficients can be considered. Consequently, it is essential to control the error resulting from discarding higher-order frame coefficients. To formalize this we next introduce scales of subspaces of $\X$, $\Y$, whose elements exhibit a certain coefficient decay.

\begin{definition}
  \label{def:smoothnessscales}  Let $\pmb{\theta}=(\theta_j)_{j\in\IN}$ be a strictly positive, monotonically decreasing sequence such that $\pmb{\theta}^{1+\eps}\in l^{1}(\IN)$ for all $\eps>0$.

  For all $r$, $t\ge0$ we introduce the subspaces $\X_{r}\subset\X$ and $\Y_t\subset \Y$ via
    \begin{align*}
        \X_r\coloneqq\{x\in\X:\,\lVert x\rVert_{\X_r}<\infty\}\qquad\text{and}\qquad
        \Y_t\coloneqq\{y\in\Y:\,\lVert y\rVert_{\Y_t}<\infty\}
    \end{align*}
    where
\begin{align*}
    \lVert x\rVert_{\X_r}^2\coloneqq\sum_{j\in\IN}\langle x,\tilde{\psi}_j\rangle_{\X}^2 \theta_j^{-2r}\qquad\text{and}\qquad
    \lVert y\rVert_{\Y_t}^2\coloneqq\sum_{j\in\IN}\langle y,\tilde{\eta}_j\rangle_{\Y} ^2\theta_j^{-2t}.
\end{align*}
\end{definition}
For every $r\ge 0$, $\X_r$ is a Hilbert space {\cite[Lemma 1]{Herrmann.2024}}.

\subsection{FrameNet}\label{sec:framNet}
In this subsection we recall the FrameNet architecture from \cite[Section 2]{Herrmann.2024}.
We start by formally introducing feedforward neural networks (NNs) following \cite{Opschoor.2022,Herrmann.2024}.  

\subsubsection{Feedforward Neural Networks}
\begin{definition}
  \label{def:NN}
  A function $f:\IR^{p_0}\to\IR^{p_{L+1}}$ is called a neural network, if there exists
  $\sigma:\IR\to\IR$, integers $p_1,\dots,p_{L+1}\in\IN$, $L\in\IN$
  and real numbers $w_{i,j}^l$, $b_j^l\in\R$ such that for all $\px=(x_i)_{i=1}^{p_0}\in\R^{p_0}$
\begin{subequations}\label{eq:NN}
\begin{align}
    z^1_j&=\sigma\biggl(\sum_{i=1}^{p_0}w_{i,j}^1x_i+b_j^1\biggr),\quad j=1,\dots,p_1,\label{eq:NNdef1}\\
        z^{l+1}_j&=\sigma\biggl(\sum_{i=1}^{p_{l}}w_{i,j}^{l+1}z_i^l+b_j^{l+1}\biggr),\quad l=1,\dots, L-1,\quad j=1,\dots,p_{l+1},\label{eq:NNdef2} \\
f(x)&=(z_j^{L+1})_{j=1}^{p_{L+1}}=\biggl(\sum_{i=1}^{p_L}w_{i,j}^{L+1}z_i^L+b_j^{L+1}\biggr)_{j=1}^{p_{L+1}}.\label{eq:NNdef3}
\end{align}
\end{subequations}
We call $\sigma$ the \emph{activation function}, $L$ the \emph{depth},
$p\coloneq \max_{l=0,\dots,L+1} p_l$ the \emph{width},
 $w_{i,j}^l\in\IR$ the \emph{weights}, and $b_j^l\in\IR$
the \emph{biases} of the NN.
\end{definition}

While different NNs can realize the same function, for simplicity we refer to a function $f:\R^{p_0}\to\R^{p_{L+1}}$ as an NN of type \eqref{eq:NN}, if it allows for (at least) one such representation. Additionally, our analysis will require some further terminology: For a NN $f:\R^{p_0}\to\R^{p_{L+1}}$ as in \eqref{eq:NN} its
\begin{itemize}
\item \emph{size} is the number of nonzero parameters
  \begin{equation*}
    {\rm size}(f)\coloneqq|\{(i,j,l)\,:\,w_{i,j}^l\neq 0\}|+|\{(j,l)\,:\,b_j^l\neq 0\}|,
  \end{equation*}
\item \emph{maximum of parameters} is
  \begin{equation*}
    {\rm mpar}(f)\coloneqq\max\,\Big\{\max_{i,j,l}|w_{i,j}^l|,\max_{j,l}|b_j^l|\Big\},
  \end{equation*}  
\item \emph{maximum range} on $\Omega\subseteq\R^{p_0}$ is
  \begin{equation*}
    {\rm mran}_{\Omega}(f)\coloneqq\sup_{x\in\Omega}\bigl\lVert f(x)\bigr\rVert_{2}=\sup_{x\in\Omega}\left(\sum_{j=1}^{p_{L+1}}\bigl(z_j^{L+1}(x)\bigr)^2\right)^{\frac{1}{2}}.
  \end{equation*}
\end{itemize}

Throughout, for $q\in\N$, $q\ge 1$, we consider the activation function
\begin{equation*}
    \sigma_q(x)\coloneqq\max\{0,x\}^{q},\qquad x\in\IR.
\end{equation*}
For $q=1$, $\sigma_1$ is the \emph{rectified linear unit} (ReLU), for $q\ge 2$,
$\sigma_q$ is 
called \emph{rectified power unit} (RePU).

\begin{remark}
    Definition \ref{def:NN} introduces 
    a NN as a function $f:\IR^{p_0}\to\IR^{p_{L+1}}$. Throughout, we also understand the realization of a NN as a map $f:\ell^2(\IN)\to \ell^2(\IN)$ via extension by zeros. This is equivalent to suitably padding the weight matrices $(w_{i,j}^l)_{i,j}$ and bias vectors $(b_j^l)_j$ in Definition \ref{def:NN} for $l\in\{0,L+1\}$ with infinitely many zeros, see \cite[Remark 13]{Herrmann.2024}.
\end{remark}

\subsubsection{The FrameNet Class}
By definition of frames, the encoding operator $\E_\X:\X\to\ell^2(\N)$
in \eqref{eq:definiitonDE} is injective, and the decoding operator $\D_\Y:\ell^2(\N)\to\Y$ is
surjective. Thus for every mapping $G:\X\to\Y$, there exists a
coefficient map $g:\ell^2(\N)\to\ell^2(\N)$ such that
\begin{equation}
  G = \D_\Y\circ g\circ\E_X.
\end{equation}
This motivates the introduction of the following function class.

Given $\sigma:\R\to\R$, positive integers $L$, $p$, $s\in\IN$, and reals $M$, $B\in\R$, let
\begin{align*}
    \label{eq:functionspaceg}
  \pmb{g}_{\trm{FN}}(\sigma,L,p,s,M,B)\coloneqq\bigl\{ &g:\IR^{p_0}\to\IR^{p_{L+1}}\text{ is NN with activation function $\sigma$ s.t.}\\
  &{\rm depth}(g)\le L,~
  {\rm width}(g)\le p,~
  {\rm size}(g)\le s,\\
  &{\rm mpar}(g)\le M,~{\rm mran}_{[-1,1]^{p_0}}(g)\le B,~p_0,p_{L+1}\le p\bigr\}.
\end{align*}
Instead of directly considering \eqref{eq:Gg}, for our analysis, contrary to \cite{Herrmann.2024},
with $\pmb{\theta}$ from Definition \ref{def:NN}, fixed $R>0$, and $U:=[-1,1]^\IN$, it will be convenient to introduce the linear scaling
\begin{equation}
    \label{eq:scaling}
    S_r\coloneq\begin{cases}
      \times_{j\in\IN}[-R\theta_j^r,R\theta_j^r]\to U\\
      (x_j)_{j\in\IN}\mapsto\big(\frac{x_j}{R\theta_j^r}\big)_{j\in\IN}.
      \end{cases}
\end{equation}
The FrameNet class then consists of all operators
\begin{equation}
\label{eq:functionclassG}
    \pmb{G}_{\rm FN}(\sigma,L,p,s,M,B)\coloneqq\bigl\{G=\D_{\Y}\circ g\circ S_r\circ\E_{\X}:\, g\in\pmb{g}_{\trm{FN}}(\sigma,L,p,s,M,B)\bigr\}.
  \end{equation}

  \subsection{Approximation of Holomorphic Operators
  }\label{sec:framenetapprox}
  Our statistical theory established in Section \ref{sec:framework} shows that sample complexity depends on 
  \begin{enumerate}
      \item the approximation quality of $\pG$ w.r.t~$G_0$, cf.~the term $\inf_{G^*\in\pG}\|G^*-G_0\|_{L^2(\gamma)}$ in \eqref{eq:WNconvergence},
      \item the metric entropy of $\pG$, cf.~the terms $\delta_n^2$ and $\tildedeltan^2$ in \eqref{eq:WNconvergence}.
  \end{enumerate}
  In this subsection we give results for both the approximation quality and the metric entropy of the FrameNet class $\pG_{\rm FN}$.
\subsubsection{Setting}
  Let us start by making the assumptions on $G_0$ and the sampling
  measure $\gamma$ on $\X$ more precise. Denote in the following
  $U\coloneqq [-1,1]^{\IN}$, let $r>\frac{1}{2}$, $R>0$, and let the
  frames $(\psi_j)_{j\in\IN}$, $(\eta_j)_{j\in\IN}$ be as in Section
  \ref{sec:networkarchitecture}, and $(\theta_j)_{j\in\IN}$ as in
  Definition \ref{def:smoothnessscales}.  Then
\begin{equation}
\label{eq:sigmarR}
    \sigma_R^r\coloneqq\begin{cases}
        U\to\X,\\
        \By\mapsto R\sum_{j\in\IN}\theta_j^ry_j\psi_j
    \end{cases}
\end{equation}
yields a well-defined map, since by construction $(y_j\theta_j^r)_{j\in\IN}\in\ell^2(\N)$
for every $\By\in U$, and $(\psi_j)_{j\in\IN}$ is a frame.
Next, we introduce the ``cubes''
\begin{align}
    \label{eq:CrR}
    C_R^r(\X)=\biggl\{a\in\X:\sup_{j\in\IN}\theta_j^{-r}|\langle a,\tilde{\psi}_j\rangle_{\X}|\leq R\biggr\}.
\end{align}
Observe that with $\sigma_R^r(U)\coloneq \set{\sigma_R^r(\By)}{\By\in U}$, clearly
\begin{equation*}
  C_R^r(\X)\subseteq \sigma_R^r(U),
\end{equation*}
but equality holds in general only if the $(\psi_j)_{j\in\IN}$ form a Riesz
basis \cite[Remark 10]{Herrmann.2024}.
\begin{example}
  Let $\X=\IR$, $r=1$, $R=1$, $\theta_1=3/2$ and $\theta_2=1/2$. Consider the frame $\pmb{\Psi}=\{1,1\}$ (which is not a basis) with dual analogue $\pmb{\tilde \Psi}=\{1/2,1/2\}$. Then with $U=[-1,1]^2$
  \begin{equation*}
    C_1^1(\IR)=[-1,1]\subsetneq [-2,2]=\set{\theta_1 y_1\psi_1+\theta_2y_2\psi_2}{\By\in U}.
  \end{equation*}
\end{example}

The holomorphy assumption on $G_0$ can now be formulated as follows, \cite[Assumption 1]{Herrmann.2024}. Recall that for a real Hilbert space $\X$, we denote by $\X_{\IC}$ its complexification.
  \begin{assumption}
    \label{assump:holomorphicextension}
   
      For some $r>1$, $R>0$, $t>0$, $C_{G_0}<\infty$ there exists an open set
      $O_{\IC}\subset\X_{\IC}$ containing $\sigma_R^r(U)$, such that
      \begin{enumerate}[label=(\alph*)]
      \item      $\sup_{a\in O_{\IC}}\lVert G_0(a)\rVert_{\Y_{\IC}^t}\leq
      C_{G_0}$, and $ G_0:O_{\IC}\to\Y_{\IC}$ is holomorphic,
      \item $\gamma$ is a probability measure on $\X$ with $\supp(\gamma)\subseteq C_R^r(\X)$.
      \end{enumerate}
    \end{assumption}   

    The assumption requires $G_0$ to be holomorphic on a superset of $\sigma_R^r(U)$. The probability measure $\gamma$ is allowed to have support on $C_R^r(\X)$ which is potentially smaller than $\sigma_R^r(U)$. In particular, $G_0$ is then holomorphic on a superset of the support of $\gamma$.
    
    \begin{example}\label{ex:pi}
      Denote by $\lambda$ the Lebesgue measure. Then $\pi\coloneq\otimes_{j\in\N}\frac{\lambda}{2}$ is the uniform probability measure on $U=[-1,1]^\IN$, and
      \begin{equation}\label{eq:gamma}
        \gamma\coloneq (\sigma_R^r)_\sharp\pi
      \end{equation}
      defines a probability measure on $\X$ with support $\sigma_R^r(U)$.
      If $\pmb{\psi}$ is a Riesz basis, then $\supp(\gamma)=C_R^r(\X)$,
      so that $\supp(\gamma)\subseteq C_R^r(\X)$ as required
      in Assumption \ref{assump:holomorphicextension}.
    \end{example}   

\subsubsection{Approximation Theory}\label{sec:NNapproximationtheory}
Theorem 1 in \cite{Herrmann.2024} establishes a convergence rate for approximating $G_0$ within $\pG_{\rm FN}$ (with unbounded weights), in terms of the number of trainable network parameters. Our main results on sample complexity depend on bounds on the metric entropy, which require bounded network weights. To address this, we now extend Theorem 1 from \cite{Herrmann.2024} to FrameNet architectures with bounded weights, as introduced in Section \ref{sec:networkarchitecture}.

Similar as in \cite{Schwab2019Deep,Opschoor2022Bayesian,Herrmann.2024}, the analysis builds on polynomial chaos expansions, e.g.\ \cite{Xiu2002}. Denote by $(L_j)_{j\in\IN}$ the univariate Legendre polynomials normalized such that $\frac{1}{2}\int_{-1}^1L_j(x)^2\dd x = 1$ for all $j\in\IN$. Then, \cite[Chapter 22]{abramowitz1948handbook}
    \begin{equation}
        \label{eq:LinfLegendreunivarainte}
        \sup_{x\in[-1,1]}\lvert L_j(x)\rvert\leq\sqrt{2j+1}\qquad \forall j\in\IN_0.
    \end{equation}
Next, let $\F$ be the set of infinite-dimensional multiindices with finite support, i.e.\
\begin{equation}\label{eq:multiindices}
    \F\coloneqq\Big\{(\Bnu_j)_{j\in\IN_0}\in\IN_0^{\IN}:\,|\supp\Bnu|<\infty\Big\},
  \end{equation}
  where $\supp\Bnu\coloneq\set{j}{\nu_j\neq 0}$.
  For finite sets of multiindices $\Lambda\subseteq\F$, their effective dimension and maximal order is defined as
  \begin{equation*}
    d(\Lambda)\coloneq \sup\set{|\supp\Bnu|}{\Bnu\in\Lambda}\qquad\text{and}\qquad
    m(\Lambda)\coloneq \sup\set{|\Bnu|}{\Bnu\in\Lambda},
  \end{equation*}
  where $|\Bnu|\coloneq \sum_{j\in\IN}\nu_j$. Moreover, with $U\coloneq [-1,1]^\IN$
  for all $\By\in U$ and $\Bnu\in\F$, we let $L_{\Bnu}(\By)\coloneq \prod_{j\in\IN}L_{\nu_j}(y_j)$ be the corresponding multivariate Legendre polynomial. This infinite product is well-defined, since for all but finitely many $j$ holds $\nu_j=0$ so that $L_{\nu_j}\equiv 1$.
  The next Proposition gives an approximation result for multivariate Legendre polynomials.
  It is an 
  extension of \cite[Proposition 2.13]{Opschoor.2022} to the case of bounded network
  parameters. The proof is 
  provided in Appendix \ref{app:prooftensorizedsig1}.
\begin{proposition}[$\sigma_1$-NN approximation of $L_{\Bnu}$]
\label{prop:tensorlegendreapproximation}
Let $\delta\in (0,1/2)$ and $\Lambda\subset\F$ be finite. 
Then there exists a $\sigma_1$-NN $f_{\Lambda,\delta}$, such that its outputs $\{\tilde{L}_{\Bnu,\delta}\}_{\Bnu\in\Lambda}$ satisfy
\begin{equation}
\label{eq:Linfboundtenorizedlegendre}
    \forall\Bnu\in\Lambda:\qquad \sup_{\py\in U}\,|L_{\Bnu}(\py)-\tilde{L}_{\Bnu,\delta}(\py)|\leq \delta.
\end{equation}
Furthermore, there exists a constant $C>0$ independent of $\Lambda$, $d(\Lambda)$, $m(\Lambda)$ and $\delta$ such that
\begin{subequations}
    \begin{align}
      {\rm depth}(f_{\Lambda,\delta})&\leq C\Big[\log(d(\Lambda))d(\Lambda)\log(m(\Lambda))m(\Lambda)+ m(\Lambda)^2+\log(\delta^{-1})\bigl(\log(d(\Lambda))+ m(\Lambda)\bigr)\Big]\label{eq:Legdnresigma1approxtensorNNNdepth}\\
{\rm width}\bigl(f_{\Lambda,\delta}\bigr)&\leq C|\Lambda|d(\Lambda)\label{eq:Legdnresigma1approxNNNtensorwidth}\\
       {\rm size}\bigl(f_{\Lambda,\delta}\bigr)&\leq C\Bigl[|\Lambda|d(\Lambda)^2\log\bigl(m(\Lambda)\bigr)+m(\Lambda)^3d(\Lambda)^2+\log\bigl(\delta^{-1}\bigr)\bigl(m(\Lambda)^2+|\Lambda|\bigr)d(\Lambda)\Bigr]\\ 
{\rm mpar}\bigl(f_{\Lambda,\delta}\bigr)&\leq1.\label{eq:Legdnresigma1approxNNNtensorweightbound}
    \end{align}
  \end{subequations}
\end{proposition}

A similar result holds for RePU neural networks, see Proposition \ref{prop:tensorlegendreapproximationsigq}. 
Given $q\in\IN$, 
and $N\in\IN$, we introduce the two FrameNet classes
\begin{subequations}\label{eq:architectures}
   \noeqref{eq:Framenet}
  \begin{equation}
    \begin{aligned}
    \GNNSP&:=\pmb{G}_{\trm{FN}}(\sigma_q,{\rm width}_N,{\rm depth}_N,{\rm size}_N,M,B),\label{eq:Framenet}\\
      \GNNFC&:=\pmb{G}_{\trm{FN}}(\sigma_q,{\rm width}_N,{\rm depth}_N,\infty,M,B)
    \end{aligned}
    \end{equation}
    where for certain constants $C_L$, $C_p$, $C_s$, $M$, $B\ge 1$ (to be determined later)
    \begin{align}\label{eq:Framenetconstants}
      {\rm depth}_N = \max\{1,\lceil C_L \log(N)\rceil\},\qquad
     {\rm width}_N = \lceil C_p N\rceil,\qquad
      {\rm size}_N = \lceil C_sN\rceil,
    \end{align}
\end{subequations}    
    and $\lceil x \rceil$ 
    denotes the smallest integer larger or equal to $x\in\IR$.
    
    We emphasize that $\GNNSP$ corresponds to a sparsely connected architecture, whereas $\GNNFC$ represents a fully connected architecture (because there is no constraint on its size). In particular, since every linear transformation in between activation functions has at most ${\rm width}_N+{\rm width}_N^2$ parameters, we have
    \begin{equation}\label{eq:fullsize}
      {\rm size}(\GNNFC)\le ({\rm depth}_N+1)({\rm width}_N+{\rm width}_N^2)
      =O(\log(N)N^2)\qquad\text{as }N\to\infty.
    \end{equation}
    Thus $\GNNFC$ can essentially be quadratically larger than $\GNNSP$.
    
The next theorem extends \cite[Theorems 1 and 2]{Herrmann.2024} to the case of bounded network parameters. 
The proof is provided in Appendix \ref{app:proofapproximationerror}.
    \begin{theorem}[Sparse network approximation]\label{th:approximationNN}
      Let $G_0$, $\gamma$ satisfy Assumption \ref{assump:holomorphicextension} with $r>1$, $t>0$.      
      Let $q\geq 1$ be an integer and fix $\tau>0$ (arbitrarily small).
        \begin{enumerate}
        \item \label{item:Linfapprox}
    There exists $C>0$ s.t.\ for all $N\in\IN$
    \begin{align*}
\inf_{G\in\GNNSP}\,\biggl[\bigl\lVert G- G_0\bigr\rVert_{\infty,\supp(\gamma)}^2\biggr]\leq C N^{-2\min\{r-1,t\}+\tau}.
    \end{align*}
  \item \label{item:Ltwoapprox}Let $\pmb{\Psi}_{\X}$ be a Riesz basis and let 
    $\gamma$ be as in \eqref{eq:gamma}.
 Then there exists $C>0$ s.t.\ for all $N\in\IN$
        \begin{align*}
       \inf_{G\in\GNNSP}\,\Big[\bigl\lVert G- G_0\bigr\rVert_{L^2(\gamma)}^2\Big]\leq C N^{-2\min\{r-\frac{1}{2},t\}+\tau}.
    \end{align*}
    \end{enumerate}
  \end{theorem}
  
\begin{remark}
  Theorem \ref{th:approximationNN} provides an approximation rate for $\GNNSP$ in
  terms of $N$.  Since $\GNNFC\supseteq \GNNSP$, trivially the statement
  remains true for $\GNNFC$. However, as the size of $\GNNFC$ can be
    quadratically larger by \eqref{eq:fullsize}, the convergence rate in terms of
    \emph{network size} is essentially halfed for the fully connected
    architecture.
\end{remark}

\subsubsection{Entropy Bounds for FrameNet}
In the following we bound the metric entropy (cp.~\eqref{eq:entropy}) of FrameNet for ReLU activation. Recall that $\Lambda_{\pmb{\Psi}_{\Y}}$ is the frame constant in \eqref{eq:FrameConstant}. The proof of Lemma \ref{lem:entropybound} is given in Appendix \ref{app:proofentropy}.

\begin{lemma}[cf.~{\cite[Lemma 5]{SH20}}]\label{lem:entropybound}
  Let $L$, $p$, $s$, $M$, $B\ge 1$, $\sigma_R^r$ be as in \eqref{eq:sigmarR} and $U=[-1,1]^{\IN}$. Then $\pG_{\rm FN}=\pG_{\rm FN}(\sigma_1,L,p,s,M,B)$ is compact with respect to $\|\cdot\|_{\infty,\sigma_R^r(U)}$ and $\|\cdot\|_n$.
  Furthermore, $\pG_{\rm FN}$ satisfies for all $\delta>0$
     \begin{align}\label{eq:Hsig1general}
H(\pG_{\rm FN},\|\cdot\|_{\infty,\sigma_R^r(U)},\delta)\leq(s+1)\log\Big(2^{L+6}\Lambda_{\pmb{\Psi}_{\Y}}L^2M^{L+1}p^{L+4}\max\{1,\delta^{-1}\}\Big).
     \end{align}
In particular there exist constants $C_H^\trm{SP},\,C_H^\trm{FC}>0$ such that for the sparse and fully-connected FrameNet classes from \eqref{eq:architectures} it holds
\begin{align}
H(\GNNSP[\sigma_1],\|\cdot\|_{\infty,\sigma_R^r(U)},\delta)&\leq C_H^\trm{SP} N\left(1+\log(N)^2+\log\left(\max\left\{1,\delta^{-1}\right\}\right)\right)\label{eq:Hsig1sparse}\\
 H(\GNNFC[\sigma_1],\|\cdot\|_{{\infty,\sigma_R^r(U)}},\delta)&\leq C_H^\trm{FC} N^2\left(1+\log(N)^3+\log\left(\max\left\{1,\delta^{-1}\right\}\right)\right)\label{eq:Hsig1FC}
 \end{align}
for all $\delta>0$.
\end{lemma}

  \begin{remark}
    The entropy bound is independent of the constant $B$  bounding the maximum range of the network, see Subsection \ref{sec:framNet}. However, $B<\infty$ will be necessary to apply Theorem \ref{th:samplecomplexitywhitnoise}.
  \end{remark}
For RePU activation, as mentioned before, the metric entropy bounds exhibit a worse dependency on the network parameters, due to the lack of global Lipschitz continuity of $\sigma_q$ if $q\ge 2$.
The proof of Lemma \ref{lem:entropyboundRePU} is given in Appendix \ref{app:proofentropyRePU}.
\begin{lemma}\label{lem:entropyboundRePU}
  Let $q\in\IN$, $q\ge 2$, and $L$, $p$, $s$, $M$, $B\ge 1$, $\sigma_R^r$ be as in \eqref{eq:sigmarR} and $U=[-1,1]^{\IN}$. Then $\pG_{\rm FN}=\pG_{\rm FN}(\sigma_q,L,p,s,M,B)$ is compact with respect to $\|\cdot\|_{{\infty,\sigma_R^r(U)}}$ and $\|\cdot\|_n$.
  Furthermore, $\pG_{\rm FN}$ satisfies for all $\delta>0$
     \begin{align}\label{eq:Hsigqgeneral}
       H(\pG_{\rm FN},\|\cdot\|_{{\infty,\sigma_R^r(U)}},\delta)\leq(s+1)\log\Big(\Lambda_{\pmb{\Psi}_{\Y}} L q^{L+q}(2p M)^{4q^{2L+2}}\max\{1,\delta^{-1}\}\Big).
     \end{align}
    Consider the constant $C_L$ from \eqref{eq:architectures}. Then there exists $C_H^\trm{SP},\,C_H^\trm{FC}>0$ such that
\begin{align}
H(\GNNSP[\sigma_q],\|\cdot\|_{{\infty,\sigma_R^r(U)}},\delta)&\leq C_H^\trm{SP} N^{1+2C_L\log(q)}\left(1+\log(N)^2+\log\left(\max\left\{1,\delta^{-1}\right\}\right)\right)\label{eq:Hsigqsparse}\\
 H(\GNNFC[\sigma_q],\|\cdot\|_{{\infty,\sigma_R^r(U)}},\delta)&\leq C_H^\trm{FC} N^{2+2C_L\log(q)}\left(1+\log(N)^2+\log\left(\max\left\{1,\delta^{-1}\right\}\right)\right)\label{eq:HsigqFC}
 \end{align}
for all $\delta>0$.
\end{lemma}
\subsection{Statistical Theory}\label{sec:Framenetstatistic}
Using our statistical results from Theorem \ref{th:samplecomplexitywhitnoise} and Corollaries \ref{cor:MSE}-\ref{cor:explicitentropy} as well as the approximation result from Theorem \ref{th:approximationNN}, we can bound 
$\IE_{\G_0}[\lVert \hat{G}_n-\G_0\rVert_{L^2(\gamma)}^2]$
solely in terms of the statistical sample size $n$. This is formalized in the next two theorems, for ReLU and RePU activation.

Our result distinguishes between the sparse and fully connected architectures $\GNNSP$, $\GNNFC$ in \eqref{eq:architectures}. In practice, fully connected architectures are often preferred, due to their simpler implementation, and because training sparse NN architectures can run into problems like ``bad'' local minima, see, e.g., \cite{Evci.2019}. Our theoretical upper bounds are sharper for sparse architectures; this is because the additional free parameters in the fully connected architecture increase the entropy of this class, but do not yield better approximation properties in our proofs. 

\begin{theorem}[ERM for white noise and ReLU]
  \label{th:RgnG0boundWN}
  Let $G_0:\X\to\Y$ and $\gamma$ satisfy Assumption \ref{assump:holomorphicextension} for some $r>1$, $t>0$.
  Fix $\tau>0$ (arbitrarily small) and set (cp.~Example \ref{ex:pi})
    \begin{align}\label{eq:kappa}
        \kappa\coloneq\begin{cases}
           2\min\{r-\frac{1}{2},t\} &\text{if }\pmb{\Psi}_{\X}\text{ is a Riesz basis and } \gamma=(\sigma_R^r)_{\#}\pi,\\
           2\min\{r-1,t\} &\text{otherwise.}
        \end{cases}
   \end{align}
   For every $n\in\IN$, let $(x_i,y_i)_{i=1}^n$ be data generated by \eqref{eq:dataxy} either in the \hyperlink{WN}{white noise model} or in the \hyperlink{SGN}{sub-Gaussian noise model}.
   Then there exist constants
   $C_L$, $C_p$, $C_s$, $M$, $B\geq1$ in \eqref{eq:architectures} and $C>0$
   (all independent of $n$) such that
    \begin{enumerate} 
    \item\label{item:sparseReLU}
      {\bf Sparse FrameNet:} with $N=N(n)=\lceil n^{\frac{1}{\kappa+1}}\rceil$ and $\pG=\GNNSP[\sigma_1]$, there exists a measurable choice of an ERM $\hat G_n$ of \eqref{eq:empiricalrisminimizer}. Any such $\hat G_n$ satisfies
          \begin{equation}
            \IE_{\G_0}[\lVert \hat{G}_n-\G_0\rVert_{L^2(\gamma)}^2]
            \leq Cn^{-\frac{\kappa}{\kappa+1}+\tau},\label{eq:RGnG0bound}
    \end{equation}
  \item\label{item:fullyconnectedReLU}
    {\bf Fully connected FrameNet:}
with $N=N(n)=\lceil n^{\frac{1}{\kappa+2}}\rceil$ and $\pG=\GNNFC[\sigma_1]$, there exists a measurable choice of an ERM $\hat G_n$ of \eqref{eq:empiricalrisminimizer}. Any such $\hat G_n$ satisfies
        \begin{equation}
          \IE_{\G_0}[\lVert \hat{G}_n-\G_0\rVert_{L^2(\gamma)}^2]
        \leq Cn^{-\frac{\kappa}{\kappa+2}+\tau}.\label{eq:RGnG0boundFC}
    \end{equation}
    \end{enumerate}
\end{theorem}
One can further use Theorem \ref{th:samplecomplexitywhitnoise} to show concentration inequalities for the risk $\lVert \hat{G}_n-\G_0\rVert_{L^2(\gamma)}$.

\begin{proof}
  The proof follows directly from the entropy bounds in Lemma \ref{lem:entropybound} and Corollary \ref{cor:explicitentropy} \ref{item:concrete2}: First,
  Lemma \ref{lem:entropybound} in particular verifies Assumption \ref{assump:g0assump} for all $G^*\in\pG$, which is required for Corollary \ref{cor:explicitentropy}.
  Applying the corollary with
  $\beta=\kappa$ then gives \eqref{eq:RGnG0bound},
  and with
  $\beta=\frac{\kappa}{2}$ we obtain \eqref{eq:RGnG0boundFC}.
  Note that crucially $\Finf$ (see \eqref{eq:F-bound}) does not depend on $N$, because $\|G_0\|_{\infty,\supp{\gamma}} <\infty$ and $\pG_{\rm FN}$ is universally bounded by the $N$-independent constant $B$, see \eqref{eq:functionclassG}.
\end{proof}
\begin{theorem}[ERM for white noise and RePU]
    \label{th:RgnG0boundWNREPU}
    Consider the setting of Theorem \ref{th:RgnG0boundWN} and let $q\in\IN$, $q\ge 2$.
   There exist constants
   $C_L$, $C_p$, $C_s$, $M$, $B\ge 1$ in \eqref{eq:architectures} and $C>0$ (all independent of $N$) such that 
    \begin{enumerate} 
    \item\label{item:sparseRePU}
      {\bf Sparse FrameNet:} with $N=N(n)=\lceil n^{\frac{1}{\kappa+1+4C_L\log(q)}}\rceil$ and $\pG=\GNNSP[\sigma_q]$, there exists a measurable choice of an ERM $\hat G_n$ of \eqref{eq:empiricalrisminimizer}. Any such $\hat G_n$ satisfies
     \begin{equation}
          \IE_{\G_0}[\lVert \hat{G}_n-\G_0\rVert_{L^2(\gamma)}^2]    
        \leq Cn^{-\frac{\kappa}{\kappa+1+4C_L\log(q)}+\tau}.\label{eq:RGnG0boundRePU}
    \end{equation}
  \item\label{item:fullyconnectedRePU}
    {\bf Fully connected FrameNet:}
    with $N=N(n)=\lceil n^{\frac{1}{\kappa+2+4C_L\log(q)}}\rceil$ and $\pG=\GNNFC[\sigma_q]$, there exists a measurable choice of ERM $\hat G_n$ of \eqref{eq:empiricalrisminimizer}. Any such $\hat G_n$ satisfies
        \begin{equation}
          \IE_{\G_0}[\lVert \hat{G}_n-\G_0\rVert_{L^2(\gamma)}^2]          
        \leq Cn^{-\frac{\kappa}{\kappa+2+4C_L\log(q)}+\tau}.\label{eq:RGnG0boundRePUFC}
    \end{equation}
    \end{enumerate}
\end{theorem}
\begin{proof}
Similar to the proof of Theorem  \ref{th:RgnG0boundWN}, the proof of Theorem \ref{th:RgnG0boundWNREPU} follows from the entropy bounds in Lemma \ref{lem:entropyboundRePU} below and Corollary \ref{cor:explicitentropy} \ref{item:concrete2} with $\gamma=\kappa/(1+2C_L\log(q))$ for \eqref{eq:RGnG0boundRePU} and $\gamma=\kappa/(2+2C_L\log(q))$ for \eqref{eq:RGnG0boundRePUFC}. 
\end{proof}
A few remarks are in order. In the RePU case the activation function $\sigma_q(x)=\max\{0,x\}^q$ has no global Lipschitz condition for $q\geq 2$. As a result, the entropy bounds obtained for the corresponding FrameNet class are larger than for ReLU. This leads to worse convergence rates.
Moreover, for the RePU case, the convergence rate depends on the constant $C_L$ in \eqref{eq:Framenetconstants}. The proof of Theorem \ref{th:approximationNN} shows that $C_L$ depends on the decay properties of the Legendre coefficients $(c_{\Bnu_i,j})_{i,j\in\IN}$ of the function $G_0\circ\sigma_R^r$, i.e.~$C_L$ depends on $G_0$ (see \eqref{eq:1normnu} and \eqref{eq:suppnu} in Theorem \ref{th:algebraicboundonlegendrecoefficients}). Explicit bounds on $C_L$ are possible, see \cite[Lemma 1.4.15]{Zech.2018}. 

\section{Applications}\label{sec:applications}
We now present two applications of our results. First, in finite dimensional regression, our analysis recovers well-known minimax-optimal rates for standard smoothness classes. This indicates that our main statistical result in Theorem \ref{th:samplecomplexitywhitnoise} is \emph{in general optimal}. However, we do not claim optimality specifically for the approximation of holomorphic operators discussed in Section \ref{sec:networkarchitecture}. Second, for an infinite dimensional problem we address the learning of a solution operator to a parameter dependent PDE.

\subsection{Finite Dimensional Regression}\label{sec:finitedimensional}
Let $d\in\N$ and $\mathcal X=\R^d$, $\mathcal Y=\R$.
Moreover, let $D\subseteq\R^d$ be a bounded, open, smooth domain, and
$G_0:D\to\R$ a ground-truth regression function. Suppose that
\begin{subequations}\label{eq:datafinite}
\begin{equation}
  x_i\overset{\text{iid}}{\sim}\gamma\quad\text{and}\quad
  \eps_i\overset{\text{iid}}{\sim}\mathcal{N}(0,1)\qquad\forall i\in\IN
\end{equation}
are independent samples for some probability measure $\gamma$ on $D$, and let
\begin{equation}
  y_i= G_0(x_i) +\eps_i\qquad\forall i\in\IN.
\end{equation}
\end{subequations}
Given a regression class $\pG$ of measurable mappings from $D\to\R$, the least-squares problem is to determine
\begin{equation}\label{nonparam:ls}
    \hat G_n \in \argmin_{G\in \pG} \sum_{i=1}^n -2 G(x_i)y_i+G(x_i)^2= \argmin_{G\in\pG} \sum_{i=1}^n|G(x_i) -y_i|^2.
\end{equation}
For $\mathfrak{s}\ge 0$, it is well-known that the minimax-optimal rate of recovering a ground-truth function $G_0$ in $\mathfrak{s}$-smooth function classes, such as the Sobolev space $H^\mathfrak{s}(D)$ or the Besov space $B_{\infty,\infty}^\mathfrak{s}(D)$ (see, e.g., \cite{ET96} for definitions), equals $n^{-\frac{2\mathfrak{s}}{2\mathfrak{s}+d}}$, e.g., \cite{Gine.op.2016, vanGeer.2009}.

Denote now by $\pG^\mathfrak{s}_R$ the ball of radius $R>0$ around the origin in either $H^\mathfrak{s}(D)$ or $B_{\infty,\infty}^\mathfrak{s}(D)$. Then,
\begin{equation*}
  H(\pG^\mathfrak{s}_R,\norm[{\infty,\supp(\gamma)}]{\cdot},\delta) \simeq \Big(\frac{R}{\delta}\Big)^{d/\mathfrak{s}}\qquad\forall \delta\in (0,1),
\end{equation*}
which holds for all $\mathfrak{s}>0$ if $\pG_R^\mathfrak{s}$ is the ball in $B_{\mathfrak{s},\mathfrak{s}}^\infty(D)$,
and for all $\mathfrak{s}>d/2$ in case $\pG^\mathfrak{s}_R$ is the ball in $H^\mathfrak{s}(D)$, see Theorem 4.10.3 in \cite{Triebel.1995}.
Corollary \ref{cor:explicitentropy} \ref{item:concrete1} (with $\alpha=d/\mathfrak{s}<2$) then directly yield the following theorem. It recovers the minimax optimal rate for nonparametric least squares/maximum likelihood estimators.
\begin{theorem}\label{th:besovfinitedimensional} Let $R>0$ and $\mathfrak{s}>d/2$. 
  Then, there exists $C>0$ such that for all $G_0\in \pG_R^\mathfrak{s}$,
  the estimator $\hat G_n$ in \eqref{nonparam:ls} with $\pG=\pG_R^\mathfrak{s}$
  and data as in \eqref{eq:datafinite} satisfies
  \begin{equation*}
    \mathbb E_{G_0}\big[ \|\hat G_n - G_0\|_{L^2(D)}^2 \big]\le Cn^{-\frac
{2}{2+\alpha}} = Cn^{-\frac
{2\mathfrak{s}}{2\mathfrak{s}+d}}\qquad \forall n\in\IN.
\end{equation*}
\end{theorem}

\subsection{Parametric Darcy Flow}
\label{sec:DiffEqTorus}
As a second application we apply Theorem \ref{th:RgnG0boundWN} to the solution operator of the diffusion equation, extending the discussion of approximation errors in \cite[Section 7.1]{Herrmann.2024}.
\subsubsection{Setup}\label{sec:torusG}
We recall the setup from \cite[Sections 7.1.1, 7.1.2]{Herrmann.2024}.

Let $d\in\N$, and denote by $\bbT^d\simeq [0,1]^d$ the $d$-dimensional
torus. In the following, all function spaces on $\bbT^d$ are
understood to be one-periodic in each variable. Fix
$\bar{a}\in L^\infty(\bbT^d)$ and $f\in H^{-1}(\bbT^d)/\IR$ such that
for some constant $a_{\min}>0$
\begin{equation}\label{eq:u_bar}
    \essinf_{x\in \bbT^d} (\bar{a}(x)+a(x)) > a_{\min}.
\end{equation}
We consider the ground truth $G_0:a\mapsto u$, mapping
$a\in L^\infty(\bbT^d)$ to the solution $u\in H^1(\bbT^d)$ of
\begin{equation} \label{eq:DiffPeriodic}
 - \nabla \cdot ((\bar{a} + a )\nabla u) = f 
\text{ on } \bbT^d\qquad
\text{ and }\qquad \int_{\bbT^d}u(x)\dd x  =0.
\end{equation}
Then $G_0:\set{a\in L^\infty(\bbT^d)}{\text{\eqref{eq:u_bar} holds}}\to H^1(\bbT^d)$
is well-defined.

To represent $a$ and $u$, we use Fourier expansions on $\bbT^d$.
Denote for $j\in\N_0$ and $\bsj\in\N_0^d$, $d\ge 2$,
\begin{equation*}
\xi_0 \coloneq 1,\quad
\xi_{2j}(x) \coloneq \sqrt 2\cos(2\pi j x),\quad
\xi_{2j-1}(x) \coloneq \sqrt 2\sin(2\pi jx),\quad
\xi_\bsj(x_1,\dots,x_d)\coloneq \prod_{k=1}^d \xi_{j_k}(x_k).
\end{equation*}
Then for $r\ge 0$, $\set{\max\{1,|\bsj|\}^r\xi_\bsj}{\bsj\in\IN_0^d}$ forms an ONB of $H^r(\bbT^d)$ equipped with inner product
\begin{align*}
    \langle u,v\rangle_{H^r(\bbT^d)}\coloneqq\sum_{\pmb{j}\in\IN_0^d}\langle u,\xi_{\pmb{j}}\rangle_{L^2}\langle v,\xi_{\pmb{j}}\rangle_{L^2}\max\left\{1,\lvert \pmb{j}\rvert\right\}^{2r}.
\end{align*}
In the following, fix $r_0$, $t_0\ge 0$ and set
\begin{equation}\label{eq:cXscYt_torus}
  \begin{aligned}
  \cX&:=H^{r_0}(\bbT^d),\qquad\psi_\bsj:=\max\{1,|\bsj|\}^{-r_0}\xi_\bsj,\\
  \cY&:=H^{t_0}(\bbT^d),\qquad\eta_\bsj:=\max\{1,|\bsj|\}^{-t_0}\xi_\bsj,
  \end{aligned}
\end{equation}
so that $\bsPsi_\cX:=(\psi_{\bsj})_{\bsj\in\IN_0^d}$,
$\bsPsi_\cY:=(\eta_{\bsj})_{\bsj\in\IN_0^d}$ form ONBs of $\cX$, $\cY$
respectively.  The encoder $\cE_\X$ and decoder $\cD_\Y$ are now as in
\eqref{eq:definiitonDE}. Direct calculation shows $\X^r=H^{r_0+rd}$
and $\Y^t=H^{t_0+td}$ for $r$, $t\geq 0$; for more details see
\cite[Section 7.1.2]{Herrmann.2024}.

\subsubsection{Sample Complexity}
We now analyze the sample complexity for learning the PDE solution operator $G_0$ in Section \ref{sec:torusG}. For a proof of Theorem \ref{thm:torus}, see Appendix \ref{app:prooftorus}. 

\begin{theorem}\label{thm:torus}
  Let $d\in\N$, $d\geq 2$, $\mathfrak{s}>\frac{3d}{2}$ and $t_0\in[0,1]$.
  Fix $\tau_1>0$,
  $\tau_2\in (0,\min\{\mathfrak{s}-\frac{3d}{2},\frac{\tau_1 d}{8}\})$ (both
  arbitrarily small), and set
\begin{equation}\label{eq:torusr0}
  r_0 =\begin{cases}
        \frac{d}{2}+\tau_2 &\text{if }\mathfrak{s}\in (\frac{3d}{2},2d+1-t_0]\\
        \frac{\mathfrak{s}+t_0-1}{2} &\text{if }\mathfrak{s} > 2d+1-t_0.
        \end{cases}
      \end{equation}
      Moreover let $f\in C^\infty(\bbT^d)$, and let
  \begin{enumerate}[label=(\alph*)]
  \item {\bf ground truth}: $G_0:a\mapsto u$ be given through \eqref{eq:DiffPeriodic},
  \item {\bf representation system}: $\cE_\X$, $\cD_\Y$ be as in \eqref{eq:definiitonDE} with the orthonormal basis in \eqref{eq:cXscYt_torus}, and $r_0$, $t_0$ from above,
  \item {\bf data}: $\gamma$ be the measure defined in \eqref{eq:sigmarR} and \eqref{eq:gamma}
    with $r=\frac{\mathfrak{s}-r_0}{d}$
    such that $\bar a+a$ satisfies \eqref{eq:u_bar} for all $a\in \supp(\gamma)$,
    and let $(x_i,y_i)_{i=1}^n$ be generated by \eqref{eq:dataxy} with the \hyperlink{WN}{additive white noise model} or the \hyperlink{SGN}{sub-Gaussian noise model},
  \item {\bf regression class}: $\pG=\GNNSP[\sigma_1]$ be the $n$-dependent sparse FrameNet architecture in \eqref{eq:architectures} with $N(n)=\lceil n^{\frac{1}{\kappa+1}}\rceil$.
  \end{enumerate}

  Then there exists a constant $C>0$ such that for all $n\in\IN$ there
  exists a measurable ERM $\hat G_n\in\pG^{\rm sp}_{\rm FN}(\sigma_1,N(n))$ in
  \eqref{eq:empiricalrisminimizer}, and any such $\hat G_n$ satisfies
  \begin{equation}\label{eq:Rtorus}
          \IE_{\G_0}[\lVert \hat{G}_n-\G_0\rVert_{L^2(H^{r_0}(\bbT^d),\gamma;H^{t_0}(\bbT^d))}^2]    
\leq C n^{-\frac{\kappa}{\kappa+1}+\tau_1}
\end{equation}
where
  \begin{equation}\label{eq:toruskappa}
    \kappa =
    \begin{cases}
      2\min\bigl\{\frac{\mathfrak{s}}{d}-1,\frac{1-t_0}{d}\bigr\} &\text{if }\mathfrak{s}\in (\frac{3d}{2},2d+1-t_0],\\
      \frac{\mathfrak{s}+1-t_0}{d}-1 &\text{if }\mathfrak{s} > 2d+1-t_0.
    \end{cases}
  \end{equation}
\end{theorem}
\begin{remark}\label{rmk:torus}
    Consider the setting of Theorem \ref{thm:torus}, and let $\supp(\gamma)\subseteq C_R^r(\X)$. A slight modification of the proof of Theorem \ref{thm:torus} similar to \cite{Herrmann.2024} (using the approximation bound in Theorem \ref{th:approximationNN} \ref{item:Linfapprox} instead of \ref{item:Ltwoapprox}) then yields
  \begin{equation}\label{eq:Linftorus}
          \IE_{\G_0}[\lVert \hat{G}_n-\G_0\rVert_{L^2(H^{r_0}(\bbT^d),\gamma;H^{t_0}(\bbT^d))}^2]    
\leq C n^{-\frac{\kappa}{\kappa+1}+\tau_1}
  \end{equation}
    where for some (small) $\tau_2>0$
    \begin{equation*}
      (r_0,\kappa)=\begin{cases}
        (\frac{d}{2}+\tau_2,\frac{2\mathfrak{s}}{d}-3) &\text{if }\mathfrak{s}\in (\frac{3d}{2},\frac{3d}{2}+1-t_0]\\
        (\frac{\mathfrak{s}+t_0-\frac{d}{2}-1}{2},\frac{\mathfrak{s}+1-t_0}{d}-\frac{3}{2}) &\text{if }\mathfrak{s} > \frac{3d}{2}+1-t_0.
        \end{cases}
  \end{equation*}
  Since 
  \begin{align*}
      B_R(H^\mathfrak{s}(\bbT^d))&=B_R(\X^r)
    =\{x\in\X:~\|x\|_{\X^r}\leq R\}=\left\{x\in\X: \sum_{j\in\IN}\langle x,\tilde{\psi}_j\rangle_{\X}^2 \theta_j^{-2r}\leq R^2 \right\}\\
      &\subseteq\left\{a\in\X:\sup_{j\in\IN}\theta_j^{-r}|\langle a,\tilde{\psi}_j\rangle_{\X}|\leq R\right\}=C_R^r(\X),
  \end{align*}
  in particular, \eqref{eq:Linftorus} holds
  for any $\gamma$ with $\supp(\gamma)\subseteq B_R(H^\mathfrak{s}(\bbT^d))$. This shows Theorem \ref{thm:intro}.  
\end{remark}
Similar rates can also be obtained for this PDE model on a convex, polygonal domain $D\subset\bbT^2$ with Dirichlet boundary conditions.
The argument uses the Riesz basis constructed in \cite{Davydov.2005},
but is otherwise similar to the torus, for details see \cite[Section 7.2]{Herrmann.2024}.
Moreoever, using the RePU activation function, \eqref{eq:Rtorus} holds with convergence rate $\kappa
 /(\kappa+1+4C_L\log(q))$, where $\kappa$ is from \eqref{eq:toruskappa} and $C_L$ from \eqref{eq:Framenetconstants}. The proof is similar to Theorem \ref{thm:torus} using Theorem \ref{th:RgnG0boundWNREPU} instead of Theorem \ref{th:RgnG0boundWN}. Finally, rates for the fully-connected class $\GNNFC$ can be established using Theorems \ref{th:RgnG0boundWN} \ref{item:fullyconnectedReLU} and \ref{th:RgnG0boundWNREPU} \ref{item:fullyconnectedRePU}.

\section{Conclusions}\label{sec:discussion}
In this work, we established convergence theorems for empirical risk minimization to learn mappings $G_0$ between infinite dimensional Hilbert spaces. Our setting assumes given data in the form of $n$ input-output pairs, with an additive noise model. We discuss both the case of Gaussian white noise and sub-Gaussian noise. Our main statistical result, Theorem \ref{th:samplecomplexitywhitnoise}, bounds the mean-squared $L^2$-error $\IE[\lVert\hatgn-G_0\rVert_{L^2(\gamma)}^2]$ in terms of the approximation error, and algebraic rates in $n$ depending only on the metric entropy of $\pG$. This provides a general framework to study operator learning from the perspective of sample complexity.

In the second part of this work, we applied our statistical results to a specific operator learning architecture from \cite{Herrmann.2024}, termed FrameNet. 
As our main application, we showed that holomorphic operators
$\G_0$ can be learned with NN-based 
surrogates without suffering from the curse of dimension, cf.~Theorem \ref{th:RgnG0boundWN}.
Such results have wide applicability, as the required holomorphy assumption is well-established in the literature, and has been verified for a variety of models including for example general elliptic PDEs \cite{CDS10,CDS11,CohenDeVore,harbrecht2016analysis}, Maxwell's equations \cite{JSZ17}, the Calderon projector \cite{henriquez2021shape}, the Helmholtz equation \cite{Hiptmair2018,spence2023wavenumber} and also nonlinear PDEs such as the Navier-Stokes equations \cite{CSZ18_2319}. 

\newpage
\printbibliography[heading=bibintoc,title=References]
\newpage
\begin{appendices}
\section{Auxiliary Probabilistic Lemmas}\label{sec:secnotation}
We recall the classical Bernstein inequality.
\begin{lemma}[Bernstein's inequality]\label{lem:bernstein}
   Let $X_1,\dots,X_n$ be independent, centered RVs with finite second moments $\IE[X_i^2]<\infty$ and uniform bound $|X_i|\leq M$  for $i=1,\dots, n$. Then it holds
\begin{align*}
     \IP\left(\left|\sum_{i=1}^nX_i\right|\geq t\right)\leq 2\exp\left(-\frac{t^2}{2\sum_{i=1}^n \IE[X_i^2]+\frac{2}{3}Mt}\right),\quad t\geq0.
 \end{align*}
\end{lemma}
For a proof of Bernstein's inequality, see \cite[page 193]{Pollard.1984}.

\begin{lemma}[Basic Inequality]\label{lem:basicineq}
    Let $\eps_i$, $i=1,\dots,n$ be i.i.d.~white noise or sub-Gaussian noise. Then it holds for all $G^*\in\pG$
    \begin{align*}
      \lVert \hat{G}_n-\G_0\rVert_n^2\leq
      \lVert G^*-\G_0\rVert_n^2+
      \frac{2\sigma}{n}
        \sum_{i=1}^n\langle\eps_i,\hat{G}_n(x_i)-G^*(x_i)\rangle_{\Y}.
    \end{align*}
\end{lemma}
\begin{proof}
Let $\gstar\in\pG$ be arbitrary. Using the definition of $\hat{G}_n$ from \eqref{eq:empiricalrisminimizer}, it holds
\begin{align*}
    &\lVert \hat{G}_n-\G_0\rVert_n^2\\
    =\,\,&\frac{1}{n}\sumin\lVert\hatgn(x_i)\rVert_{\Y}^2-2\langle\G_0(x_i),\hatgn(x_i)\rangle_{\Y} +\lVert \G_0(x_i)\rVert_{\Y}^2\\
=\,\,&\frac{1}{n}\sumin\lVert\hatgn(x_i)\rVert_{\Y}^2-2\langle\G_0(x_i)+\sigma\eps_i,\hatgn(x_i)\rangle_{\Y} +2\sigma\langle\eps_i,\hatgn(x_i)\rangle_{\Y} +\lVert \G_0(x_i)\rVert_{\Y}^2\\
\leq\,\,& \frac{1}{n}\sumin\lVert\gstar(x_i)\rVert_{\Y}^2-2\langle\G_0(x_i)+\sigma\eps_i,\gstar(x_i)\rangle_{\Y} +2\sigma\langle\eps_i,\hatgn(x_i)\rangle_{\Y} +\lVert \G_0(x_i)\rVert_{\Y}^2\\
=\,\,& \frac{1}{n}\sumin\lVert\gstar(x_i)\rVert_{\Y}^2-2\langle\G_0(x_i),\gstar(x_i)\rangle_{\Y} +\lVert\G_0(x_i)\rVert_{\Y}^2 +2\sigma\langle\eps_i,\hatgn(x_i)-\gstar(x_i)\rangle_{\Y}\\
=\,\,&\lVert G^*-\G_0\rVert_n^2+\frac{2\sigma}{n}\sumin\langle\eps_i,\hatgn(x_i)-\gstar(x_i)\rangle_{\Y},
\end{align*}
which shows the claim.\end{proof}
Next, we state a generic chaining result from \cite[Theorem 3.2]{Dirksen.2015}, originally derived for finite index sets, for countable index sets $T$. We restrict ourselves to the case of real-valued stochastic processes.
\begin{lemma}\label{lem:dirksengenerichaining}
Let $T$ be a countable index set and $d:T\times T\to[0,\infty)$ a pseudometric. Furthermore, let $(X_t)_{t\in T}$ be an $\IR$-valued stochastic process such that for some $\alpha>0$ and all $s,t\in T$,
\begin{align}\label{eq:psialpha}
    \IP\left(\lvert X_t-X_s\rvert\geq ud(t,s)\right)\leq 2 \exp\left(-u^{\alpha}\right),\quad u\geq 0.
\end{align}
Then, there exists a constant $M>0$ depending only on $\alpha$, such that for all $t_0\in T$ it holds that
\begin{align*}
    \IP\left(\sup_{t\in T}\,\lvert X_t-X_{t_0}\rvert\geq M\left(J_{\alpha}(T,d)+u\sup_{s,t\in T} d(s,t)\right)\right)\leq\exp\left(-\frac{u^{\alpha}}{\alpha}\right),\quad u\geq 1,
\end{align*}
where $J_\alpha$ denotes the metric entropy integral
\begin{align*}
J_{\alpha}(T,d)=\int_0^{\infty}\left(\log N(T,d,u)\right)^{\frac{1}{\alpha}}\dd u.
\end{align*}
\end{lemma}
\begin{proof} Since $T$ is countable, we can write $T=\{t_j:\,j\in\IN\}$. Using this, we define $T_n=\{t_j:\,j\leq n\}$ for $n\in\IN$. Since $T_n$ is finite \cite[Theorem 3.2, Eq.~(3.2) and its proof]{Dirksen.2015} gives $\tilde{M}>0$ s.t.
\begin{align}
\label{eq:chainninglemma1}
    \left(\IE\left[\sup\limits_{t\in T_n} \lvert X_t-X_{t_0}\rvert^p\right]\right)^{\frac{1}{p}}\leq \tilde{M}\left(J_{\alpha}(T_n,d)+\sup_{s,t\in T_n} d(s,t)p^{\frac{1}{a}}\right)
\end{align}
for all $p\geq 1$, $t_0\in T$ and $n\in\IN$. In \eqref{eq:chainninglemma1}, we used \cite[Eq.\ (2.3)]{Dirksen.2015} to upper bound the $\gamma_{\alpha}$-functionals by the respective metric entropy integrals $J_{\alpha}$. The monotone convergence theorem shows \eqref{eq:chainninglemma1} for $T$ in the limit $n\to\infty$. Applying \cite[Lemma A.1]{Dirksen.2015} then gives the claim with $M=\exp(\alpha^{-1})\tilde{M}$.
\end{proof}

We now use Lemma \ref{lem:dirksengenerichaining} to establish the following concentration bound, which is tailored towards the empirical processes appearing in our proofs, cf.~Lemma \ref{lem:basicineq}. Note that this lemma can be viewed as a generalization of the key chaining Lemma 3.12 in \cite{NW20} to $\Y$-valued regression functions; the proof follows along the same lines.

\begin{lemma}[Chaining Lemma]\label{lem:chaining1} Let $\X, \Y$ be separable Hilbert spaces, and suppose $\Theta$ is a (possibly uncountable) set parameterizing a class of maps
\begin{align*}
    \HH =\left\{h_{\theta}:\,\X\to\Y,\,\theta\in\Theta\right\}.
\end{align*}
Consider an empirical process of the form
  \begin{align*}
      Z_n(\theta)=\frac{1}{n}\sum_{i=1}^n\langle h_{\theta}(x_i),\eps_i\rangle_{\Y},
  \end{align*}
where $x_1,\dots,x_n\in\X$ are fixed elements and $\eps_i,\dots, \eps_n$ are either \textbf{(i)} i.i.d.~Gaussian white noise processes indexed by $\mathcal Y$, or \textbf{(ii)} i.i.d.~sub-Gaussian random variables in $\mathcal Y$ with parameter $1$.

Recall the empirical seminorm $\|\cdot\|_n$. 
Suppose that
\begin{equation}\label{eq:hupper}
  \sup_{\theta\in\Theta}\,\lVert h_{\theta}\rVert_n=: U<\infty,
\end{equation}
and define the metric entropy integral 
\begin{align*}
    J(\HH,d_n)=\int_0^U\sqrt{\log N(\HH,d_n,\tau)}\dd\tau,\quad d_n(\theta,\theta')=\lVert h_{\theta}-h_{\theta'}\rVert_n.
\end{align*}
Let the space $(\Theta,d_n)$ be separable. Then $\sup_{\theta\in\Theta} Z_n(\theta)$ is measurable and there exists a universal constant $C_{{\rm Ch}}>0$ such that for all $\delta>0$ with 
  \begin{align}\label{eq:deltanchaining}
      \sqrt{n}\delta\geq C_{{\rm Ch}}J(\HH,d_n),
  \end{align}
  it holds
  \begin{align}
\IP\left(\sup_{\theta\in\Theta} |Z_n(\theta)|\geq\delta\right)\leq\exp\left(-\frac{8n\delta^2}{C_{{\rm Ch}}^2U^2}\right).
      \label{eq:chaining}
  \end{align}
\end{lemma}
\begin{proof} In both of the cases, we will apply \cite[Theorem 3.2]{Dirksen.2015}, which we stated in Lemma \ref{lem:dirksengenerichaining}.

\textbf{White noise case.} 
Let $\theta,\theta'\in\Theta$ be arbitrary. Since $\eps_i$, $i=1,\dots,n$ are independent white noise processes, we have $Z_n(\theta)-Z_n(\theta')\sim\NN(0,n^{-1}\lVert h_{\theta}-h_{\theta'}\rVert_n^2)$, i.e.~the increments of $Z_n$ are normal (recall that $\px$ is regarded as fixed here). Thus,
\[\IP\left(|Z_n(\theta)-Z_n(\theta')|\geq t\right)\leq 2\exp\left(-nt^2/(2\lVert h_{\theta}-h_{\theta'}\rVert_n^2)\right),~~~~~t\geq 0,\]
and
\begin{align}\label{eq:sbgnwrtdn}
    \IP\left(\left|Z_n(\theta)-Z_n(\theta')\right|\geq \frac{\sqrt{2}td_n(\theta,\theta')}{\sqrt{n}}\,\right)\leq  2\exp\left(-\frac{d_n(\theta,\theta')^2t^2}{\lVert h_{\theta}-h_{\theta'}\rVert_n^2}\right) =2\exp\left(-t^2\right),~ t\geq 0, 
\end{align}
which verifies the assumption \eqref{eq:psialpha} for $\alpha=2$ and $\bar{d}_n:=\sqrt{2}d_n/\sqrt{n}$.

Eq.~\eqref{eq:sbgnwrtdn} shows that the process $Z_n(\theta)$ is sub-Gaussian w.r.t.~the pseudometric $\bar{d}_n(\theta,\theta')=\sqrt{2}\lVert h_{\theta}-h_{\theta'}\rVert_n/\sqrt{n}$. Therefore \cite[Theorem 2.3.7 (a)]{Gine.op.2016} yields that $Z_n(\theta)$ is sample bounded and uniformly sample continuous. Since $(\Theta,d_n)$ is separable, so is $(\Theta,\bar{d}_n)$. Thus it holds
\begin{align}
\label{eq:countablesupremawn}\sup\limits_{\theta\in\Theta}Z_n(\theta)=\sup\limits_{\theta\in\Theta_0}Z_n(\theta)\quad {\rm a.s.},
\end{align}
where $\Theta_0\subset \Theta$ denotes a countable, dense subset. The right hand side of \eqref{eq:countablesupremawn} is measurable as a countable supremum. Therefore also the left hand side $\sup_{\theta\in\Theta}Z_n(\theta)$ is measurable. Applying Lemma \ref{lem:dirksengenerichaining} (to the countable set $\Theta_0$) and using \eqref{eq:countablesupremawn} gives that for some universal constant $M$ and all $\theta_{\dagger}\in\Theta$,
\begin{align}
\label{eq:Dirksenchaining}
\IP\left(\sup_{\theta\in\Theta}\left|Z_n(\theta)-Z_n(\theta_{\dagger})\right|\geq M\left(J(\HH,\bar{d}_n)+\frac{tU}{\sqrt{n}}\right)\right)\leq\exp\left(-\frac{t^2}{2}\right),\quad t\geq 1.
\end{align}
Due to \eqref{eq:hupper} it holds
  $N(\HH,d_n,\delta)=1$ for all $\delta\ge U$, and thus
  \begin{equation*}
    N(\HH,\bar d_n,\tau)=N(\HH,d_n,\sqrt{n}\tau/\sqrt 2)=1\qquad\forall \tau\ge \frac{\sqrt 2U}{\sqrt{n}}.
  \end{equation*}
  Substituting $\rho = \sqrt n \tau/\sqrt 2$ we get
\begin{align*}
  J(\HH,\bar d_n)&= \int_0^{\sqrt 2U/\sqrt{n}}\sqrt{\log N(\HH,\bar{d}_n,\tau)}\dd\tau\\
  &= \frac{\sqrt 2}{\sqrt{n}}\int_0^{U}\sqrt{\log N\Big(\HH,\bar{d}_n,\frac{\sqrt 2\rho}{\sqrt{n}}\Big)}\dd\rho\\
  &= \frac {\sqrt 2}{\sqrt n} \int_0^{U}\sqrt{\log N(\HH,d_n,\rho)}\dd\rho = \frac {\sqrt 2J(\HH, d_n)}{\sqrt n}
\end{align*}
and therefore 
\begin{align}
\label{eq:dirksen2}\IP\left(\sup_{\theta\in\Theta}|Z_n(\theta)-Z_n(\theta_{\dagger})|\geq \frac{\sqrt 2M}{\sqrt{n}}\left(J(\HH,d_n)+tU\right)\right)\leq\exp\left(-\frac{t^2}{2}\right),\quad t\geq 1.
\end{align}
Since $Z_n(\theta_{\dagger})\sim\NN(0,n^{-1}\lVert h_{\theta_{\dagger}}\rVert_n^2)$, it holds for all $\theta_{\dagger}\in\Theta$
\begin{align}
\label{eq:Zngaussiantails}
\IP\left(|Z_n(\theta_{\dagger})|\geq \frac{MtU}{\sqrt{n}}\right)\leq \exp\left(\frac{-M^2U^2t^2}{2\lVert h_{\theta_{\dagger}}\rVert_n^2}\right)\leq\exp\left(-\frac{M^2t^2}{2}\right),\quad t\geq 0.
\end{align}
Combining \eqref{eq:dirksen2} and \eqref{eq:Zngaussiantails} yields for $t\geq 1$
\begin{align*}
&\IP\left(\sup_{\theta\in\Theta} |Z_n(\theta)|\geq \frac{3M}{\sqrt{n}}\left(J(\HH, d_n)+tU\right)\right)\\
    \leq\,\,&
\IP\left(\sup_{\theta\in\Theta}|Z_n(\theta)-Z_n(\theta_{\dagger})|\geq \frac{\sqrt 2M}{\sqrt{n}}\left(J(\HH,d_n)+tU\right)\right)+\IP\left(|Z_n(\theta_{\dagger})|\geq \frac{M}{\sqrt{n}}\left(J(\HH,d_n)+tU\right)\right)\\
    \leq\,\,&\exp\left(\frac{-t^2}{2}\right)+ \IP\left(|Z_n(\theta_{\dagger})|\geq \frac{MtU}{\sqrt n}\right)\\
    \leq\,\,& \exp\left(\frac{-t^2}{2}\right)+\exp\left(\frac{-M^2t^2}{2}\right)\\
    =\,\,& 2\exp\left(\frac{-t^2}{2}\right),
\end{align*}
where we assumed without loss of generality that $M\geq 1$. Substitute $\delta=3M/\sqrt{n}\,(J(\HH,d_n)+tU)$, i.e.~$t=(\sqrt{n}\delta/3M-J(\HH,d_n))/U$. Because $N(\HH,d_n,\tau)\geq 2$ for $\tau\leq U/2$, we have that $J(\HH,d_n)\geq U/2\sqrt{\log(2)}>U/4$. Therefore $\sqrt{n}\delta\geq 15M J(\HH,d_n)\coloneqq \CCh J(\HH,d_n)$ implies $t\geq 1$ and thus
\begin{align}
\label{eq:chainingfinal}
    \IP\left(\,\sup_{\theta\in\Theta}|Z_n(\theta)|\geq\delta\right)\leq2\exp\left(-\frac{\left(\sqrt{n}\delta/(3M)-J(\HH,d_n)\right)^2}{2U^2}\right)\leq 2\exp\left(-\frac{8n\delta^2}{\CCh^2U^2}\right),
\end{align}
which gives the claim for the white noise case.

\textbf{Sub-Gaussian case.} For $\theta,\theta'\in\Theta$ it holds
\begin{align*}
    Z_n(\theta)-Z_n(\theta')=\frac{1}{n}\sum_{i=1}^n\langle\eps_i,h_{\theta}(x_i)-h_{\theta'}(x_i)\rangle_{\Y}.
\end{align*}
Since the centered RVs $n^{-1}\langle\eps_i,h_{\theta}(x_i)-h_{\theta'}(x_i)\rangle_{\Y}$ are i.i.d.~sub-Gaussian with parameter $n^{-1}\lVert h_{\theta}(x_i)\\-h_{\theta'}(x_i)\rVert_{\Y}$, the `generalized' Hoeffding inequality for sub-Gaussian variables (see \cite[Theorem 2.6.2]{Vershynin.2018}) implies that for some universal constant $c>0$~the increment $Z_n(\theta)-Z_n(\theta')$ is sub-Gaussian with parameter $c\lVert h_{\theta}-h_{\theta'}\rVert_n/\sqrt{n}$. Therefore
\begin{align*}
    \IP\left(\left|Z_n(\theta)-Z_n(\theta')\right|\geq \frac{\sqrt{2}tc}{\sqrt{n}}d_n(\theta,\theta')\right)\leq2\exp\left(-t^2\right),\quad t\geq 0.
\end{align*}
From here on, the proof is similar to the white noise case and we obtain \eqref{eq:chainingfinal} by absorbing $c$ into $C_{{\rm Ch}}$.
\end{proof}

\begin{lemma}\label{lem:variance} Let $x_i \overset{iid}{\sim}\gamma$ and consider the entropy $H(\delta)=H(\pG,\|\cdot\|_{\infty,\supp(\gamma)},\delta)$. Let $(\tildedeltan)_{n\in\IN}$ be a positive sequence with
\begin{align}\label{eq:hatdeltan}
    n\tildedeltan^2\geq 6\Finf^2 H(\tildedeltan).
\end{align}
Then for $R\geq\max\{8\tildedeltan,18\Finf/\sqrt{n}\}$, it holds that
\begin{align*}
    \IP\left(\lVert \hatgn-G_0\rVert_{L^2(\gamma)}\geq R,\,\,\lVert\hatgn-G_0\rVert_{L^2(\gamma)}\geq 2\lVert\hatgn-G_0\rVert_n\right)\leq2\exp\left(-\frac{nR^2}{320\Finf^2}\right).
\end{align*}
\end{lemma}
\begin{proof}
In the following we write $\|\cdot\|_{\infty}=\|\cdot\|_{\infty,\supp(\gamma)}$. Let $\pF=\{G-\G_0:\,G\in\pG\}$. Then for $R\geq 0$ it holds that
\begin{align}
\label{eq:varianceproof1}
&\IP_{G_0}\left(\lVert \hatgn-G_0\rVert_{L^2(\gamma)}\geq R,\,\,\lVert\hatgn-G_0\rVert_{L^2(\gamma)}\geq 2\lVert\hatgn-G_0\rVert_n\right)\notag\\
\leq\,\,&\IP_{G_0}\left(\sup_{F\in\pF,\,\lVert F\rVert_\ltg\geq R}\lVert F\rVert_\ltg\geq 2\lVert F\rVert_n\right).
\end{align}
For $s\in\IN$, we define $\pF_s=\{F\in\pF:\,sR\leq\lVert F\rVert_\ltg\leq(s+1)R\}$. As a union of disjoint events, it holds
\begin{align}
    \label{eq:varianceproof2}
 \IP_{G_0}\left(\sup_{F\in\pF,\,\lVert F\rVert_\ltg\geq R}\lVert F\rVert_\ltg\geq 2\lVert F\rVert_n\right)&=\sum_{s=1}^\infty\IP_{G_0}\left(\sup_{F\in\pF_s}\lVert F\rVert_\ltg\geq 2\lVert F\rVert_n\right) \\
 &\leq\sum_{s=1}^\infty\IP_{G_0}\left(\sup_{F\in\pF_s}\left|\lVert F\rVert_\ltg-\lVert F\rVert_n\right|\geq\frac{sR}{2}\right),\notag
\end{align}
where we used for $F\in\pF_s$ and $s\in\IN$
\begin{align*}
    \lVert F\rVert_\ltg\geq 2\Fnorm_n\implies \Fnorm_\ltg-\Fnorm_n\geq\frac{\Fnorm_\ltg}{2}\geq \frac{sR}{2}.
\end{align*}
Now consider some covering $\pF_s^*=\{F_{s,j}\}_{j=1}^ N$ with $ N= N(\pF_s,\|\cdot\|_{\infty},R/8)$. Thus for arbitrary $s\in\IN$ and $F\in\pF_s$ there exists $F_{s,j}\in\pF_s^*$ s.t.~$\lVert F-F_{s,j}\rVert_{\infty}\leq R/8$. We get
\begin{align*}
    \left|\Fnorm_\ltg-\Fnorm_n\right|&\leq \left|\Fnorm_\ltg-\Fsjnorm_\ltg\right|+\left|\Fsjnorm_\ltg-\Fsjnorm_n\right|\notag\\
    &\quad+\left|\Fsjnorm_n-\Fnorm_n\right|\\
    &\leq  \frac{R}{4}+\left|\Fsjnorm_\ltg-\Fsjnorm_n\right|.
\end{align*}
Therefore 
\begin{align}
    \label{eq:varianceproof3}
&\sum_{s=1}^\infty\IP_{G_0}\left(\sup_{F\in\pF_s}\left|\lVert F\rVert_\ltg-\lVert F\rVert_n\right|\geq\frac{s R}{2}\right)\notag\\
\leq\,\,& \sum_{s=1}^\infty\IP_{G_0}\left(\max_{j=1,\dots, N}\left|\Fsjnorm_\ltg-\Fsjnorm_n\right|\geq\frac{s R}{4}\right)\notag\\
\leq\,\,& \sum_{s=1}^\infty  N(\pF_s,\|\cdot\|_{\infty}, R/8)\max_{j=1,\dots, N}\IP_{G_0}\left(\left|\Fsjnorm_\ltg-\Fsjnorm_n\right|\geq\frac{s R}{4}\right),
\end{align}
where we used $sR/2-R/4\geq sR/4$ for $s\in\IN$.
Since $F_{s,j}\in\pF_s$, we have $\Fsjnorm_\ltg\geq s R$ and therefore 
\begin{align*}
     \left|\Fsjnorm_\ltg-\Fsjnorm_n\right|&=\frac{\left|\Fsjnorm_\ltg^2-\Fsjnorm_n^2\right|}{ \Fsjnorm_\ltg+\Fsjnorm_n}\\
     &\leq \frac{1}{s R}\left|\Fsjnorm_\ltg^2-\Fsjnorm_n^2\right|.
\end{align*}
Inserting into \eqref{eq:varianceproof3} gives 
\begin{align}
     \label{eq:varianceproof4}
&\sum_{s=1}^\infty  N(\pF_s,\|\cdot\|_{\infty}, R/8)\max_{j=1,\dots, N}\IP_{G_0}\left(\left|\Fsjnorm_\ltg-\Fsjnorm_n\right|\geq\frac{s R}{4}\right)\notag\\
\leq\,\,&\sum_{s=1}^\infty  N(\pF_s,\|\cdot\|_{\infty}, R/8)\max_{j=1,\dots, N}\IP_{G_0}\left(\left|\Fsjnorm_\ltg^2-\Fsjnorm_n^2\right|\geq\frac{s^2 R^2}{4}\right).
\end{align}
Define the variables
\begin{align*}
    Y_i=\frac{1}{n}\left(\Fsjnorm_\ltg^2-\lVert F_{s,j}(x_i)\rVert_\Y^2\right),\quad i=1,\dots,n.
\end{align*}
It holds for all $i=1,\dots,n$
\begin{align*}
    \IE[Y_i]&=0,\\
    |Y_i|&\leq\frac{2\Finf^2}{n},\\
    \IE[Y_i^2]&=\frac{1}{n^2}\IE\left[\left(\Fsjnorm_\ltg^2-\Fsjxinorm_\Y^2\right)^2\right]\\
    &=\frac{1}{n^2}\IE\left[\left(\Fsjnorm_\ltg^4-2\Fsjnorm_\ltg^2\Fsjxinorm_\Y^2+\Fsjxinorm_\Y^4\right)\right]\\
    &=\frac{1}{n^2}\IE\left[\Fsjxinorm_\Y^4\right]-\Fsjnorm_\ltg^4\\
    &\leq\frac{\Finf^2}{n^2}\Fsjnorm_\ltg^2.
\end{align*}
Applying Bernstein's inequality (Lemma \ref{lem:bernstein}) for the variables $Y_i$ yields
\begin{align}
    \label{eq:varianceproof42}
& \sum_{s=1}^\infty  N(\pF_s,\|\cdot\|_{\infty}, R/8)\max_{j=1,\dots, N}\IP_{G_0}\left(\left|\Fsjnorm_\ltg^2-\Fsjnorm_n^2\right|\geq\frac{s^2 R^2}{4}\right)\notag\\
 \leq\,\,&\sum_{s=1}^\infty  N(\pF_s,\|\cdot\|_{\infty}, R/8)\max_{j=1,\dots, N}\exp\left(-\frac{ns^4 R^4}{32\Finf^2(\Fsjnorm_\ltg^2+\frac{s^2 R^2}{6})}\right)\notag\\
\leq\,\,&\sum_{s=1}^\infty\exp\left(H(\pF_s,\|\cdot\|_{\infty}, R/8)-\frac{ns^2 R^2}{160\Finf^2}\right),
\end{align}
where we used $\Fsjnorm_\ltg^2\leq(s+1)^2 R^2\leq 4s^2 R^2$ for all $j=1,\dots, N$ and $s\in\IN$, since $F_{j,s}\in\pF_s$.

Since $H(\delta)/\delta^2$ is non-increasing in $\delta$ and $R\geq8\tildedeltan$, we have 
\begin{align*}
    \frac{ns^2 R^2}{160\Finf^2}\geq\frac{n}{3\Finf^2}\left(\frac{R}{8}\right)^2\geq 2 H(\pG,\|\cdot\|_{\infty}, R/8)\geq 2 H(\pF_s,\|\cdot\|_{\infty}, R/8).
\end{align*}
Therefore we get
\begin{align}
\label{eq:varianceproof5}
&\sum_{s=1}^\infty\exp\left(H(\pF_s,\|\cdot\|_{\infty}, R/8)-\frac{ns^2 R^2}{160\Finf^2}\right)\notag\\
\leq\,\,&\sum_{s=1}^\infty\exp\left(-\frac{ns^2 R^2}{320\Finf^2}\right)\leq\sum_{s=1}^\infty\frac{1}{s^2}\exp\left(-\frac{n R^2}{320\Finf^2}\right)<2\exp\left(-\frac{n R^2}{320\Finf^2}\right).
\end{align}
Note that for $x,y\geq 1$, we have $\exp(-xy)\leq\exp(-y)/x$.  Applying this with $x=s^2\geq 1$ and $y=n R^2/(320\Finf^2)\geq 1$ for $R\geq 18\Finf/\sqrt n$, together with $\sum_{s=1}^\infty 1/s^2=\pi^2/6<2$, gives \eqref{eq:varianceproof5}.
Combining \eqref{eq:varianceproof1}--\eqref{eq:varianceproof5} shows the claim.
\end{proof}
\begin{lemma}\label{lem:sub-Gaussianepsilon}
    Consider the sub-Gaussian noise model, i.e.~suppose $\lVert\eps_i\rVert_{\Y}$ are i.i.d.~sub-Gaussian with parameter $1$. 
    Abbreviate $H(\delta)=H(\pG,\|\cdot\|_{\infty,\supp(\gamma)},\delta)$. Then there exists a universal constant $C>0$ s.t.~for all positive sequences $(\delta_n)_{n\in\IN}$ with

    \begin{align}\label{eq:sbgndelatan}
    n\delta_n^4\geq C^2\sigma^2\Finf^2H\left(\frac{\delta_n^2}{8\sigma^2+\delta_n^2}\right),
    \end{align}
   all $G^*\in\pG$ and $R\geq\max\{\delta_n,\sqrt{2}\lVert G^*-G_0\rVert_n\}$, it holds for every $\px=(x_1,...,x_n)\in \X^n$
    \begin{align}\label{eq:concentartionvarianceSBGN}
        \IP_{G_0}^{\px}\left(\lVert\hatgn-G_0\rVert_n\geq R\right)\leq 4\exp\left(-\frac{nR^4}{C^2\sigma^2(1+F_{\infty}^2)}\right).
    \end{align}

\end{lemma}
\begin{proof}
In the following we write $\|\cdot\|_{\infty}=\|\cdot\|_{\infty,\supp(\gamma)}$.
For $ R^2\geq 2\lVert G^*-G_0\rVert_n^2$, we use the basic inequality (Lemma \ref{lem:basicineq}) to obtain
\begin{align}\label{eq:varianceSBGN}
    \IP_{G_0}^{\px}\left(\lVert \hatgn-G_0\rVert_n^2\geq R^2\right)&\leq \IP_{G_0}^{\px}\left(\bigl|\lVert\hatgn-G_0\rVert_n^2-\lVert G^*-G_0\rVert_n^2\big|\geq  \frac{R^2}{2}\right).\notag\\
    &\leq \IP_{G_0}^{\px}\left(\left|\frac{1}{n}\sumin\langle\eps_i,\hatgn(x_i)-G^*(x_i)\rangle_{\Y}\right|\geq \frac{R^2}{4\sigma}\right).
\end{align}
For $R> 0$, let $\pG^*=\{G-G^*,\,\,G\in\pG\}$ and $(G_j)_{j=1}^{N}$ with $N=N(\pG^*,\|\cdot\|_{\infty},R^2/(8\sigma \IE[\lVert \eps_i\rVert_{\Y}]+R^2))$ denote a minimal $\|\cdot\|_{\infty}$-cover of $\pG^*$. 

It holds for $R> 0$
\begin{align}
    \label{eq:sbgninitialsplit}
    &\IP_{G_0}^{\px}\left(\left|\frac{1}{n}\sumin\langle\eps_i,\hatgn(x_i)-G^*(x_i)\rangle_{\Y}\right|\geq \frac{R^2}{4\sigma}\right)\notag\\
    \leq\,\,& \IP_{G_0}^{\px}\left(\sup_{G\in\pG}\left|\frac{1}{n}\sumin\langle\eps_i,G(x_i)-G^*(x_i)\rangle_{\Y}\right|\geq \frac{R^2}{4\sigma}\right)\notag\\
    \leq\,\,& \IP_{G_0}^{\px}\left(\sup_{G\in\pG}\left|\frac{1}{n}\sumin\langle\eps_i,G(x_i)-G^*(x_i)-G_{j^*}(x_i)\rangle_{\Y}\right|+\max_{j=1,\dots N}\left|\frac{1}{n}\sumin\langle\eps_i,G_{j}(x_i)\rangle_{\Y}\right|\geq \frac{R^2}{4\sigma}\right)\notag\\
    \leq\,\,& \underbrace{\IP_{G_0}^{\px}\left(\sup_{G\in\pG}\left|\frac{1}{n}\sumin\langle\eps_i,G(x_i)-G^*(x_i)-G_{j^*}(x_i)\rangle_{\Y}\right|\geq\frac{R^2}{8\sigma}\right)}_{(i)}\notag\\
    &\quad+\underbrace{\IP_{G_0}^{\px}\left(\max_{j=1,\dots N}\left|\frac{1}{n}\sumin\langle\eps_i,G_{j}(x_i)\rangle_{\Y}\right|\geq \frac{R^2}{8\sigma}\right)}_{(ii)}.
\end{align}
We estimate the terms $(i)$ and $(ii)$ separatly. For $(i)$, use $\lVert G-G^*-G_{j^*}\rVert_{\infty}\leq R^2/(8\sigma \IE[\lVert\eps_i\rVert_{\Y}]+R^2)
$ for all $G\in\pG$ and estimate
\begin{align}
    \label{eq:sbgni}
&\IP_{G_0}^{\px}\left(\sup_{G\in\pG}\left|\frac{1}{n}\sumin\langle\eps_i,G(x_i)-G^*(x_i)-G_{j^*}(x_i)\rangle_{\Y}\right|\geq\frac{R^2}{8\sigma}\right)\notag\\
\leq\,\,&\IP_{G_0}^{\px}\left(\left|\frac{1}{n}\sumin\lVert\eps_i\rVert_{\Y}\frac{R^2}{8\sigma \IE[\lVert\eps_i\rVert_{\Y}]+R^2}\right|\geq\frac{R^2}{8\sigma}\right)\notag\\
\leq\,\,&\IP_{G_0}^{\px}\left(\left|\frac{1}{n}\sumin\lVert\eps_i\rVert_{\Y}-\IE[\lVert\eps_i\rVert_{\Y}]\right|\geq\frac{R^2}{8\sigma}\right)\leq 2\exp\left(-\frac{nR^4}{C^2\sigma^2}\right),
\end{align}
where we used the Hoeffding inequality for sub-Gaussian random variables \cite[Theorem 2.6.2]{Vershynin.2018}.

For $(ii)$, it holds 
\begin{align*}
&\IP_{G_0}^{\px}\left(\max_{j=1,\dots N}\left|\frac{1}{n}\sumin\langle\eps_i,G_{j}(x_i)\rangle_{\Y}\right|\geq \frac{R^2}{8\sigma}\right)\\
\leq\,\,&N\,\max_{j=1,\dots N}\IP_{G_0}^{\px}\left(\left|\frac{1}{n}\sumin\langle\eps_i,G_{j}(x_i)\rangle_{\Y}\right|\geq \frac{R^2}{8\sigma}\right)\\
\leq\,\,&2N\,\exp\left(-\frac{2nR^4}{C^2\sigma^2\Finf^2}\right)\leq 2\exp\left(H\left(\frac{R^2}{8\sigma^2+R^2}\right)-\frac{2nR^4}{C^2\sigma^2\Finf^2}\right),
\end{align*}
where we used $\IE[\lVert\eps_i\rVert_{\Y}]\leq\sigma$ (Cauchy-Schwarz) and \cite[Theorem 2.6.2]{Vershynin.2018} for $Y_i=n^{-1}\langle\eps_i,G_j(x_i)\rangle_{\Y}$, $i=1,\dots,n$.

Since $H(\delta)/\delta^2$ is non-increasing, also $H(\delta/(8\sigma^2+\delta))/\delta^2$ is non-increasing and thus it holds for $ R\geq\delta_n$ 
\begin{align*}
    \frac{nR^4}{C^2\sigma^2\Finf^2}\geq H\left(\frac{R^2}{8\sigma^2+R^2}\right).
\end{align*}
This gives for $R \geq\delta_n$
\begin{align}\label{eq:sbgnii}
    (ii)\leq 2\exp\left(-\frac{nR^4}{C^2\sigma^2\Finf^2}\right).
\end{align}
 Combining \eqref{eq:varianceSBGN}--\eqref{eq:sbgnii} gives the result.
\end{proof}
\section{Proofs of Section \ref{sec:framework}}
\subsection{Proof of Theorem \ref{th:empiricalerror}}\label{app:proofempiricaltot}
\subsubsection{Existence and Measurability of \texorpdfstring{$\hat{G}_n$}{Gn}}\label{app:existencegnhat}
\begin{proof}[Proof of Theorem \ref{th:empiricalerror} \ref{item:minimizer}]

\textbf{White noise case.} For $i=1,\dots,n$, denote the probability space of the RVs $x_i$ and the white noise processes $\eps_i$ as $(\Omega,\Sigma,\IP)$. 
Furthermore, equip the Hilbert spaces $\X$ and $\Y$ with Borel $\sigma$-algebras $\B_{\X}$ and $\B_{\Y}$ and the space $\pG$ with the Borel $\sigma$-algebra $\B_{\pG}$. Then, there exists an orthonormal basis $(\psi_j)_{j\in\IN}$ of $\Y$ and i.i.d.~Gaussian variables $Z_j\sim\NN(0,1)$ s.t.

\begin{align*}
    &\eps_i:\Omega\times\Y\to\IR,\notag\\
    &\eps_i(\omega,y)=\sum_{j=1}^{\infty}\langle y,\psi_j\rangle_{\Y} Z_j(\omega)
\end{align*}
for $i=1,\dots,n$, see also \cite[Example 2.1.11]{Gine.op.2016}

Recall the noise level $\sigma>0$. For $i=1,\dots, n$, we define
\begin{align}
    &u_i:\Omega\times\pG\to\IR,\notag\\
    &u_i(\omega,G)=2\sigma\eps_i\left(\omega,G(x_i(\omega))\right)+2\langle G_0(x_i(\omega)),G(x_i(\omega))\rangle_{\Y}-\lVert G(x_i(\omega))\rVert_{\Y}^2.\label{eq:ui}
\end{align}
We aim to apply \cite[Proposition 5]{Nickl.2007} to $u\coloneqq 1/n\sum_{i=1}^n u_i$ in order to get existence and measurability of $\hat{G}_n$ in \eqref{eq:empiricalrisminimizer}.

Per assumption, the metric space $(\pG,\|\cdot\|)$ is compact. We show that $u_i$ from \eqref{eq:ui} is measurable in the first component and continuous in the second component for all $i=1,\dots,n$. Then \cite[Proposition 5]{Nickl.2007} shows the claim. Consider an arbitrary $G\in\pG$. We show that $u_i(.\,,G)$ is $(\Sigma,\B_{\IR})$-measurable, where $\B_{\IR}$ is the Borel $\sigma$-algebra in $\IR$.

The RVs $x_i$ are $(\Sigma,\B_{\X})$-measurable by definition. The maps $G$ and $\G_0$ are assumed to be $(\B_{\X},\B_{\Y})$-measurable. Furthermore, because of their continuity, the scalar product $\langle .\,, .\rangle_{\Y}$ is $(\B_{\Y},\B_{\IR})$-measurable in both components and also the norm $\lVert\, .\, \rVert_{\Y}$ is $(\B_{\Y},\B_{\IR})$-measurable. Therefore, since the composition of measurable functions is measurable, the 
latter
two summands in \eqref{eq:ui} are $(\Sigma,\B_{\IR})$-measurable.

Proceeding with the first summand, the RVs $Z_j$ are $(\Sigma,\B_{\IR})$-measurable
by definition for all $j\in\IN$. Therefore the products $\langle\,.\,,\psi_j\rangle_{\Y} Z_j$ are $(\Sigma\otimes\B_{\Y},\B_{\IR})$-measurable for all $j\in\IN$. Thus $\eps_i$ is, as the pointwise limit, $(\Sigma\otimes\B_{\Y},\B_{\IR})$-measurable. Then, as the composition of measurable functions, the first summand in \eqref{eq:ui} is $(\Sigma,\B_{\IR})$-measurable. Therefore $u_i(\,.\,, G)$, $i=1,\dots,n$, and thus $u(\,.\,,G)$ is $(\Sigma,\B_{\IR})$-measurable for all $G\in\pG$.

We proceed and show that $u(\omega,\,.\,)$ is continuous w.r.t.~$\|\cdot\|$. Therefore choose $G,G'\in\pG$ and $\omega\in\Omega$. Then it holds for $i=1,\dots,n$ and $x_i=x_i(\omega)$
\begin{align}
\label{eq:Gnhatcontinuity1}
    \left|u_i(\omega,G)-u_i(\omega,G')\right|&\leq 2\sigma\left|\eps_i\left(\omega,G(x_i)-G'(x_i)\right)\right|+2\left|\langle G_0(x_i),G(x_i)-G'(x_i)\rangle_{\Y}\right|\notag\\
    &\qquad+\left|\lVert G(x_i)\rVert_{\Y}-\lVert G'(x_i)\rVert_{\Y}\right|\,\left|\lVert G(x_i)\rVert_{\Y}+\lVert G'(x_i)\rVert_{\Y}\right|\notag\\
    &\leq 2\sigma\left|\eps_i\left(\omega,G(x_i)-G'(x_i)\right)\right|\notag\\
    &\quad+\sqrt n\left(2\lVert \G_0(x_i)\rVert_{\Y}+\lVert G(x_i)\rVert_{\Y}+\lVert G'(x_i)\rVert_{\Y}\right)\,\lVert G-G'\rVert_n\notag\\
    &\leq 2\sigma\left|\eps_i\left(\omega,G(x_i)-G'(x_i)\right)\right|\notag\\
    &\quad+C \sqrt n\left(2\lVert \G_0(x_i)\rVert_{\Y}+\lVert G(x_i)\rVert_{\Y}+\lVert G'(x_i)\rVert_{\Y}\right)\,\lVert G-G'\rVert,
\end{align}
where we used that $\|\cdot\|_n\leq C \|\cdot\|$ at the last inequality.
Furthermore, \cite[page 40 and Proposition 2.3.7(a)]{Gine.op.2016} yields that the white noise processes $\eps_i$ are a.s.\ sample $d_{\eps_i}$-continuous w.r.t.\ their intrinsic pseudometrics $d_{\eps_i}:\Y\times\Y\to\IR$, $d_{\eps_i}(y,y')=\IE[\langle\eps_i,y-y'\rangle_{\Y}^2]^{1/2}=\lVert y-y'\rVert_{\Y}$.
Since for all $G$, $G'\in\pG$ we have
\begin{align*}
\lVert G(x_i)-G'(x_i)\rVert_\Y\leq \sqrt{n}\|G-G'\|_n\leq C\sqrt{n}\lVert G-G'\rVert,
\end{align*}
the white noise processes are also a.s.~sample $d$-continuous, where $d$ is the metric induced by $\lVert\cdot\rVert$.

Together with \eqref{eq:Gnhatcontinuity1} this shows that there exists a null-set $\Omega_0\subset\Omega$ such that for all $\omega\in\Omega\backslash\Omega_0$, $u(\omega,\, .)$ is continuous w.r.t.\ $\|\cdot\|$. Now we choose versions $\tilde{x}_i$ and $\tilde{\eps}_i$ s.t.~$u(\omega,\,.)=0$ for all $\omega\in\Omega_0$. Then
$G\mapsto u(\omega,G):(\pG,\norm{\cdot})\to\R$
is continuous 
for all $\omega\in\Omega$. Applying \cite[Proposition 5]{Nickl.2007} gives an $(\Sigma,\B_{\pG})$-measurable MLSE $\hat{G}_n$ in \eqref{eq:empiricalrisminimizer} with the desired minimization property. 

\textbf{Sub-Gaussian case.} For $i=1,\dots,n$, consider the functions $u_i$ from \eqref{eq:ui}. 
The measurability of $u_i(\,.\,,G)$, $i=1,\dots,n$ follows from the measurability of the scalar product $\langle.\,,.\rangle_{\Y}$, the norm $\|\cdot\|_{\Y}$, $\G_0$ and all $G\in\pG$. Note that in contrast to the white noise case, $\eps_i$ are RVs in $\Y$ for all $i=1,\dots,n$ and therefore measurable without any further investigation. Also, since $\eps_i(\omega)\in\Y$ for all $\omega$, Cauchy-Schwarz immediately shows that $u_i(\omega,\,.)$ and therefore $u(\omega,\,.)$ is continuous w.r.t.\ $\|\cdot\|$ for all $\omega\in\Omega\backslash\Omega^0$. Choosing versions $\tilde{x}_i$ and $\tilde{\eps}_i$ and applying \cite[Proposition 5]{Nickl.2007} as above gives the existence of an $(\Sigma,\B_{\pG})$-measurable LSE $\hat{G}_n$ in \eqref{eq:empiricalrisminimizer} and therefore finishes the proof.
\end{proof}
\subsubsection{Concentration Inequality for $\hat G_n$}\label{app:proofempirical}
\begin{proof}[Proof of Theorem \ref{th:empiricalerror} \ref{item:conditionaltail}]

  The proof follows ideas developed in \cite[Section 10.3]{vanGeer.2009} as well as \cite{NW20}, where generic chaining bounds from \cite{Dirksen.2015} were used to bound the relevant empirical processes appearing below.

\paragraph*{Slicing argument.} Recall the definition \eqref{eq:empiricalnorm} of the empirical norm. For $ R^2\geq 2\lVert G^*-G_0\rVert_n^2$, we have
\begin{align}\label{eq:proofempirical1}
    \IP_{G_0}^{\px}\bigl(\lVert \hatgn-G_0\rVert_n^2\geq R^2\bigr)\leq \IP_{G_0}^{\px}\bigl(2\bigl|\lVert\hatgn-G_0\rVert_n^2-\lVert G^*-G_0\rVert_n^2\bigr|\geq  R^2\bigr).
\end{align}
As a union of disjoint events, it holds
\begin{align}\label{eq:proofempirical2}
    &\IP_{G_0}^{\px}\bigl(2\bigl|\lVert\hatgn-G_0\rVert_n^2-\lVert G^*-G_0\rVert_n^2\bigr|\geq  R^2\bigr)\notag\\
    =\,\,&\sum_{s=0}^{\infty}\IP_{G_0}^{\px}\biggl(2^{2s} R^2\leq 2\bigl|\lVert\hatgn-G_0\rVert_n^2-\lVert \gstar-G_0\rVert_n^2\bigr|< 2^{2s+2} R^2\biggr).
\end{align}
Applying the basic inequality (Lemma \ref{lem:basicineq}) and defining the empirical process $(X_G:G\in\pG)$ indexed by the operator class $\pG$ as
\begin{align*}
    X_G=\frac{1}{n}\sum_{i=1}^n\langle\eps_i,G(x_i)-\gstar(x_i)\rangle_{\Y},
\end{align*}
gives
\begin{align}\label{eq:proofempirical3}
   &\sum_{s=0}^{\infty}\IP_{G_0}^{\px}\biggl(2^{2s} R^2\leq 2\bigl|\lVert\hatgn-G_0\rVert_n^2-\lVert \gstar-G_0\rVert_n^2\bigr|< 2^{2s+2} R^2\biggr)\notag\\
    \leq\,\,& \sum_{s=0}^{\infty}\IP_{G_0}^{\px}\biggl(\bigl|X_{\hatgn}\bigr|\geq \frac{2^{2s-2}R^2}{\sigma},\,\,\bigl|\lVert\hatgn-G_0\rVert_n^2-\lVert \gstar-G_0\rVert_n^2\bigr|< 2^{2s+1} R^2\biggr)\notag\\
    \leq\,\,&\sum_{s=0}^{\infty}\IP_{G_0}^{\px}\biggl(\sup_{G\in\pG_n^*(2^{s+3/2} R)}\bigl|X_{G}\bigr|\geq \frac{2^{2s-2}R^2}{\sigma}\biggr).
\end{align}
In \eqref{eq:proofempirical3}, we additionally used that if $2\lVert \gstar-G_0\rVert_n^2\leq R
^2$ and
\begin{align*}
    \bigl|\lVert\hatgn-G_0\rVert_n^2-\lVert \gstar-G_0\rVert_n^2\bigr|<2^{2s+1} R^2
\end{align*}
then $\lVert\hatgn-G_0\rVert_n^2\leq 2^{2s+1} R^2+ R^2/2$ and thus
\begin{align*}
    \lVert \hatgn-\gstar\rVert_n^2\leq 2\bigl(\lVert\hatgn-G_0\rVert_n^2+\lVert \gstar-G_0\rVert_n^2\bigr)
    \leq2^{2s+2} R^2+2 R^2\leq 2^{2s+3} R^2.
\end{align*}
\paragraph*{Concentration inequality for each slice.} We wish to apply Lemma \ref{lem:chaining1} to bound the probabilities in \eqref{eq:proofempirical3}. Let $C_{\rm Ch}$ be the generic constant from this lemma
and let $\delta_n$ satisfy \eqref{eq:deltan}. Then, due to $\delta\mapsto \Psi_n(\delta)/\delta^2$ being non-increasing, \eqref{eq:deltan} gives for all $R\ge \delta_n$ and $s\in\IN_0$
\begin{equation*}
  \sqrt{n}(2^{2s+3}R^2)\ge 32C_{\rm Ch}\sigma\Psi_n(2^{s+3/2}R)
\end{equation*}
so that with $\Theta:=\pgnstar(2^{s+3/2} R)$
\begin{equation}\label{eq:deltanchainingcheck}
  \sqrt{n}\frac{2^{2s-2}R^2}{\sigma}\ge C_{\rm Ch}\Psi_n(2^{s+3/2}R)
  \ge C_{\rm Ch} J(\Theta,\norm[n]{\cdot}).
\end{equation}
With 
$h_G:=G-\gstar$ and
$\HH:=\set{h_G}{G\in\Theta}$
we have $J(\Theta,\norm[n]{\cdot})=J(\HH,\norm[n]{\cdot})$ and
thus \eqref{eq:deltanchainingcheck} verifies assumption \eqref{eq:deltanchaining} of Lemma \ref{lem:chaining1} for $\delta=2^{2s-2}R^2/\sigma$. 

Furthermore, since $\Theta=\pgnstar(2^{s+3/2} R)\subset\pG$ for $s\in\IN_0$, $R>0$ and $(\pG,\|\cdot\|_{n})$ is compact, the space $(\Theta,\norm[n]{\cdot})$ is separable, which verifies the last assumption of Lemma \ref{lem:chaining1}. 
Applying this lemma  with $U=2^{s+3/2} R$ shows that the (uncountable) suprema in \eqref{eq:proofempirical3} are measurable
and that for all $R\geq \delta_n$,

\begin{align}\label{eq:proofempirical4}
\sum_{s=0}^{\infty}\IP_{G_0}^{\px}\biggl(\sup_{G\in\pG_n^*(2^{s+3/2} R)}\bigl|X_{G}\bigr|\geq \frac{2^{2s-2} R^2}{\sigma}\biggr)
  &\leq\sum_{s=0}^{\infty}\exp\biggl(-\frac{8n2^{4s-4} R^4}{\CCh^2\sigma^2 2^{2s+3} R^2}\biggr)\notag\\
    &\leq \sum_{s=0}^{\infty}\exp\biggl(-\frac{2^{2s-4}n R^2}{\CCh^2\sigma^2}\biggr)\notag\\
  &\leq \sum_{s=0}^{\infty} 2^{-2s}\exp\biggl(-\frac{n R^2}{16\CCh^2\sigma^2}\biggr)\notag\\
  &< 2 \exp\biggl(-\frac{n R^2}{16\CCh^2\sigma^2}\biggr).
\end{align}
In \eqref{eq:proofempirical4} we additionally used $\exp(-xy)\leq\exp(-x)/y$, which holds for all $x,y\geq1$, i.e.~for $R\geq 4\CCh\sigma/\sqrt{n}$. Combining \eqref{eq:proofempirical1}--\eqref{eq:proofempirical3} and \eqref{eq:proofempirical4} gives \eqref{eq:empiricalconcentration} and therefore shows the claim.
\end{proof}
\subsection{Proof of Theorem \ref{th:samplecomplexitywhitnoise}}\label{app:proofsamplecomplexity}
We use the concentration inequality from Theorem \ref{th:empiricalerror} for the empirical error and combine it with a key concentration result for the empirical norm around the $L^2(\gamma)$-norm, proved in Lemma \ref{lem:variance}. For any $R>0$,
\begin{align}\label{eq:cleversplitrgn}
\IP_{G_0}\left(\lVert\hatgn-G_0\rVert_\ltg\geq R\right)& \leq \IP_{G_0}\left(\lVert\hatgn-G_0\rVert_n\geq \frac{R}{2}\right)\notag\\
&\quad+\IP_{G_0}\left(\lVert\hatgn-G_0\rVert_\ltg\geq R,\,\,\lVert\hatgn-G_0\rVert_\ltg\geq2\lVert\hatgn-G_0\rVert_n\right).
\end{align}
Now suppose that $R$ satisfies
\[R \ge 2\max\left\{\delta_n,4\tilde \delta_n,\sqrt{2}\lVert G^*-G_0\rVert_{\infty,\supp(\gamma)},\frac{4C_{ch}\sigma}{\sqrt{n}},\frac{9\Finf}{\sqrt{n}}\right\}.\]
This is implied by (\ref{eq:Rinequality}) for an appropriate choice of $C$. In particular, $R\geq 2\sqrt{2}\lVert G^*-G_0\rVert_n$ for all 
$\px\in\X^n$. Since Theorem \ref{th:empiricalerror} holds for $\gamma^n$-almost every $\px\in \mathcal X^n$, taking expectations over $\px\sim \gamma^n$ gives
\begin{align}\label{eq:proofl2concentration1}
    \IP_{G_0}\left(\lVert\hatgn-G_0\rVert_n\geq \frac{R}{2}\right)\leq 2\exp\left(-\frac{nR^2}{64\CCh^2\sigma^2}\right).
\end{align}
Furthermore, Lemma \ref{lem:variance} gives
\begin{align}
\label{eq:proofl2concentration2}
   \IP_{G_0}\left(\lVert\hatgn-G_0\rVert_\ltg\geq R,\,\,\lVert\hatgn-G_0\rVert_\ltg\geq2\lVert\hatgn-G_0\rVert_n\right)\leq 2\exp\left(-\frac{nR^2}{320\Finf^2}\right),
\end{align}
since we have assumed $R\geq \max\{8\tildedeltan,18\Finf/\sqrt{n}\}$. Combining~\eqref{eq:proofl2concentration1} and \eqref{eq:proofl2concentration2} shows \eqref{eq:L2concentration} for some $C$.

\subsection{Proof of Corollary \ref{cor:MSE}}\label{app:samplecomplexityexpectation}
It remains to show \eqref{eq:WNconvergence}.
We use
\eqref{eq:cleversplitrgn} to estimate
\begin{align}
\label{eq:whitenoiseproof1}
&\IE_{\G_0}\bigl[\lVert\hatgn-\G_0\rVert_\ltg^2\bigr]\notag\\
=\,\,&\int_0^{\infty}\IP_{G_0}\left(\lVert\hatgn-\G_0\rVert_\ltg\geq  \sqrt{R}\right)\dd  R\notag\\ 
\leq\,\,& \int_0^\infty\int_{\X^n} \IP_{G_0}^{\px}\left(\lVert\hatgn-G_0\rVert_n\geq \frac{\sqrt R}{2}\right)\dd\gamma(\px)\dd R\notag\\
&\;\;+\int_0^\infty \IP_{G_0}\left(\lVert\hatgn-G_0\rVert_\ltg\geq \sqrt{R},\,\,\lVert\hatgn-G_0\rVert_\ltg\geq2\lVert\hatgn-G_0\rVert_n\right)\dd R.
\end{align}
For $G^*\in\pG$ set $R_1=2\max\{\delta_n,\sqrt 2 \lVert G^*-G_0\rVert_n,\frac{4\CCh\sigma}{\sqrt{n}}\}$, where $\CCh$ is from Lemma \ref{lem:chaining1}. Then Theorem \ref{th:empiricalerror} gives
\begin{align}
    \label{eq:Rgnestimate1WN}
&\int_0^\infty\int_{\X^n} \IP_{G_0}^{\px}\left(\lVert\hatgn-G_0\rVert_n\geq \frac{\sqrt R}{2}\right)\dd\gamma(x)\dd R\notag\\
\leq\,\,&\int_{\X^n}R_1^2\dd\gamma(\px)+2\int_{R_1^2}^\infty\exp\left(-\frac{nR}{64\CCh^2\sigma^2}\right)\dd R\notag\\
\leq \,\,&4\delta_n^2+8\lVert G^*-G_0\rVert_\ltg^2+\frac{16\CCh^2\sigma^2}{n}+2\int_{R_1^2}^\infty\exp\left(-\frac{nR}{64\CCh^2\sigma^2}\right)\dd R\notag\\
\leq\,\,&4\delta_n^2+8\lVert G^*-G_0\rVert_\ltg^2+\frac{80\CCh^2\sigma^2}{n}.
\end{align}
Lemma \ref{lem:variance} gives with $R_2=\max\{8\tildedeltan,18\Finf/\sqrt n\}$
\begin{align}
  \label{eq:Rgnestimate2WN}  
&\int_0^\infty \IP_{G_0}\left(\lVert\hatgn-G_0\rVert_\ltg\geq \sqrt{R},\,\,\lVert\hatgn-G_0\rVert_\ltg\geq2\lVert\hatgn-G_0\rVert_n\right)\dd R\notag\\
\leq\,\,&R_2^2+2\int_{R_2^2}^\infty\exp\left(-\frac{nR}{320\Finf^2}\right)\dd R\notag\\
\leq\,\,& 64\tildedeltan^2+\frac{964\Finf^2}{n}.
\end{align}
Combining \eqref{eq:whitenoiseproof1}--\eqref{eq:Rgnestimate2WN} and 
taking an infimum over $G^*\in\pG$
shows \eqref{eq:WNconvergence} for some $C$ and thus finishes the proof of Corollary \ref{cor:MSE}.

\subsection{Proof of Corollary \ref{cor:explicitentropy}}\label{app:prooconcreteentropyWN}
We construct sequences $(\delta_n)_{n\in\IN}$ and $(\tilde{\delta}_n)_{n\in\IN}$ satisfying \eqref{eq:deltandeltanhat} and balance the approximation term and the entropy terms in \eqref{eq:WNconvergence} by choosing $N(n)$ appropriately.
\paragraph*{Proof of \ref{item:concrete1}:} Choosing $\tilde{\delta}_n^2\simeq n^{-2/(2+\alpha)}$ immediately shows the second part of \eqref{eq:deltandeltanhat}.
For $J(\delta)$ from \eqref{eq:Jintegral} it holds 
\[J(\delta)\leq\int_0^{\delta} \sqrt{H(\rho,N)}\dd\rho\lesssim \delta^{1-\alpha/2}\eqqcolon \Psi_n(\delta).\]
Hence $\delta_n^2\simeq n^{-2/(2+\alpha)}$ satisfies the first part of \eqref{eq:deltandeltanhat}. Therefore Corollary \ref{cor:MSE} shows 
\[  \IE_{\G_0}\,\Big[\lVert \hat{G}_n-\G_0\rVert_{L^2(\gamma)}^2\Big]\lesssim N^{-\beta}+n^{-\frac{2}{2+\alpha}}.\]
Since the entropy term is independent of $N$, the limit $N\to\infty$ gives the claim.

\paragraph*{Proof of \ref{item:concrete2}:}
Choosing $\tilde{\delta}_n^2\simeq N(1+\log(n))/n$ shows in particular $\log(n)\gtrsim \log(\tilde{\delta}_n^{-1})$, which in turn shows that $\tilde{\delta}_n$ satisfies the second part of \eqref{eq:deltandeltanhat}.
For $J(\delta)$ from \eqref{eq:Jintegral} it holds 
\[J(\delta)\lesssim\int_0^{\delta} \sqrt{H(\rho,N)}\dd\rho\lesssim\sqrt{N}\delta\left(1+\log(\delta^{-1})\right) \eqqcolon \Psi_n(\delta).\]
Hence $\delta_n^2\simeq N(1+\log(n))/n$ satisfies the first part of \eqref{eq:deltandeltanhat}. Therefore Corollary \ref{cor:MSE} shows 
\[  \IE_{\G_0}\,\Big[\lVert \hat{G}_n-\G_0\rVert_{L^2(\gamma)}^2\Big]\lesssim N^{-\beta}+\frac{N(1+\log(n))}{n}.\] 
Choosing $N(n)=\lceil n^{\frac{1}{\beta+1}}\rceil$ and using $\log(n)\leq n^{\tau}/\tau$ for all $\tau>0$ and $n\geq 1$ gives the claim.

\subsection{Proof of Theorem \ref{th:samplecomplexitysub-Gaussiannoise}}\label{app:subgaussnoJ}
The concentration inequality \eqref{eq:L2concentrationSBGN} follows immediately from \eqref{eq:cleversplitrgn}, \eqref{eq:proofl2concentration2} (which holds for sub-Gaussian noise) and \eqref{eq:concentartionvarianceSBGN} from Lemma \ref{lem:sub-Gaussianepsilon}.
Application of Lemma \ref{lem:sub-Gaussianepsilon} yields
\begin{align*}
&\int_0^\infty\int_{\X^n} \IP_{G_0}^{\px}\left(\lVert\hatgn-G_0\rVert_n\geq \frac{\sqrt \delta}{2}\right)\dd\gamma(\px)\dd \delta\notag\\
\qquad&\leq 4\delta_n^2+8\lVert G^*-G_0\rVert_\ltg^2+4\int_{0}^\infty\exp\left(-\frac{n\delta^2}{16C_1^2\sigma^2(1+F_{\infty}^2)}\right)\dd \delta\notag\\
\qquad&\leq 4\delta_n^2+8\lVert G^*-G_0\rVert_\ltg^2+16C_1\sigma(1+\Finf)\sqrt\frac{\pi}{n}.
\end{align*}
Using \eqref{eq:whitenoiseproof1} and \eqref{eq:Rgnestimate2WN}, and minimizing over $G^*\in\pG$ gives the bound on 
the mean-squared error \eqref{eq:SBGNconvergence2} for some $C_2$.

\section{Neural Network Theory}
In this section we recap elementary operations and approximation theory for neural networks, based on \cite{Petersen.2018}. For a NN $f:\IR^{p_0}\to\IR^{p_{L+1}}$ (see Definition \ref{def:NN}) we denote by ${\rm size}_{{\rm in}}(f)$ and ${\rm size}_{{\rm out}}(f)$ the number of nonzero weights and biases of the first and last layer respectively.

\subsection{Operations on Neural Networks}\label{app:NNoperations}
In this subsection, we recap elementary operations on NNs. We start with the parallelization of two NNs following 
\cite[Section 2.2.1]{Opschoor.2022}.
\begin{definition}[Parallelization]
\label{def:NNparalleliazion}
Let $q\in\IN$, $q\geq2$ and $\sigma\in\{\sigma_1,\sigma_q\}$. Let $f$ and $g$ be two $\sigma$-NNs realizing the functions $f$ and $g$ with the same depth $L$. Furthermore, define the input dimensions of $f$  and $g$ as $n_f$ and $n_g$ and the output dimensions as $m_f$ and $m_g$.\footnote{Using the syntax from \eqref{eq:NNdef1}--\eqref{eq:NNdef3}, it holds  $n_f=p_0(f)$, $n_g=p_0(g)$, $m_f=p_{L+1}(f)$ and $m_g=p_{L+1}(g)$.} Then there exists a $\sigma$-NN $(f,g)$, called the \emph{parallelization} of $f$ and $g$, which simultaneously realizes $f$ and $g$, i.e.
\begin{equation*}
    \label{eq:NNparallel}
(f,g): \IR^{n_f}\times\IR^{n_g}\to \IR^{m_f}\times\IR^{m_g}: (\pmb{x},\tilde{\pmb{x}})\mapsto (f(\pmb{x}),g(\tilde{\pmb{x}})). 
\end{equation*}
It holds
\begin{align}
    {\rm size}\left((f,g)\right)&={\rm size}(f)+{\rm size}(g),\label{eq:parallelization2NNsize}\\
    {\rm depth}\left((f,g)\right)&={\rm depth}(f)={\rm depth}(g)\label{eq:parallelization2NNdepth},\\
    {\rm width}\left((f,g)\right)&={\rm width}(f)+{\rm width}(g)\label{eq:parallelization2NNwidth},\\
    {\rm mpar}\left((f,g)\right)&=\max\left\{{\rm mpar}(f),{\rm mpar}(g)\right\}\label{eq:parallelization2NNmaximumweights},\\
        {\rm mran}_{\Omega}\left((f,g)\right)^2&={\rm mran}_{\Omega}(f)^2+{\rm mran}_{\Omega}(g)^2.\label{eq:parallelization2NNmaximuml2}
\end{align}
\end{definition}
Let $N\in\IN$, $N\geq3$. We extend Definition \ref{def:NNparalleliazion} to parallelize $N$ $\sigma$-neural networks $f_i$, $i=1,\dots N$ with equal depth and denote the resulting $\sigma$-NN as $(\{f_i\}_{i=1}^{N})$. It holds that
\begin{align}
{\rm size}\left(\left(\left\{f_i\right\}_{i=1}^{N}\right)\right)&=\sum_{i=1}^N{\rm size}(f_i)\label{eq:sizemoreparallelizationNN},\\
{\rm size}_{{\rm in}}\left(\left(\left\{f_i\right\}_{i=1}^{N}\right)\right)&=\sum_{i=1}^N{\rm size}_{{\rm in}}(f_i)\label{eq:sizeinmoreparallelizationNN},\\
{\rm size}_{{\rm out}}\left(\left(\left\{f_i\right\}_{i=1}^{N}\right)\right)&=\sum_{i=1}^N{\rm size}_{{\rm out}}(f_i)\label{eq:sizeoutmoreparallelizationNN},\\
{\rm depth}\left(\left(\left\{f_i\right\}_{i=1}^{N}\right)\right)&={\rm depth}(f_1)\label{eq:lengthmoreparallelizationNN},\\
{\rm width}\left(\left(\left\{f_i\right\}_{i=1}^{N}\right)\right)&=\sum_{i=1}^{N}{\rm width}(f_i)\label{eq:widthmoreparallelizationNN},\\
{\rm mpar}\left(\left(\left\{f_i\right\}_{i=1}^{N}\right)\right)&=\max_{i=1,\dots,N}{\rm mpar}(f_i)\label{eq:maxweightsmoreparallelizationNN},\\
{\rm mran}_{\Omega}\left(\left(\left\{f_i\right\}_{i=1}^{N}\right)\right)^2&=\sum_{i=1}^N {\rm mran}_{\Omega}(f_i)^2\label{eq:maxl2moreparallelizationNN}.
\end{align}

Next, we recall the concatenation of NNs, \cite[Definition 2.2]{Petersen.2018}.
\begin{lemma}[Concatenation]
\label{def:NNconcatenation}
Let $q\in\IN$, $q\geq2$ and $\sigma\in\{\sigma_1,\sigma_q\}$. Let $f$ and $g$ be two $\sigma$-NNs. Furthermore, let the output dimension $m_g$ of $g$ equal the input dimension $n_f$ of $f$. Then there exists a $\sigma$-NN $f\sbullet g$ realizing the composition $f\circ g:\,x\mapsto f(g(x))$ of the functions $f$ and $g$. It holds 
\begin{align*}
   {\rm depth}\left(f\sbullet g\right)&={\rm depth}(f)+{\rm depth}(g),\\
   {\rm width}\left(f\sbullet g\right)&= \max\{{\rm width}(f),{\rm width}(g)\},\\
    {\rm mran}_{\Omega}\left(f\sbullet g\right)&={\rm mran}_{g(\Omega)}\left(f\right).\
\end{align*}
\end{lemma}

There is no simple control over the size and the weight bound of the concatenation $f\sbullet g$ in Definition \ref{def:NNconcatenation}.
The reason is that the network $f\sbullet g$ multiplies network weights and biases of the NNs $f$ and $g$ at layer $l={\rm depth}(g)+1$ (for details see \cite[Definition 2.2]{Petersen.2018}). In the following we use \emph{sparse concatenation} to get control over the size and the weights. We first introduce the realization of the identity map and separate the analysis for the $\sigma_1$- and $\sigma_q$-case. The following lemma is proven in \cite[Remark 2.4]{Petersen.2018}.
\begin{lemma}[$\sigma_1$-realization of identity map]
\label{def:Identityemulationsigma1}
 Let $d\in\IN$ and $L\in\IN$. Then there exists a \emph{$\sigma_1$-identity network} $\Id_{\IR^d}$ of depth $L$, which exactly realizes the identity map $\Id_{\IR^d}:\IR^d\to\IR^d,~\pmb{x}\mapsto\pmb{x}$. It holds
\begin{align}
    {\rm size}\left(\Id_{\IR^d}\right)&\leq 2d(L+1),\label{eq:sigma1identityemulationsize} \\
    {\rm width}\left(\Id_{\IR^d}\right)&\leq 2d,\label{eq:sigma1identityemulationwidth}\\
    {\rm mpar}\left(\Id_{\IR^d}\right)&\leq 1 \label{eq:sigma1identityemulationweightbound}.
\end{align}
\end{lemma}

We proceed with the analogous result for the RePU activation function. The following lemma follows from the construction in \cite[Theorem 2.5 (2)]{Li.2020} with a parallelization argument.
\begin{lemma}[$\sigma_q$-realization of identity map]
\label{def:Identityemulationsigmaq}
 Let $q\in\IN$ with $q\geq 2$. Further let $d\in\IN$ and $L\in\IN$ be arbitrary. Then there exists a $\sigma_q$-NN $\Id_{\IR^d}$ of depth $L$, which exactly realizes the identity map $\Id_{\IR^d}$. It holds
\begin{align}
{\rm size}(\Id_{\IR^d})&\leq C_qdL,\label{eq:sigmarqemulationidentitysize}\\
{\rm width}(\Id_{\IR^d})&\leq C_qd,\label{eq:sigmarqemulationidentitywidth}\\
{\rm mpar}(\Id_{\IR^d})&\leq C_q,\label{eq:sigmarqemulationidentitymaximumweights}
\end{align}
where $C_q$ is independent of $d$ (but does depend on $q$).
\end{lemma}

We next define the sparse concatenation of two NNs.
\begin{definition}[$\sigma_1$-sparse concatenation, {\cite[Definition 2.5]{Petersen.2018}}]
\label{def:sparseconcatenationNN}
Let $f$ and $g$ be two $\sigma_1$-NNs. Furthermore, let the output dimension $m_g$ of $g$ equal the input dimension $n_f$ of $f$. Then there exists a $\sigma_1$-NN $f\circ g$ with
\begin{equation*}
    \label{eq:sparseconncatentaionNN}
    f\circ g\coloneqq f\sbullet\Id_{\IR_{m_g}}\sbullet g
\end{equation*}
realizing the composition $f\circ g:\,x\mapsto f(g(x))$ of the functions $f$ and $g$.\footnote{The symbol $\circ$ does mean either the functional concatenation of $f$ and $g$ or the sparse concatenation of the NN $f$ and $g$ (which realizes the function $f\circ g$).} It holds 
\begin{align}
{\rm size}(f\circ g)&\leq {\rm size}(g)+{\rm size}_{{\rm out}}(g)+{\rm size}_{{\rm in}}(f)+{\rm size}(f)\leq 2{\rm size}(f)+2{\rm size}(g),\label{eq:sparseconcatsize}\\
{\rm size}_{{\rm in}}(f\circ g)&\leq\begin{cases}
    {\rm size}_{{\rm in}}(g)\quad & {\rm depth}(g)\geq 1,\\
    2{\rm size}_{{\rm in}}(g)\quad & {\rm depth}(g)=0, 
\end{cases}\label{eq:sparseconcatsizein}\\
{\rm size}_{{\rm out}}(f\circ g)&\leq \begin{cases}
    {\rm size}_{{\rm out}}(f)\quad & {\rm depth}(f)\geq 1,\\
    2{\rm size}_{{\rm out}}(f)\quad & {\rm depth}(f)=0, 
\end{cases}\label{eq:sparseconcatsizeout}\\
{\rm depth}(f\circ g)&={\rm depth}(f)+{\rm depth}(g)+1,\label{eq:sparseconcatdepth}\\
{\rm width}(f\circ g)&\leq 2\max\{{\rm width}(f),{\rm width}(g)\},\label{eq:sparseconcatwidth}\\
{\rm mpar}(f\circ g)&\leq \max\left\{{\rm mpar}(f),{\rm mpar}(g)\right\},\label{eq:sparseconcatweightbound}\\
 {\rm mran}_{\Omega}\left(f\circ g\right)&=B_{g(\Omega)}\left(f\right).\label{eq:sparseconcatl2bound}
\end{align}
\end{definition}
\begin{proof}
The bounds in \eqref{eq:sparseconcatdepth}, \eqref{eq:sparseconcatwidth} and \eqref{eq:sparseconcatl2bound} follow from Definition \ref{def:Identityemulationsigma1} with the NN calculus from Definition \ref{def:NNconcatenation}. The bounds on the sizes in \eqref{eq:sparseconcatsize}--\eqref{eq:sparseconcatsizeout} and the weight and bias bound \eqref{eq:sparseconcatweightbound} follow from the specific structure of the $\sigma_1$-identity network, see \cite[Remark 2.6]{Petersen.2018}.\end{proof}

We proceed with the sparse concatenation of $\sigma_q$-NNs.
\begin{definition}[$\sigma_q$-sparse concatenation]
\label{def:sigqsparseconcatenationNN} Let $q\in\IN$ with $q\geq 2$. Let $f$ and $g$ be two $\sigma_q$-NNs. Furthermore, let the output dimension $m_g$ of $g$ equal the input dimension $n_f$ of $f$. Then there exists a $\sigma_q$-NN $f\circ g$ with
\begin{equation*}
    \label{eq:sigqsparseconncatentaionNN}
    f\circ g\coloneqq f\sbullet\Id_{\IR_{m_g}}\sbullet g
\end{equation*}
realizing the composition $f\circ g:\,x\mapsto f(g(x))$ of the functions $f$ and $g$. It holds 
\begin{align}
{\rm size}(f\circ g)&\leq {\rm size}(g)+(C_q-1){\rm size}_{{\rm out}}(g)+(C_q-1){\rm size}_{{\rm in}}(f)+{\rm size}(f)\\
&\leq C_q{\rm size}(f)+C_q{\rm size}(g),\label{eq:sigqsparseconcatsize}\\
{\rm size}_{{\rm in}}(f\circ g)&\leq\begin{cases}
    {\rm size}_{{\rm in}}(g)\quad & {\rm depth}(g)\geq 1,\\
    C_q{\rm size}_{{\rm in}}(g)\quad & {\rm depth}(g)=0, 
\end{cases}\label{eq:sigqsparseconcatsizein}\\
{\rm size}_{{\rm out}}(f\circ g)&\leq \begin{cases}
    {\rm size}_{{\rm out}}(f)\quad & {\rm depth}(f)\geq 1,\\
    C_q{\rm size}_{{\rm out}}(f)\quad & {\rm depth}(f)=0, 
\end{cases}\label{eq:sigqsparseconcatsizeout}\\
{\rm depth}(f\circ g)&={\rm depth}(f)+{\rm depth}(g)+1,\label{eq:sigqsparseconcatdepth}\\
{\rm width}(f\circ g)&\leq C_q\max\{{\rm width}(f),{\rm width}(g)\},\label{eq:sigqsparseconcatwidth}\\
{\rm mpar}(f\circ g)&\leq C_q\max\left\{{\rm mpar}(f),{\rm mpar}(g)\right\},\label{eq:sigqsparseconcatweightbound}\\
 {\rm mran}_{\Omega}\left(f\circ g\right)&=B_{g(\Omega)}\left(f\right)\label{eq:sigqsparseconcatl2bound}
\end{align}
with a constant $C_q>1$ depending only on $q$.
\end{definition}
\begin{proof} 
The bounds in \eqref{eq:sigqsparseconcatdepth}, \eqref{eq:sigqsparseconcatwidth} and \eqref{eq:sigqsparseconcatl2bound} follow from Definition \ref{def:Identityemulationsigmaq} with the NN calculus from Definition \ref{def:NNconcatenation}. The bounds on the sizes in \eqref{eq:sigqsparseconcatsize}--\eqref{eq:sigqsparseconcatsizeout} and the weight and bias bound  in \eqref{eq:sigqsparseconcatweightbound} hold because of the specify structure of the $\sigma_q$-identity network $\Id_{\IR^n}$, see \cite[Eq.~(2.57)]{Li.2020}.\end{proof}

Definitions \ref{def:sparseconcatenationNN} and \ref{def:sigqsparseconcatenationNN} show that one can control the size, as well as the weights and biases of the concatenation of two NN by inserting one additional identity layer between the two networks.
We end this subsection by introducing summation and scalar multiplication networks. 
\begin{definition}[Summation networks] \label{def:summationnetworks}
Let $q\in\IN$, $q\geq2$, $\sigma\in\{\sigma_1,\sigma_q\}$ and $d,m\in\IN$. Then there exists a $\sigma$-NN $\Sigma_m$ such that for $x_1,\dots,x_m\in\IR^d$
\begin{align*}
    \Sigma_m(x_1,\dots,x_m)=\sum_{i=1}^mx_i,
\end{align*}
with ${\rm depth}(\Sigma_m)=0$, ${\rm width}(\Sigma_m)=md$, ${\rm size}(\Sigma_m)=md$ and ${\rm mpar}(\Sigma_m)=1$.
\end{definition}
\begin{proof} Set $\pmb{w}^1=(\mathbb{1}_d,\dots,\mathbb{1}_d)$ and $\pmb{b}^1=0$ with the $d\times d$ identity matrices $\mathbb{1}_d$.\end{proof}
\begin{definition}[Scalar multiplication networks]\label{def:scalarmultiplicatoinsimaq}
Let $q\in\IN$ with $q\geq 2$ and $\sigma\in\{\sigma_1,\sigma_q\}$. Let $\alpha\in\IR$ and $d\in\IN$. Then there exists a $\sigma$-NN $SM_{\alpha}$ with 
\begin{equation*}
    SM_{\alpha}(x)=\alpha x,\qquad x\in\IR^d.
\end{equation*}
Furthermore, there exists a constant $C_q$ only depending on $q$ such that
\begin{align}
    {\rm depth}\left(SM_{\alpha}\right)&\leq C_q\max\{1,\log(\lvert\alpha\rvert)\},\label{eq:SMsigmaqL}\\
    {\rm width}\left(SM_{\alpha}\right)&\leq C_q d,\label{eq:SMsigmaqp}\\
    {\rm size}\left(SM_{\alpha}\right)&\leq C_q d\max\{1,\log(\lvert\alpha\rvert)\},\label{eq:SMsigmaqs}\\
    {\rm mpar}\left(SM_{\alpha}\right)&\leq C_q .\label{eq:SMsigmaqM}
    \end{align}
\end{definition}
\begin{proof}
For the proof of the $\sigma_1$-case, see \cite[Lemma A.1]{Elbrachter.2021}. We prove the RePU case in the following.

Without loss of generality we can choose $\alpha>1$. For $\alpha<0$ we set $SM_{\alpha}=-SM_{-\alpha}$. For $0<\alpha<1$ we set $SM_{\alpha}=\alpha\Id_{\IR^d}$, where we directly multiply the weights and biases of the identity network (with depth $L=1$) with $\alpha$.

Therefore let $\alpha>1$.
Let $K$ be the maximum integer smaller than $\log_2(\alpha)$, and set $\tilde{\alpha}=2^{-K+1}\alpha<1$. Furthermore, set $A_1=(\Id_{\IR^d},\Id_{\IR^d})$ and $A_2=\Sigma_2$ with the one-layered identity network $\Id_{\IR^d}$ and the summation network $\Sigma_2$ from Definition \ref{def:summationnetworks}. We notice that 
\begin{align*}
    A_2\sbullet A_1 x=2x\quad\forall x\in\IR^d.
\end{align*}
Using the bounds for $\Sigma_2$ from Definition \ref{def:summationnetworks} and $\Id_{\IR^d}$ from Definition \ref{def:Identityemulationsigmaq}, we have ${\rm depth}(A_2\sbullet A_1)=1$, ${\rm width}(A_2\sbullet A_1)\leq C_q d$, ${\rm size}(A_2\sbullet A_1)\leq C_qd$ and ${\rm mpar}(A_2\sbullet A_1)\leq C_q$. Setting $A_{2k+1}=A_1$, $A_{2k+2}=A_2$ for $k=1,\dots, K$ and $A_{2K+3}=\tilde{\alpha}\Id_{\IR^d}$, we get
\begin{align*}
    A_{2K+3}\circ A_{2K+2}\sbullet A_{2K+1}\circ A_{2K}\sbullet A_{2K-1}\circ\dots..\circ A_2\sbullet A_1x=\alpha x\quad\forall x\in\IR^d.
\end{align*}
Applying the $\sigma_q$-NN calculus for concatenation (Definition \ref{def:NNconcatenation}) and sparse concatenation (Definition \ref{def:sigqsparseconcatenationNN}) gives the desired bounds for $SM_{\alpha}$.
\end{proof}

\subsection{Neural Network Approximation Theory}\label{app:NNapproximation}
In this subsection, we summarize approximation results for ReLU and RePU neural networks. In recent years, the expressivity and approximation properties of neural network architectures have been extensively studied in the literature (e.g., \cite{mhaskar1996neural,MR1819645,Yarotsky.2017,poggio2017and,Petersen.2018,Elbrachter.2021,MR4376568,de2021approximation}). However, with few exceptions (e.g., \cite{SH20,de2021approximation,Elbrachter.2021}), most of these works do not provide bounds on the size of the weights, which are crucial for controlling the entropy. Therefore, we revisit some of these arguments to provide complete proofs of our results.

We start with the well-known result that ReLU-NNs can approximate the multiplication map exponentially fast. The following proposition was shown in \cite[Proposition III.3]{Elbrachter.2021}. 
\begin{proposition}[$\sigma_1$-NN approximation of multiplication, cf.~{\cite[Propsition III.3]{Elbrachter.2021}}]
    \label{prop:approxmult2numbersReLU1}
    Let $D\in\IR$, $D\geq 1$ and $\delta\in (0,1/2)$. Then there exists a $\sigma_1$-NN $\tilde{\times}_{\delta,D}:\,[-D,D]^2\to\IR$ satisfying
    \begin{equation*}
       \sup_{x,y\in [-D,D]} \left\lvert\tilde{\times}_{\delta,D}(x,y)-xy\right\rvert\leq\delta.
    \end{equation*}
    Furthermore, there exists a constant $C$, independent of $D$ and $\delta$, such that
    \begin{align}
       {\rm depth}\left(\tilde{\times}_{\delta,D}\right)&\leq C\left(\log(D)+\log\left(\delta^{-1}\right)\right),\label{eq:NNmultiplication2unmbersdepth}\\
       {\rm width}\left(\tilde{\times}_{\delta,D}\right)&\leq 5,\label{eq:NNmultiplication2unmberswidth}\\
       {\rm size}\left(\tilde{\times}_{\delta,D}\right)&\leq C\left(\log( D)+\log\left(\delta^{-1}\right)\right)\label{eq:NNmultiplication2unmberssize},\\{\rm mpar}\left(\tilde{\times}_{\delta,D}\right)&\leq 1.\label{eq:NNmultiplication2unmbersmaximumweights}
    \end{align}
\end{proposition}

The next proposition shows that the multiplication map can be exactly realized by a $\sigma_q$-NN, which follows directly from \cite[Theorem 2.5 and Eq.~(2.59)]{Li.2020}.
\begin{proposition}[$\sigma_q$-NN approximation of multiplication]
    \label{prop:sigqapproxmult2numbersReLU}
    Let $q\in\IN$, $q\geq2$. There exists a $\sigma_q$-NN
    $\tilde{\times}:\,\IR^2\to\IR$ with ${\rm depth}(\tilde{\times})=1$ exactly realizing the multiplication of two numbers, i.e. 
    \begin{equation*}
       \tilde{\times}(x,y)=xy\quad \forall x,y\in\IR.
    \end{equation*}
\end{proposition}

We can extend the above results to the multiplication of $N$ numbers, see \cite[Proposition 2.6]{Opschoor.2022}.
\begin{proposition}[$\sigma_1$-NN multiplication of $N$ numbers]
\label{prop:sigma1multiplicationNNnumbers}
Let $N\in\IN$ with $N\geq 2$. Furthermore, let $D\in\IR$, $D\geq 1$ and $\delta\in(0,1/2)$. Then there exists a $\sigma_1$-NN $\widetilde{\prod}_{\delta,D}$: $[-D,D]^N\to \IR$ such that
\begin{equation}
    \label{eq:sigma1multinnunmberapproximation}
    \sup_{(y_i)_{i=1}^N\in [-D,D]^N}\left|\prod_{j=1}^N y_j-\widetilde{\prod}_{\delta,D}(y_1,\dots,y_N)\right|\leq\delta.
\end{equation}
Furthermore, there exists a constant $C$ independent of $N$, $\delta$ and $D$ such that
\begin{align}
 {\rm depth}\left(\widetilde{\prod}_{\delta,D}\right)&\leq C\log(N) \left(\log(N)+N\log(D)+\log\left(\delta^{-1}\right)\right),\label{sigma1multinnunmberapproximationlength}\\
 {\rm width}\left(\widetilde{\prod}_{\delta,D}\right)&\leq 5N,\label{sigma1multinnunmberapproximationwidth}\\
 {\rm size}\left(\widetilde{\prod}_{\delta,D}\right)&\leq CN\left(\log(N)+N\log(D)+\log\left(\delta^{-1}\right)\right),\label{sigma1multinnunmberapproximationsize}\\
 {\rm mpar}\left(\widetilde{\prod}_{\delta,D}\right)&\leq 1.
 \label{sigma1multinnunmberapproximationweihgtbound}
\end{align}
\end{proposition}
\label{app:proofsig1Nnumbers}
\begin{proof}
Analogous to \cite[Proposition 2.6]{Opschoor.2022} we construct $\widetilde{\prod}_{\delta,D}$ as a binary tree of $\tilde{\times}_{.,.}$-networks from Proposition \ref{prop:approxmult2numbersReLU1}. We modify the proof of \cite{Opschoor.2022} to get a construction with bounded weights.

Define $\tilde{N}\coloneqq\min\{2^k:\,k\in \IN,\,2^k\geq N\}$. We now consider the multiplication of $\tilde{N}$ numbers with $y_{N+1},\dots,y_{\tilde{N}}\coloneqq 1$. This can be implemented by a zero-layer network with 
\begin{align}
w_{i,j}^1&=
\begin{cases}
1\qquad i=j\leq N,\\
0\qquad{\rm otherwise},
\end{cases}\\
b_j^1&=
\begin{cases}
    0\qquad j\leq N,\\
    1\qquad N<j\leq\tilde{N}.
\end{cases}\label{eq:sigma1NNmultNlayersfirstlayer}
\end{align}
For $l=0,\dots,\log_2{\tilde{N}}-1$ we define the mapping $R^l$ 
\begin{align}
    \label{eq:proofNNmultapproxR}
    R^l&:[-D_l,D_l]^{\frac{\tilde{N}}{2^l}}\mapsto [-D_{l+1},D_{l+1}]^{\frac{\tilde{N}}{2^{l+1}}},\notag\\ 
    R^l(y_1^l,\dots,y_{2^{\log_2(\tilde{N})-l}}^l)&\coloneqq \left(\tilde{\times}_{\delta',D^l}(y_1^l,y_2^l),\dots,\tilde{\times}_{\delta',D^l}(y_{2^{\log_2(\tilde{N})-l}-1}^l,y_{2^{\log_2(\tilde{N})-l}}^l)\right)
\end{align}
with $\delta'\coloneqq\delta/(\tilde{N}^2D^{2\tilde{N}})$ and $D_l\coloneqq 2^{l} D^{2^{l}}$.
We now set
\begin{equation}
    \label{eq:proofNNmultapproxprodtilde}
    \widetilde{\prod}_{\delta,D}\coloneqq\ R^{\log_2(\tilde{N})-1}\circ\dots\circ R^0.
\end{equation}
Eq.~\eqref{eq:proofNNmultapproxprodtilde} shows that the map $R^l$ describes the multiplications on level $l$ of the binary tree $\widetilde{\prod}_{\delta,D}$. In order for \eqref{eq:proofNNmultapproxprodtilde} to be well-defined, we have to show that the outputs of the NN $R^ l$ are admissible inputs for the NN $R^{l+1}$. 

We therefore denote with $y_j^l$, $j=1,\dots,2^{\log_2(\tilde{N})-l}$, $l=1,\dots,\log_2(\tilde{N})-1$, the output of the network $R^{l-1}\circ\dots\circ R^0$ applied to the input $y_k^0=y_k$, $k=1,\dots,\tilde{N}$. Then we have to show $|y_j^l|\leq D_l$ for $l=0,\dots,\log_2(\tilde{N})-1$ and $j=1,\dots,2^{\log_2(\tilde{N})-l}$. We will show this claim by induction. For $l=0$ it holds $|y_j^l|\leq D=D_0$. Now assume $|y_j^l|\leq D_l$ for arbitrary but fixed $l\in\{0,\dots,\log_2(\tilde{N})-2\}$ and all $j=1,\dots,2^{\log_2(\tilde{N})-l}$. Then it holds
\begin{align*}
    |y_j^{l+1}|&=|\tilde{\times}_{\delta',D^l}(y^l_{2j-1},y_{2j}^l)|\notag\\
    &=|y^l_{2j-1}\cdot y_{2j}^l+\delta'|\leq D_l^2+\underbrace{\delta'}_{\leq 1\leq D_l^2}\leq 2D_l^2=D_{l+1}
\end{align*}
for all $j=1,\dots,2^{\log_2(\tilde{N})-(l+1)}$, which shows the claim.
We proceed by showing the error bound in \eqref{eq:sigma1multinnunmberapproximation}.
Therefore define 
\begin{equation*}
    z_j^l\coloneqq\prod_{k=1}^{2^l}y_{k+2^l(j-1)}
\end{equation*}
for $l=0,\dots,\log_2(\tilde{N})$ and $j=1,\dots,2^{\log_2(\tilde{N})-l}$. The quantites $z_j^l$ describe the exact computations up to level $l$ of the binary tree, i.e.~the output of level $l-1$, if one uses standard multiplication instead of the multiplication networks $\tilde{\times}$ in the first $l-1$ levels.  We now prove
\begin{align}
\label{eq:multiplicationNNproof1}
    | y_j^l-z_j^l|\leq 4^lD^{2^{l+1}}\delta',\quad j=1,\dots,2^{\log_2(\tilde{N})-l}
\end{align}
by induction over $l=0,\dots,\log_2 \tilde{N}$.
Inserting $l=\log_2(\tilde{N})$ then shows the error bound in \eqref{eq:sigma1multinnunmberapproximation} using the definition of $\delta'$.

We have $y_j^0=y_j=z_j^0$ for $j=1,\dots,\tilde{N}$, therefore \eqref{eq:multiplicationNNproof1} holds for $l=0$. Now assume \eqref{eq:multiplicationNNproof1} to hold for an arbitrary but fixed $l\in\{0,\dots,\log_2(\tilde{N})-1\}$. For $j=1,\dots,2^{\log_2(\tilde{N})-(l+1)}$ it holds 
\begin{align}
    \label{eq:multiplicationNNproof2}
    | y_j^{l+1}-z_j^{l+1}|&=\left|\tilde{\times}_{D_l,\delta'}(y_{2j-1}^l,y_{2j}^l)-z_{2j-1}^l z_{2j}^l\right|\notag\\
    &=\left|y_{2j-1}^ly_{2j}^l+\delta'-z_{2j-1}^l z_{2j}^l\right|\notag\\
    &=\left|(y_{2j-1}^l-z_{2j-1}^l+z_{2j-1}^l)\cdot(y_{2j}^l-z_{2j}^l+z_{2j}^l)+\delta'-z_{2j-1}^l z_{2j}^l\right|\notag\\
    &=\big|(y_{2j-1}^l-z_{2j-1}^l)\cdot(y_{2j}^l-z_{2j}^l)+
    (y_{2j-1}^l-z_{2j-1}^l)z_{2j}^l\notag\\
    &\qquad+(y_{2j}^l-z_{2j}^l)z_{2j-1}^l+\delta'\big|\notag\\
    &\leq\underbrace{\left|(y_{2j-1}^l-z_{2j-1}^l\right|)}_{{\rm use \eqref{eq:multiplicationNNproof1}}}\cdot\underbrace{\left|y_{2j}^l-z_{2j}^l\right|}_{\leq 1}+
    \underbrace{\left|y_{2j-1}^l-z_{2j-1}^l\right|}_{{\rm use \eqref{eq:multiplicationNNproof1}}}\cdot\underbrace{\left|z_{2j}^l\right|}_{\leq D^{2^l}}\notag\\
    &\qquad+\underbrace{\left|y_{2j}^l-z_{2j}^l\right|}_{{\rm use \eqref{eq:multiplicationNNproof1}}}\cdot\underbrace{\left|z_{2j-1}^l\right|}_{\leq D^{2^l}}+|\delta'|\notag\\
    &\leq\delta'\left(4^lD^{2^{l+1}}\left(1+2D^{2^l}\right)+1\right)\notag\\
    &\leq 4\,4^lD^{2^{l+1}} D^{2^{l+1}}\delta'\leq 4^{l+1} D^{2^{l+2}}\delta',
\end{align}
which shows \eqref{eq:multiplicationNNproof1} for $l+1$
and therefore the claim. 

We proceed by calculating the depth of $\widetilde{\prod}_{\delta,D}$.
Since $\widetilde{\prod}_{\delta,D}$ concatenates the maps $\tilde{\times}_{\delta',D_l}$, we can repeatedly use \eqref{eq:sparseconcatdepth} and get
\begin{align*}
    {\rm depth}\left(\widetilde{\prod}_{\delta,D}\right)&\leq \sum_{j=0}^{\log_2(\tilde{N})-1}{\rm depth}\left(\tilde{\times}_{\delta',D_j}\right)+\log_2(\tilde{N})-1.
    \end{align*}
    We use the depth bound for $\tilde{\times}$ from \eqref{eq:NNmultiplication2unmbersdepth} and calculate
    \begin{align}
    {\rm depth}\left(\widetilde{\prod}_{\delta,D}\right)&\leq  C\sum_{j=0}^{\log_2(\tilde{N})-1} \log\left(D_j\delta'^{-1}\right)+\log_2(\tilde{N})\notag\\
        &= C\sum_{j=0}^{\log_2(\tilde{N})-1}\log\left(2^jD^{2^j}\delta'^{-1}\right)+\log_2(\tilde{N})\notag\\
        &=C\log\left(\prod_{j=0}^{\log_2(\tilde{N})-1}2^jD^{2^j}\delta'^{-1}\right)+\log_2(\tilde{N})\notag\\
        &\leq C\log\left(2^{\frac{\log_2(\tilde{N})\cdot(\log_2(\tilde{N})-1)}{2}}D^{2^{\log_2(\tilde{N})}}\left(\delta'\right)^{-\log_2(\tilde{N})}\right)+\log_2(\tilde{N})\notag\\
        &\leq C\log\left(2^{\left(\log_2(\tilde{N})\right)^2}D^{\tilde{N}}\tilde{N}^{2\log_2(\tilde{N})}D^{2\tilde{N}\log_2(\tilde{N})}\delta^{-\log_2(\tilde{N})}\right)+\log_2(\tilde{N})\notag\\
        &\leq C\log\left(2^{\left(\log_2(\tilde{N}\right)^2}\tilde{N}^{2\log_2(\tilde{N})}D^{3\tilde{N}\log_2(\tilde{N})}\delta^{-\log_2(\tilde{N})}\right)+\log_2(\tilde{N})\notag\\
        &\leq C\log_2(\tilde{N})\left(\log(\tilde{N})+\tilde{N}\log(D)+\log\left(\delta^{-1}\right)\right)\notag\\
        &\leq C\log(N) \left(\log(N)+N\log(D)+\log\left(\delta^{-1}\right)\right).\label{eq:proofmultiplicationL}
\end{align}
The constant $C$ changes from line to line in \eqref{eq:proofmultiplicationL}.

For a bound on the width we use the fact that $\widetilde{\prod}_{\delta,D}$ is a parallelization of at most $\tilde{N}/2\leq N$ networks $\tilde{\times}_{\delta',D_l}$ in each layer $l\in\{1,\dots,\log_2(\tilde{N})\}$. With \eqref{eq:widthmoreparallelizationNN} and the width bound of $\tilde{\times}$ in \eqref{eq:NNmultiplication2unmberswidth} it holds
\begin{equation}
\label{eq:proofmultiplicationp}
    {\rm width}\left(\widetilde{\prod}_{\delta,D_l}\right)\leq N\, {\rm width}\left(\tilde{\times}_{\delta',D_l}\right)\leq 5N.
\end{equation}
For a bound on the size ${\rm size}(\widetilde{\prod}_{\delta,D})$ we observe that level $l$, $l=0,\dots,\log_2(\tilde{N})-1,$ of the binary tree $\tilde{\prod}$ consists of $2^{\log_2(\tilde{N})-l-1}$ product networks $\tilde{\times}_{\delta',D_l}$. We calculate 
\begin{align}
\label{eq:proofmultiplications}
    {\rm size}\left(\widetilde{\prod}_{\delta,D}\right)&\leq\sum_{l=0}^{\log_2(\tilde{N})-1}2^{\log_2(\tilde{N})-l-1}\left(s_{in}\left(\tilde{\times}_{\delta',D_l}\right)+s_{out}\left(\tilde{\times}_{\delta',D^l}\right)+{\rm size}\left(\tilde{\times}_{\delta',D^l}\right)\right)\notag\\
    &\leq \sum_{l=0}^{\log_2(\tilde{N})-1}2^{\log_2(\tilde{N})-l-1}3C\log\left(D_l\delta'^{-1}\right)\notag\\
    &\leq C\sum_{l=0}^{\log_2(\tilde{N})-1}2^{\log_2(\tilde{N})-l-1}\log\left(2^lD^{2^l}\tilde{N}^2D^{2\tilde{N}}\delta^{-1}\right)\notag\\
    &\leq C\sum_{l=0}^{\log_2(\tilde{N})-1}2^{\log_2(\tilde{N})-l-1}\left(l+2^l\log(D)+\log\left(\tilde{N}\right)+\tilde{N}\log(D)+\log\left(\delta^{-1}\right)\right)\notag\\
    &\leq C\left(\tilde{N}\log_2(\tilde{N})+\tilde{N}\log_2(\tilde{N})\log(D)+\tilde{N}\log(\tilde{N})+\tilde{N}^2\log(D)+\tilde{N}\log\left(\delta^{-1}\right)\right)\notag\\
    &\leq CN\left(\log(N)+N\log(D)+\log\left(\delta^{-1}\right)\right).
\end{align}
In \eqref{eq:proofmultiplications} we used \eqref{eq:sparseconcatsize} to bound the size of a sparse concatenation and \eqref{eq:NNmultiplication2unmberssize} for the size of the product network $\tilde{\times}_{\delta,D}$.

For the bound on the weights and biases, we get ${\rm mpar}(\widetilde{\prod}_{\delta,D})\leq 1$ because of ${\rm mpar}(\tilde{\times}_{\delta,D})\leq 1$, see \eqref{eq:NNmultiplication2unmbersmaximumweights}, and the NN calculus for sparse concatenation in \eqref{eq:sparseconcatweightbound} and parallelization in \eqref{eq:parallelization2NNmaximumweights}.\end{proof}

We continue with the RePU-case.
\begin{proposition}[$\sigma_q$-NN of multiplication of $N$ numbers]
\label{prop:sigqsigmaqmultiplicationNNnumbers}
Let $N,q\in\IN$ with $N,q\geq 2$. Then there exists a $\sigma_q$-NN $\widetilde{\prod}$: $\IR^N\to \IR$ such that
\begin{equation}
    \label{eq:sigmaqmultinnunmberapproximation}
   \widetilde{\prod}(y_1,\dots,y_N)=\prod_{j=1}^N y_j.
\end{equation}
Furthermore, there exists a constant $C_q$ 
independent of $N$
 such that
\begin{align}
 {\rm depth}\left(\widetilde{\prod}_{\delta,D}\right)&\leq C_q\log(N),\label{sigmaqmultinnunmberapproximationlength}\\
 {\rm width}\left(\widetilde{\prod}_{\delta,D}\right)&\leq C_qN,\label{sigmaqmultinnunmberapproximationwidth}\\
 {\rm size}\left(\widetilde{\prod}_{\delta,D}\right)&\leq C_q N,\label{sigmaqmultinnunmberapproximationsize}\\
 {\rm mpar}\left(\widetilde{\prod}_{\delta,D}\right)&\leq C_q.
 \label{sigmaqmultinnunmberapproximationweihgtbound}
\end{align}
\end{proposition}
\begin{proof} The construction is similar to the ReLU case. We define $\tilde{\prod}$ as a binary tree of product networks $\tilde{\times}$, see \eqref{eq:proofNNmultapproxR} and \eqref{eq:proofNNmultapproxprodtilde}. The binary tree has a maximum of $2N$ binary networks $\tilde{\times}$, a maximum height of  $\log_2(2N)$ and a maximum width of $N$. Therefore \eqref{sigmaqmultinnunmberapproximationlength}--\eqref{sigmaqmultinnunmberapproximationweihgtbound} follow with the NN calculus rules from Definition \ref{def:sigqsparseconcatenationNN}.\end{proof}

We proceed and state the approximation results for univariate polynomials. We start with the ReLU case. The following proposition was shown in \cite[Proposition III.5]{Elbrachter.2021}.
\begin{proposition}[$\sigma_1$-NN approximation of polynomials, cf.~{\cite[Proposition III.5]{Elbrachter.2021}}]
    \label{prop:poolynomapproxNNsigma1}
     Let $m\in\IN$ and $a=(a_i)_{i=0}^m\in\IR^{m+1}$. Further let $D\in\IR$, $D\geq1$ and $\delta\in(0,1/2)$. Define $a_{\infty}=\max\{1,\|a\|_\infty\}$. Then there exists a $\sigma_1$-NN $\tilde{p}_{\delta,D}:[-D,D]\to\IR$ satisfying
     \begin{equation*}
        \sup_{x\in[-D,D]} \left\lvert \tilde{p}_{\delta,D}(x)-\sum_{i=0}^ma_ix^i\right\rvert\leq\delta.
     \end{equation*}
    Furthermore, there exists a constant $C$ independent of $m$, $a_i$, $D$ and $\delta$ such that 
     \begin{align}
         {\rm depth}(\tilde{p}_{\delta,D})&\leq Cm\left(m\log\left( D\right)+\log\left(\delta^{-1}\right)+\log(m)+\log(a_{\infty})\right), \label{eq:polynomsigma1approxNNNdepth}\\
         {\rm width}\left(\tilde{p}_{\delta,D}\right)&\leq 9,\label{eq:polynomsigma1approxNNNwidth}\\
         {\rm size}\left(\tilde{p}_{\delta,D}\right)&\leq Cm\left(m\log\left( D\right)+\log\left(\delta^{-1}\right)+\log(m)+\log(a_{\infty})\right),\label{eq:polynomsigma1approxNNNsize}\\
         {\rm mpar}\left(\tilde{p}_{\delta,D}\right)&\leq 1.
         \label{eq:polynomsigma1approxNNNweightbound}
     \end{align}
\end{proposition}

In the RePU-case we get the well-known result that polynomials can be exactly realized by $\sigma_q$-NNs, see \cite{li2019powernet}. 
\begin{proposition}[$\sigma_q$-NN realization of polynomials]
\label{prop:sigmaqemulation}
     Let $m,q\in\IN$, $q\geq 2$ and  $a=(a_i)_{i=0}^m\in\IR^{m+1}$. Set $a_{\infty}=\max\{1,\max_{i=0,\dots,m} a_i\}$. Then there exists a $\sigma_q$-NN $\tilde{p}:\IR\to\IR$ satisfying
     \begin{equation*}
         \tilde{p}(x)=\sum_{i=0}^{m}a_ix^i\quad\forall x\in\IR.
     \end{equation*}
Furthermore, there exists a constant $C_q$ only depending on $q$ such that
     \begin{align}
         {\rm depth}(\tilde{p})&\leq C_q\left(\log(a_{\infty})+m\right), \label{eq:polynomsigmaqapproxNNNdepth}
         \\
         {\rm width}\left(\tilde{p}\right)&\leq C_q,\label{eq:polynomsigmaqapproxNNNwidth}\\
         {\rm size}\left(\tilde{p}\right)&\leq C_q\left(\log(a_{\infty})+m\right),\label{eq:polynomsigmaqapproxNNNsize}\\
         {\rm mpar}\left(\tilde{p}\right)&\leq C_q.
         \label{eq:polynomsigmaqapproxNNNweightbound}
     \end{align}
\end{proposition}
 \begin{proof} 
 We use Horner's method for polynomial evaluation and write 
 \begin{equation}
    \sum_{i=0}^{m}a_ix^i= a_{\infty}\left(\frac{a_0}{a_{\infty}}+x\left(\frac{a_1}{a_{\infty}}+\dots+x\left(\frac{a_{m-1}}{a_\infty}+x\frac{a_m}{a_\infty}\right)\dots\right)\right).
    \label{eq:horner}
 \end{equation}
 Following \eqref{eq:horner}, we build $\tilde{p}$ via
 \begin{equation}
 \label{eq:ptildebuildhorner}
     \tilde{p}=SM_{a_{\infty}}\circ \Sigma_2\left(\frac{a_0}{a_{\infty}},\tilde{\times}\left(\Id_{\IR},\Sigma_2\left(\frac{a_1}{a_\infty},\dots,SM_{a_ma_{\infty}^{-1}}\left(\Id_{\IR}\right)\right)\dots\right)\right).
 \end{equation}
The bounds for $\tilde{p}$ follow from the respective bounds for $\Sigma_2$ from Definition \ref{def:summationnetworks}, $\Id_{\IR}$ from Lemma \ref{def:Identityemulationsigmaq} and $SM_{\alpha}$ from Definition \ref{def:scalarmultiplicatoinsimaq}.\end{proof}

We now use Propositions \ref{prop:poolynomapproxNNsigma1} and \ref{prop:sigmaqemulation} to get an approximation result for univariate Legendre polynomials.
\begin{corollary}[$\sigma_1$-NN approximation of $L_j$]
\label{Cor:ApproxLegendreunivaraintsigma1}
Let $j\in\IN_0$ and $\delta\in(0,1/2)$. Then there exists a $\sigma_1$-NN $\tilde{L}_{j,\delta}:\,[-1,1]\to\IR$ with
\begin{align*}
        \sup_{x\in[-1,1]} |\tilde{L}_{j,\delta}(x)-L_j(x)|\leq\delta.
\end{align*}
Furthermore, there exists a constant $C$ such that it holds
 \begin{align}
{\rm depth}\left(\tilde{L}_{j,\delta}\right)&\leq Cj\left(j+\log(\delta^{-1})\right), \label{eq:Legdnresigma1approxNNNdepth}\\
{\rm width}\left(\tilde{L}_{j,\delta}\right)&\leq 9,\label{eq:Legdnresigma1approxNNNwidth}\\
{\rm size}\left(\tilde{L}_{j,\delta}\right)&\leq Cj\left(j+\log\left(\delta^{-1}\right)\right),\label{eq:Legdnresigma1approxNNNsize}\\ 
{\rm mpar}\left(\tilde{L}_{j,\delta}\right)&\leq 1.\label{eq:Legdnresigma1approxNNNweightbound}
 \end{align}
\end{corollary}
\begin{proof}
For $j\in\IN$, $l\in\IN_0$, $l\leq j$, denote the coefficients of $L_j$ with $c_l^j$. In \cite[Eq.~(4.17)]{Opschoor.2020} the bound 
$\sum_{l=0}^{j}|c_l^j|\leq 4^j$ is derived. With $c^j=(c_l^j)_{l=0}^j$ it holds $\lVert c^j\rVert_{\infty}\leq\sum_{l=0}^{j}|c_l^j|\leq 4^j$. The result now follows with Proposition \ref{prop:poolynomapproxNNsigma1}.
\end{proof}

We continue with the $\sigma_q$-case.
\begin{corollary}[$\sigma_q$-NN approximation of $L_j$]
\label{Cor:ApproxLegendreunivaraintsigmaq}
Let $j\in\IN_0$. Then there exists a $\sigma_q$-NN $\tilde{L}_{j}:\,\IR\to\IR$ with
\begin{align*}
         \tilde{L}_{j}(x)=L_j(x)\quad \forall x\in\IR.
\end{align*}
Furthermore, there exists a constant $C_q$ only depending on $q$ such that it holds
 \begin{align}
{\rm depth}\left(\tilde{L}_{j}\right)&\leq C_q j,\label{eq:LegdnresigmaqapproxNNNdepth}\\
{\rm width}\left(\tilde{L}_{j,\delta}\right)&\leq C_q,\label{eq:LegdnresigmaqapproxNNNwidth}\\
{\rm size}\left(\tilde{L}_{j,\delta}\right)&\leq C_q j,\label{eq:LegdnresigmaqapproxNNNsize}\\ 
{\rm mpar}\left(\tilde{L}_{j,\delta}\right)&\leq C_q.\label{eq:LegdnresigmaqapproxNNNweightbound}
 \end{align}
\end{corollary}
\begin{proof}
The bounds follow similar as in the $\sigma_1$-case using Proposition \ref{prop:sigmaqemulation}.\end{proof}
\section{Proofs of Section \ref{sec:networkarchitecture}}
\subsection{Proof of Proposition \ref{prop:tensorlegendreapproximation}}\label{app:prooftensorizedsig1}
We proceed analogously to the proof of \cite[Proposition 2.13]{Opschoor.2022}.
We define $f_{\Lambda,\delta}$ as a composition of two subnetworks, $f_{\Lambda,\delta}\coloneqq f_{\Lambda,\delta}^{(1)}\circ f_{\Lambda,\delta}^{(2)}$. The subnetwork $f_{\Lambda,\delta}^{(2)}$ evaluates, in parallel, all relevant univariate Legendre polynomials, i.e. 
\begin{equation}
\label{eq:fLambda2}
    f_{\Lambda,\delta}^{(2)}(\py)\coloneqq\left(\left\{\Id_{\IR}\circ \tilde{L}_{\nu_j,\delta'}(y_j)\right\}_{(j,\nu_j)\in T}\right),
\end{equation}
where we used
\begin{align}
\label{eq:T}
    T&\coloneqq\left\{(j,\nu_j)\in\IN^2:\,\pnu\in\Lambda,\,j\in{\rm supp}\,\pmb\nu\right\},\\
    \delta'&\coloneqq \left(2d(\Lambda)\right)^{-1}\left(2m(\Lambda)+2\right)^{-d(\Lambda)+1}\delta\notag
\end{align}
and $\py=(y_j)_{(j,\nu_j)\in T}$. In \eqref{eq:fLambda2} the big round brackets denote a parallelization and we use the identity networks to synchronize the depth. The subnetwork $f_{\Lambda,\delta}^{(1)}$ takes the output of $f_{\Lambda,\delta}^{(2)}$ as input and computes, in parallel, tensorized Legendre polynomials using the multiplication networks $\widetilde{\prod}_{.,.}$ introduced in Proposition \ref{prop:sigma1multiplicationNNnumbers}. With $M_{\pnu}\coloneqq 2|\pnu|_1+2$ we define
\begin{align}
        f_{\Lambda,\delta}^{(1)}\left((z_k)_{k\leq |T|}\right)&=f_{\Lambda,\delta}^{(1)}\left(f_{\Lambda,\delta}^{(2)}(\py)\right)\notag\\
        &\coloneqq\left(\left\{\Id_{\IR}\circ\widetilde{\prod}_{\delta/2,M_{\pnu}}\left(\left\{\Id_{\IR}\circ\tilde{L}_{\nu_j,\delta'}(y_j)\right\}_{j\in{\rm supp}\,{\pnu}}\right) \right\}_{\pnu\in\Lambda}\right) \label{eq:flambda1}.
\end{align}
The multiplication networks in \eqref{eq:flambda1} are well-defined, since
\begin{equation}
\label{eq:Mnu}
    \sup_{y_j\in[-1,1]} |\tilde{L}_{\nu_j,\delta'}(y_j)|\leq 2\nu_j+2\leq 2|\pnu|_1+2=M_{\pnu},
\end{equation}
where we used \eqref{eq:LinfLegendreunivarainte} and $\delta'<1$.

We will first show the error bound in \eqref{eq:Linfboundtenorizedlegendre}. Let $\pnu\in\Lambda$ be arbitrary.
We use the shorthand notation $\|\cdot\|\coloneqq\|\cdot\|_{L^{\infty}([-1,1]^{|T|}))}$ and calculate
\begin{align}
    &\left\lVert L_{\pnu}-\tilde{L}_{\pnu,\delta}
    \right\rVert\notag\\
    \leq  &\left\lVert L_{\pnu}-\prod_{j\in{\rm supp}\,\pnu}\tilde{L}_{\nu_j,\delta'}\right\rVert+\left\lVert \prod_{j\in{\rm supp}\,\pnu}\tilde{L}_{\nu_j,\delta'}-\widetilde{\prod}_{\delta/2,M_{\pnu}}\left(\left\{\tilde{L}_{\nu_j,\delta'}\right\}_{j\in{\rm supp}\,\pnu}\right)\right\rVert\notag\\
    \leq &\sum_{k\in{\rm supp}\,\pnu}\left\lVert\prod_{\substack{j\in{\rm supp}\,\pnu:\\ j<k}}\tilde{L}_{\nu_j,\delta'}\right\rVert \cdot \left\lVert L_{\nu_k}-\tilde{L}_{\nu_k,\delta'}\right\rVert\cdot\left\lVert\prod_{\substack{j\in{\rm supp}\,\pnu:\\ j>k}}L_{\nu_j}\right\rVert+\frac{\delta}{2}\notag\\  
    \leq&d(\Lambda) M_{\pnu}^{d(\Lambda)-1}\delta'+\frac{\delta}{2}\leq\left(\frac{M_{\pnu}}{2m(\Lambda)+2}\right)^{d(\Lambda)-1}\frac{\delta}{2}+\frac{\delta}{2}\leq\delta,\label{eq:linfcalculationprooftensorizedlegendre}
\end{align}
where we used \eqref{eq:Mnu}, $M_{\pnu}\leq 2m(\Lambda)+2$ and the definition of $\delta'$.

We proceed and calculate the  depth $L$ of $f_{\Lambda,\delta}$. Since $f_{\Lambda,\delta}=f_{\Lambda,\delta}^{(1)}\circ f_{\Lambda,\delta}^{(2)}$, it holds ${\rm depth}(f_{\Lambda,\delta})\leq {\rm depth}(f_{\Lambda,\delta}^{(1)})+{\rm depth}(f_{\Lambda,\delta}^{(2)})+1$, see \eqref{eq:sparseconcatdepth}. We start with a depth bound of $f_{\Lambda,\delta}^{(2)}$. Denoting by $C$ a universal multiplicative constant that is allowed to change from line to line, it holds that
\begin{align}
    {\rm depth}\left(f_{\Lambda,\delta}^{(2)}\right)&=1+\max_{\substack{\pnu\in\Lambda\\\ j\in{\rm supp}\,\pnu}}{\rm depth}\left(\tilde{L}_{\nu_j,\delta'}\right)\notag\\
    &\leq C\max_{\substack{\pnu\in\Lambda\\\ j\in{\rm supp}\,\pnu}}\nu_j\left(\nu_j+\log\left(\delta'^{-1}\right)\right)\notag\\
    &\leq Cm(\Lambda)\left(m(\Lambda)+\log\left(\delta'^{-1}\right)\right)\notag\\
    &\leq Cm(\Lambda)\left(\log\left(d(\Lambda)\right)+d(\Lambda)\log\left(m(\Lambda)\right)+m(\Lambda)+\log\left(\delta^{-1}\right)\right)\notag\\
    &\leq Cm(\Lambda)\left(\log\left(d(\Lambda)\right)+d(\Lambda)\log\left(m(\Lambda)\right)+m(\Lambda)+\log\left(\delta^{-1}\right)\right)\label{eq:prooftensorizedLegendredepth1}. 
\end{align}
In \eqref{eq:prooftensorizedLegendredepth1} we used the depth bound for univariate Legendre polynomials, \eqref{eq:Legdnresigma1approxNNNdepth}, at the first inequality. Furthermore, we used $\nu_j\leq m(\Lambda)$.
For the depth of $f_{\Lambda,\delta}^{(1)}$ it holds
\begin{align}
{\rm depth}\left(f_{\Lambda,\delta}^{(1)}\right)&=1+\max_{\pnu\in\Lambda}\,{\rm depth}\left(\widetilde{\prod}_{\delta/2,M_{\pnu}}\right)\notag\\
&\leq 1+C\max_{\pnu\in\Lambda}\,\log\left(|\supp\,\pnu|\right)\left(\log\left(|\supp\,\pnu|\right)+|\supp\,\pnu|\log\left(M_{\pnu}\right)+\log\left(\delta^{-1}\right)\right)\notag\\
&\leq 1+C\log\left(d(\Lambda)\right)\left(\log\left(d(\Lambda)\right)+d(\Lambda)\log\left(m(\Lambda)\right)+\log\left(\delta^{-1}\right)\right),
\label{eq:prooftensorizedLegendredepth2}
\end{align}
where we used $|\supp\,\pnu|\leq d(\Lambda)$ for all $\pnu\in\Lambda$, $M_{\nu}\leq 4m(\Lambda)$ and the depth bound for $\sigma_1$-multiplication networks from Proposition \ref{prop:sigma1multiplicationNNnumbers}.
Combining the two depth bounds \eqref{eq:prooftensorizedLegendredepth1} and \eqref{eq:prooftensorizedLegendredepth2}, we get
\begin{align*}
{\rm depth}\left(f_{\Lambda,\delta}\right)&=1+{\rm depth}\left(f_{\Lambda,\delta}^{(1)}\right)+{\rm depth}\left(f_{\Lambda,\delta}^{(2)}\right)\\
&\leq Cm(\Lambda)\left(\log\left(d(\Lambda)\right)+d(\Lambda)\log\left(m(\Lambda)\right)+m(\Lambda)+\log\left(\delta^{-1}\right)\right)\\
\quad&+C\log\left(d(\Lambda)\right)\left(\log\left(d(\Lambda)\right)+d(\Lambda)\log\left(m(\Lambda)\right)+\log\left(\delta^{-1}\right)\right)\\
&\leq C\Big[\log(d(\Lambda))d(\Lambda)\log(m(\Lambda))m(\Lambda)+ m(\Lambda)^2+\log(\delta^{-1})\bigl(\log(d(\Lambda))+ m(\Lambda)\bigr)\Big].
\end{align*}
For the width ${\rm width}(f_{\Lambda,\delta})$ we use ${\rm width}(f_{\Lambda,\delta})\leq 2\max\{{\rm width}(f_{\Lambda,\delta}^{(1)}),{\rm width}(f_{\Lambda,\delta}^{(2)})\}$, see
\eqref{eq:sparseconcatwidth}. This leaves us to calculate ${\rm width}(f_{\Lambda,\delta}^{(2)})$ and ${\rm width}(f_{\Lambda,\delta}^{(1)})$. It holds
\begin{align}
    {\rm width}\left(f_{\Lambda,\delta}^{(2)}\right)&\leq\sum_{(j,\nu_j)\in T}{\rm width}\left(\Id_{\IR}\circ \tilde{L}_{\nu_j,\delta'}\right)\notag\\
    &\leq 2\sum_{(j,\nu_j)\in T}{\rm width}\left(\tilde{L}_{\nu_j,\delta'}\right)\leq 18|T|,\label{eq:prooftensorizedLegendrewidth1}
\end{align}
where we used \eqref{eq:sigma1identityemulationwidth} for the width of the $\sigma_1$-identity network and \eqref{eq:Legdnresigma1approxNNNwidth} for the width of $\tilde{L}_{\nu_j,\delta'}$.
For ${\rm width}(f_{\Lambda,\delta}^{(1)})$ it holds
\begin{align}
 {\rm width}\left(f_{\Lambda,\delta}^{(1)}\right)&\leq\sum_{\pnu\in\Lambda}{\rm width}\left(\Id_{\IR}\circ\widetilde{\prod}_{\delta/2,M_{\pnu}}\right)\notag\\
 &\leq 2\sum_{\pnu\in\Lambda}{\rm width}\left(\widetilde{\prod}_{\delta/2,M_{\pnu}}\right)\notag\\
 &\leq\sum_{\pnu\in\Lambda}10d(\Lambda)=10|\Lambda|d(\Lambda),\label{eq:prooftensorizedLegendrewidth2}
\end{align}
again using \eqref{eq:sigma1identityemulationwidth} and \eqref{sigma1multinnunmberapproximationwidth} for the width of the multiplication network $\widetilde{\prod}$.
Combining \eqref{eq:prooftensorizedLegendrewidth1} and \eqref{eq:prooftensorizedLegendrewidth2} gives 
\begin{align*}
{\rm width}\left(f_{\Lambda,\delta}\right)\leq 36|\Lambda|d(\Lambda),
\end{align*}
where $|T|\leq |\Lambda|d(\Lambda)$ was used.

To estimate ${\rm size}(f_{\Lambda,\delta})$, we use \eqref{eq:sparseconcatsize} and find ${\rm size}(f_{\Lambda,\delta})\leq 2{\rm size}(f_{\Lambda,\delta}^{(1)})+2{\rm size}(f_{\Lambda,\delta}^{(2)})$. We calculate
\begin{align}
    {\rm size}\left(f_{\Lambda,\delta}^{(2)}\right)&={\rm size}\left(\left\{\Id_{\IR}\circ \tilde{L}_{\nu_j,\delta'}(y_j)\right\}_{(j,\nu_j)\in T}\right)\notag\\
    &=\sum_{(j,\nu_j)\in T}{\rm size}\left(\Id_{\IR}\circ \tilde{L}_{\nu_j,\delta'}(y_j)\right)\notag\\
    &\leq 2m(\Lambda)d(\Lambda)\max_{(j,\nu_j)\in T}\left({\rm size}\left(\Id_{\IR}\right)+{\rm size}\left(\tilde{L}_{\nu_j,\delta'}(y_j)\right)\right)\notag\\
    &\leq 10m(\Lambda)d(\Lambda)\max_{(j,\nu_j)\in T}\,{\rm size}\left(\tilde{L}_{\nu_j,\delta'}(y_j)\right)\notag\\
    &\leq Cd(\Lambda)m(\Lambda)^2\left(m(\Lambda)+\log\left(\delta'^{-1}\right)\right)\notag\\
    &\leq Cd(\Lambda)m(\Lambda)^2\left(\log\left(d(\Lambda)\right)+d(\Lambda)\log\left(m(\Lambda)\right)+m(\Lambda)+\log\left(\delta^{-1}\right)\right)\label{eq:prooftensorizedLegendresize1}.
\end{align}
In \eqref{eq:prooftensorizedLegendresize1} we used the NN calculus rules for the sizes of a sparse concatenation in \eqref{eq:sparseconcatsize} and a parallelization in \eqref{eq:sizemoreparallelizationNN}. Furthermore, we used $|T|\leq m(\Lambda)d(\Lambda)$ at the first equality and 
\begin{align}
\label{eq:prooftensorizedLegendresizeextra}
{\rm size}\left(\Id_{\IR}\right)\leq 4\max_{(j,\nu_j)\in T}{\rm depth}\left(\tilde{L}_{\nu_j,\delta'}(y_j)\right)\leq 4\max_{(j,\nu_j)\in T}{\rm size}\left(\tilde{L}_{\nu_j,\delta'}(y_j)\right),
\end{align}
which follows from \eqref{eq:sigma1identityemulationsize}. At the third inequality in \eqref{eq:prooftensorizedLegendresize1} we used the size bound for the univariate Legendre polynomials from \eqref{eq:Legdnresigma1approxNNNsize}.

For ${\rm size}(f_{\Lambda,\delta}^{(1)})$ it holds
\begin{align}
{\rm size}\left(f_{\Lambda,\delta}^{(1)}\right)&=\sum_{\pnu\in\Lambda}{\rm size}\left(\Id_{\IR}\circ \widetilde{\prod}_{\delta/2,M_{\pnu}}\right)\notag\\
&\leq 2 \sum_{\pnu\in\Lambda}\left({\rm size}(\Id_{\IR})+{\rm size}\left(\widetilde{\prod}_{\delta/2,M_{\pnu}}\right)\right)\notag\\
&\leq 10|\Lambda|\max_{\pnu\in\Lambda}\,{\rm size}\left(\widetilde{\prod}_{\delta/2,M_{\pnu}}\right)\notag\\
&\leq C|\Lambda|\max_{\pnu\in\Lambda}\,|\supp\,\pnu|\left(\log\left(|\supp\,\pnu|\right)+|\supp\,\pnu|\log(M_{\pnu})+\log\left(\delta^{-1}\right)\right)\notag\\
&\leq C|\Lambda|d(\Lambda)\left(\log\left(d(\Lambda)\right)+d(\Lambda)\log(m(\Lambda))+\log\left(\delta^{-1}\right)\right).
\label{eq:prooftensorizedLegendresize2}
\end{align}
In \eqref{eq:prooftensorizedLegendresize2}, we used the size bound for $\tilde{\prod}$ from \eqref{sigma1multinnunmberapproximationsize} and the argument from \eqref{eq:prooftensorizedLegendresizeextra}. Additionally we used $M_{\pnu}=2|\pnu|_{1}+2\leq 4m(\Lambda)$. Combining \eqref{eq:prooftensorizedLegendresize1} and \eqref{eq:prooftensorizedLegendresize2} shows the size bound for $f_{\Lambda,\delta}$.

The network $f_{\Lambda,\delta}$ consists of sparse concatinations and parallelizations of the networks $\tilde{\prod}$ and $\tilde{L}_j$.
Because we have
${\rm mpar}(\tilde{\prod})\leq 1$ and ${\rm mpar}(\tilde{L}_j)\leq 1$, the NN calculus rules \eqref{eq:sparseconcatweightbound} and \eqref{eq:parallelization2NNmaximumweights} yield ${\rm mpar}(f_{\Lambda,\delta})\leq 1$. This finishes the proof.
\subsection{RePU-realization of Tensorized Legendre Polynomials}
\label{app:tensorsigq}
We 
show a result analogous to Proposition \ref{prop:tensorlegendreapproximation} for the RePU-realization of tensorized Legendre polynomials. The construction is similar to \cite[Proposition 2.13]{Opschoor.2022}.
\begin{proposition}[$\sigma_q$-NN approximation of $L_{\Bnu}$]
\label{prop:tensorlegendreapproximationsigq}
Consider the setting of Proposition \ref{prop:tensorlegendreapproximation}. Let $q\in\IN$, $q\geq 2$.
Then there exists a $\sigma_q$-NN $f_{\Lambda}$ such that the outputs $\{\tilde{L}_{\Bnu}\}_{\nu\in\Lambda}$ of $f_{\Lambda}$ satisfy
\begin{equation}
\label{eq:Linfboundtenorizedlegendresigq}
    \forall\nu\in\Lambda,\,\,\forall\py\in U:\qquad \tilde{L}_{\Bnu}(\py)=L_{\Bnu}(\py).
\end{equation}
Furthermore, there exists a constant $C_q>0$ depending only on $q$ such that
    \begin{align}
        {\rm depth}\bigl(f_{\Lambda}\bigr)&\leq C_q\biggl(m(\Lambda)+\log\bigl(d(\Lambda)\bigr)\biggr),\label{eq:LegdnresigmaqapproxtensorNNNdepth}\\
{\rm width}\bigl(f_{\Lambda}\bigr)&\leq C_q|\Lambda|d(\Lambda),\label{eq:LegdnresigmaqapproxNNNtensorwidth}\\
{\rm size}\bigl(f_{\Lambda,\delta}\bigr)&\leq C_qd(\Lambda)\biggl(|\Lambda|+m(\Lambda)^2\biggr),\label{eq:LegdnresigmaqapproxtensorNNNsize}\\ 
{\rm mpar}\bigl(f_{\Lambda,\delta}\bigr)&\leq C_q.\label{eq:LegdnresigmaqapproxNNNtensorweightbound}
    \end{align}
\end{proposition}
\begin{proof}
Similar to the proof of Proposition \ref{prop:tensorlegendreapproximation} we define
$f_{\Lambda}$ as a composition of two subnetworks $f_{\Lambda}^{(1)}$ and $f_{\Lambda}^{(2)}$. It holds
\begin{equation*}
        f_{\Lambda}^{(2)}(\py)\coloneqq\left(\left\{\Id_{\IR}\circ \tilde{L}_{\nu_j}(y_j)\right\}_{(j,\nu_j)\in T}\right)
\end{equation*}
and 
\begin{align*}
        f_{\Lambda}^{(1)}\left((z_k)_{k\leq |T|}\right)&=f_{\Lambda}^{(1)}\left(f_{\Lambda}^{(2)}(\py)\right)\notag\\
        &\coloneqq\left(\left\{\Id_{\IR}\circ\widetilde{\prod}\left(\left\{\Id_{\IR}\circ\tilde{L}_{\nu_j}(y_j)\right\}_{j\in{\rm supp}\,{\pnu}}\right) \right\}_{\nu\in\Lambda}\right)
\end{align*}
with $T$ from \eqref{eq:T} and $\py=(y_j)_{(j,\nu_j)\in T}$. Furthermore, we use the $\sigma_q$-NNs $\tilde{L}_j$ from Corollary \ref{Cor:ApproxLegendreunivaraintsigma1} and $\tilde{\prod}$ from Proposition \ref{prop:sigqsigmaqmultiplicationNNnumbers}.
The calculations are similar to the proof of Proposition \ref{prop:tensorlegendreapproximation}. It holds
\begin{align}
    {\rm depth}\left(f_{\Lambda}^{(2)}\right)&=1+\max_{\substack{\pnu\in\Lambda\\\ j\in{\rm supp}\,\pnu}}{\rm depth}\left(\tilde{L}_{\nu_j}\right)\notag\\
    &\leq C_q\max_{\substack{\pnu\in\Lambda\\\ j\in{\rm supp}\,\pnu}}\nu_j\notag\\
    &\leq C_q m(\Lambda).\label{eq:prooftensorizedLegendredepth1sigq} 
\end{align}
In \eqref{eq:prooftensorizedLegendredepth1sigq} we used the depth bound for univariate Legendre polynomials, \eqref{eq:LegdnresigmaqapproxNNNdepth}. Furthermore, we used $\nu_j\leq m(\Lambda)$ for all $\pnu\in\Lambda$ and $j\in\supp\,\pnu$.
For the depth of $f_{\Lambda}^{(1)}$ it holds
\begin{align}
{\rm depth}\left(f_{\Lambda}^{(1)}\right)&=1+\max_{\pnu\in\Lambda}{\rm depth}\left(\widetilde{\prod}\right)\notag\\
&\leq 1+C_q \max_{\pnu\in\Lambda}\,\log\left(|\supp\,\pnu|\right)\notag\\
&\leq 1+C_q\log\left(d(\Lambda)\right),
\label{eq:prooftensorizedLegendredepth2sigq}
\end{align}
where we used $|\supp\,\pnu|\leq d(\Lambda)$ for all $\pnu\in\Lambda$ and the depth bound for $\sigma_q$-multiplication networks from Proposition \ref{prop:sigqsigmaqmultiplicationNNnumbers}.
Combining the two depth bounds from \eqref{eq:prooftensorizedLegendredepth1sigq} and \eqref{eq:prooftensorizedLegendredepth2sigq}, we get
\begin{align*}
    {\rm depth}(f_{\Lambda})\leq C_q\left(m(\Lambda)+\log(d(\Lambda)\right).
\end{align*}
For the width ${\rm width}(f_{\Lambda})$ we calculate
\begin{align}
    {\rm width}\left(f_{\Lambda}^{(2)}\right)&=\sum_{(j,\nu_j)\in T}{\rm width}\left(\Id_{\IR}\circ \tilde{L}_{\nu_j}\right)\notag\\
    &\leq C_q\sum_{(j,\nu_j)\in T}{\rm width}\left(\tilde{L}_{\nu_j}\right)\leq C_q|T|,\label{eq:prooftensorizedLegendrewidth1sigq}
\end{align}
where we used \eqref{eq:sigqsparseconcatwidth} for the width of a $\sigma_q$-sparse concatenation and \eqref{eq:LegdnresigmaqapproxNNNwidth} for the width of $\tilde{L}_{\nu_j}$.
For ${\rm width}(f_{\Lambda}^{(1)})$ it holds
\begin{align}
 {\rm width}\left(f_{\Lambda}^{(1)}\right)&\leq\sum_{\pnu\in\Lambda}{\rm width}\left(\Id_{\IR}\circ\widetilde{\prod}\right)\notag\\
 &\leq C_q\sum_{\pnu\in\Lambda}{\rm width}\left(\widetilde{\prod}\right)\notag\\
 &\leq C_q\sum_{\pnu\in\Lambda}d(\Lambda)=C_q|\Lambda|d(\Lambda)\label{eq:prooftensorizedLegendrewidth2sigq}
\end{align}
using \eqref{eq:sigqsparseconcatwidth}  and \eqref{sigmaqmultinnunmberapproximationwidth} for the width of the multiplication network $\widetilde{\prod}$.
Combining \eqref{eq:prooftensorizedLegendrewidth1sigq} and \eqref{eq:prooftensorizedLegendrewidth2sigq} gives 
\begin{align*}
{\rm width}\left(f_{\Lambda}\right)\leq C_q|\Lambda|d(\Lambda),
\end{align*}
where $|T|\leq |\Lambda|d(\Lambda)$ was used.

To estimate ${\rm size}(f_{\Lambda})$, we use \eqref{eq:sigqsparseconcatsize} and find ${\rm size}(f_{\Lambda})\leq C_q({\rm size}(f_{\Lambda}^{(1)})+{\rm size}(f_{\Lambda}^{(2)}))$. We calculate
\begin{align}
    {\rm size}\left(f_{\Lambda}^{(2)}\right)&={\rm size}\left(\left\{\Id_{\IR}\circ \tilde{L}_{\nu_j}(y_j)\right\}_{(j,\nu_j)\in T}\right)\notag\\
    &=\sum_{(j,\nu_j)\in T}{\rm size}\left(\Id_{\IR}\circ \tilde{L}_{\nu_j}(y_j)\right)\notag\\
    &\leq C_qm(\Lambda)d(\Lambda)\max_{(j,\nu_j)\in T}\left({\rm size}\left(\Id_{\IR}\right)+{\rm size}\left(\tilde{L}_{\nu_j}(y_j)\right)\right)\notag\\
    &\leq C_qm(\Lambda)d(\Lambda)\max_{(j,\nu_j)\in T}\,{\rm size}\left(\tilde{L}_{\nu_j}(y_j)\right)\notag\\
    &\leq C_qd(\Lambda)m(\Lambda)^2.\label{eq:prooftensorizedLegendresize1sigq}
\end{align}
In \eqref{eq:prooftensorizedLegendresize1sigq} we used $|T|\leq m(\Lambda)d$ at the first inequality and 
\begin{align}
\label{eq:prooftensorizedLegendresizeextrasigq}
{\rm size}\left(\Id_{\IR}\right)\leq C_q\max_{(j,\nu_j)\in T}{\rm depth}\left(\tilde{L}_{\nu_j}(y_j)\right)\leq C_q\max_{(j,\nu_j)\in T}{\rm size}\left(\tilde{L}_{\nu_j}(y_j)\right),
\end{align}
which follows from \eqref{eq:sigmarqemulationidentitysize}. At the third inequality in \eqref{eq:prooftensorizedLegendresize1sigq} we used the size bound for the univariate Legendre polynomials from \eqref{eq:LegdnresigmaqapproxNNNsize}.

For ${\rm size}(f_{\Lambda}^{(1)})$ it holds
\begin{align}
{\rm size}\left(f_{\Lambda}^{(1)}\right)&=\sum_{\pnu\in\Lambda}{\rm size}\left(\Id_{\IR}\circ \widetilde{\prod}\right)\notag\\
&\leq C_q\sum_{\pnu\in\Lambda}\left({\rm size}(\Id_{\IR})+\left(\widetilde{\prod}\right)\right)\notag\\
&\leq C_q|\Lambda|\max_{\pnu\in\Lambda}\,{\rm size}\left(\widetilde{\prod}\right)\notag\\
&\leq C_q|\Lambda|\max_{\pnu\in\Lambda}\,|\supp\,\pnu|\notag\\
&\leq C_q|\Lambda|d(\Lambda).
\label{eq:prooftensorizedLegendresize2sigq}
\end{align}
In \eqref{eq:prooftensorizedLegendresize2sigq} we used the size bound for $\tilde{\prod}$ from \eqref{sigmaqmultinnunmberapproximationsize} and the argument from \eqref{eq:prooftensorizedLegendresizeextrasigq}. Combining \eqref{eq:prooftensorizedLegendresize1sigq} and \eqref{eq:prooftensorizedLegendresize2sigq} shows the size bound for $f_{\Lambda}$.

The network $f_{\Lambda}$ consists of sparse concatinations and parallelizations of the networks $\tilde{\prod}$ and $\tilde{L}_j$.
Because we have
${\rm mpar}(\tilde{\prod})\leq C_q$ and ${\rm mpar}(\tilde{L}_j)\leq C_q$, the NN calculus rules \eqref{eq:sigqsparseconcatweightbound} and \eqref{eq:parallelization2NNmaximumweights} yield ${\rm mpar}(f_{\Lambda})\leq C_q$. This finishes the proof.
\end{proof}
\subsection{Proof of Theorem \ref{th:approximationNN}}
\label{app:proofapproximationerror}
The following two theorems are similar to \cite[Theorem 5]{Herrmann.2024} and
will be required for the proof of
Theorem \ref {th:approximationNN}.
\begin{theorem}
    \label{th:aprroximationerrorLinfinity}
     For $N,q\in\IN$, consider the sparse FrameNet class $\GNNSP$. Let Assumption \ref{assump:holomorphicextension} be satisfied with $r>1$ and $t>0$. Fix $\tau>0$ (arbitrary small). Then there exists a constant $C>0$ independent of $N$, such that there exists $\Gamma_N\in\GNNSP$ with
    \begin{equation}
        \label{eq:approximationerrrorresultnew}
    \sup_{a\in C_R^r(\X)}\left\lVert\Gamma_N(a)-\G_0(a)\right\rVert_{\Y}\leq C N^{-\min\{r-1,t\}+\tau}.
    \end{equation}
    \end{theorem}
\begin{theorem}
    \label{th:aprroximationerrorl2}
    Consider the setting of Theorem \ref{th:aprroximationerrorLinfinity}. Let $\pmb{\Psi}_{\X}$ be a Riesz basis. Additionally, let $\gamma$ be as in \eqref{eq:gamma}. Fix $\tau>0$ (arbitrary small). Then there exists a constant $C>0$ independent of $N$, such that there exists $\Gamma_N\in\GNNSP$ with
    \begin{equation}
        \label{eq:approximationerrrorresultnew2}
    \left\lVert\Gamma_N-\G_0\right\rVert_{L^2(C_R^r(\X),\gamma;\Y)}\leq C N^{-\min\left\{r-\frac{1}{2},t\right\}+\tau}.
    \end{equation}
    \end{theorem}
We first show that Theorems \ref{th:aprroximationerrorLinfinity} and \ref{th:aprroximationerrorl2} imply Theorem \ref{th:approximationNN}.
\begin{proof}[Proof of Theorem \ref{th:approximationNN}]
First consider the setting of Theorem \ref{th:aprroximationerrorLinfinity}. Let $\tau>0$. Then there exists a constant $C$ independent of $N$ and a FrameNet $\Gamma_N\in\GNNSP$ such that for all $N\in\IN$
\begin{align}
    \label{eq:T11firstequation}
    \left\lVert \Gamma_N-\G_0\right\rVert_{\infty,\supp(\gamma)}^2\leq
    \sup_{a\in C_R^r(\X)}\left\lVert \Gamma_N(a)-\G_0(a)\right\rVert_{\Y}^2
      \leq C N^{-2\min\{r-1,t\}+\tau},
\end{align}
where we used \eqref{eq:approximationerrrorresultnew} with $\tau/2$ and $\supp(\gamma)\subseteq C_R^r(\X)$ by Assumption \ref{assump:holomorphicextension}. 

Now consider the setting of Theorem \ref{th:aprroximationerrorl2}. Let $\tau>0$. Then there exists a constant $C$ independent of $N$ and a FrameNet $\GNNSP$ with
\begin{align}
    \label{eq:T11firstequation2}
    \left\lVert \Gamma_N-\G_0\right\rVert_{L^2(\gamma)}^2\leq \left\lVert \Gamma_N-\G_0\right\rVert_{L^2(C_R^r(\X),\gamma;\Y)}^2
      \leq\,C N^{-2\min\left\{r-\frac{1}{2},t\right\}+\tau},
\end{align}
where we used $\supp(\gamma)\subset C_R^r(\X)$ (Assumption \ref{assump:holomorphicextension}) and \eqref{eq:approximationerrrorresultnew2} with $\tau/2$. 
\end{proof}

We are left to prove Theorems \ref{th:aprroximationerrorLinfinity} and \ref{th:aprroximationerrorl2}. We need some auxiliary results.
\subsubsection*{Auxiliary Results}
For $r>1$, $R>0$, $U=[-1,1]^{\IN}$ and $\sigma_R^r$ form \eqref{eq:sigmarR} we define
\begin{equation}
    \label{eq:u}
    u:U\to\Y,\quad u(\py)\coloneqq(\G_0\circ\sigma_R^r)(\py).
\end{equation}
For the proofs of Theorems \ref{th:aprroximationerrorLinfinity} and \ref{th:aprroximationerrorl2} we do a $\Y$-valued tensorized Legendre expansion of $u$ in the frame $(\eta_j L_{\pnu}(\py))_{j,\pnu}$ of $L^2(U,\pi;\Y)$, which reads
\begin{equation}
\label{eq:Legednreexpansionu}
u(\py)=\G_0(\sigma_R^r(\py))=\sum_{j\in\IN}\sum_{\pnu\in\F}c_{\pnu,j}\eta_jL_{\pnu}(\py)
\end{equation}
with Legendre coefficients
\begin{equation}
\label{eq:legendrecoefficnets}
c_{\pnu,j}\coloneqq\int_UL_{\pnu}(\py)\left\langle u(\py),\tilde{\eta}_j\right\rangle_{\Y}\dd\pi(\py).
\end{equation}
Our aim is to construct the network $\Gamma_N$ out of the tensorized Legendre polynomials with the ``most important'' contributions to the expansion. This contribution is quantified via the Legendre coefficients $c_{\pnu,j}$ in  \eqref{eq:legendrecoefficnets}. We therefore have to examine bounds on $c_{\pnu,j}$ and analyze their respective structure. Therefore consider the following order relation on multi-indices in $\F$ from \eqref{eq:multiindices}. For $\pmu,\pnu\in\F$ we write $\pmu\leq\pnu$ if and only if $\mu_j\leq\nu_j$ for all $j\in\IN$. We call a set $\Lambda\subset\F$ \emph{downward closed} if and only if $\pnu\in\F$ implies $\pmu\in\F$ for all $\pmu\leq\pnu$. Furthermore, for $\pnu\in\F$, define 
\[\omega_{\pnu}\coloneqq\prod_{j=1}^\infty(1+2\nu_j).\]
The following theorem is a special case of \cite[Theorem 2.2.10]{Zech.2018}. The formulation is similar to \cite[Theorem 4]{Herrmann.2024}.
\begin{theorem}
\label{th:algebraicboundonlegendrecoefficients}
Let Assumption \ref{assump:holomorphicextension} be satisfied with $r>1$ and $t>0$. Fix $\tau>0$, $p\in(\frac{1}{r},1]$ and $t'\in[0,t]$. Consider $\F$ from \eqref{eq:multiindices},
and let $\pi=\otimes_{j\in\N}\frac{\lambda}{2}$ be the infinite product (probability) measure on $U=[-1,1]^\N$, where $\lambda$ denotes the Lebesgue measure on $[-1,1]$.
Then there exists $C>0$ and a sequence $(a_{\pnu})_{\pmb{\nu}\in\F}\in l^p(\F)$ of positive numbers such that
    \begin{enumerate}[label=(\roman*)]
        \item\label{item:legendredecay} for each $\pnu\in\F$
        \begin{equation*}
            \omega_{\pnu}^\tau\left\lVert\int_U L_{\pnu}(\py)u(\py)\dd\pi(\py)\right\rVert_{\Y^{t'}}\leq Ca_{\pnu},       
            \end{equation*}
        \item there exists an enumeration $(\pmb{\nu_i})_{i\in\IN}$ of $\F$ such that $(a_{\pmb{\nu_i}})_{i\in\IN}$ is monotonically decreasing, the set $\Lambda_N\coloneqq\{\pmb{\nu_i}:\,i\leq N\}\subseteq\F$ is downward closed  for each $N\in\IN$, and additionally
        \begin{align}
            \label{eq:1normnu}
            m(\Lambda_N)&\coloneqq\max_{i=1,\dots, N}|\pmb{\nu_i}|=\O\left(\log(|\Lambda_N|)\right),\\
         d(\Lambda_N)&\coloneqq\max_{i=1,\dots,N}|\supp\pmb{\nu_i}|=o\left(\log(|\Lambda_N|)\right)\label{eq:suppnu}
        \end{align}
        for $N\to\infty$,
        \item the following expansion holds with absolute and uniform convergence:
        \begin{equation*}
            \forall\py\in U:\qquad u(\py)=\sum_{\pnu\in\F}L_{\pnu}(\py)\int_U L_{\pnu}(\px)u(\px)\dd\pi(\px)\in\Y^{t'}.
        \end{equation*}
    \end{enumerate}
\end{theorem}

The following proposition reformulates Theorem \ref{th:algebraicboundonlegendrecoefficients} \ref{item:legendredecay} into a bound for $c_{\pnu,j}$. It was shown in \cite[Proposition 2]{Herrmann.2024}. Recall that $\theta_j$ denote the weights to define the spaces $\Y^{t'}$, $t'>0$, see Definition \ref{def:smoothnessscales}.
\begin{proposition}[{\cite[Proposition 2]{Herrmann.2024}}]
\label{prop:cnujbound}
Consider the setting of Theorem \ref{th:algebraicboundonlegendrecoefficients}. Then for each $\pnu\in\F$ 
\begin{equation}
\label{eq:cnujbound}
\omega_{\pnu}^{2\tau}\sum_{j\in\IN}\theta_j^{-2t'}c_{\pnu,j}^2\leq C^2 a_{\pnu}^2.
\end{equation}
\end{proposition}

Proposition \ref{prop:cnujbound} gives decay of the coefficients $c_{\pnu,j}$ in both $j$ and $\pnu$. Since $\theta_j=\O(j^{-1+\tau}) $ for all $\tau>0$ we have $c_{\pnu,j}^2=\O(j^{-1-2t'+\tilde{\tau}})$ for $\tilde{\tau}<2\tau t'$ and every $\pnu\in\Lambda_N$.  Furthermore, since $(a_{\pnu})_{\pmb{\nu}\in\F}\in l^p(\F)$ the Legendre coefficients $c_{\pnu,j}$ decay algebraically in $\pnu$. We continue with a technical lemma, which was shown in \cite[Lemma 4]{Herrmann.2024}.

\begin{lemma}[{\cite[Lemma 4]{Herrmann.2024}}]
\label{eq:summabilitylemma}
Let $\alpha>1$, $\beta>0$ and assume two sequences $(a_i)_{i\in\IN}$ and $(d_j)_{j\in\IN}$ in $\IR$ with $a_i\lesssim i^{-\alpha}$ and $d_j\lesssim j^{-\beta}$ for all $i,j\in\IN$. Additionally assume that $(d_j)_{j\in\IN}$is monotonically decreasing. Suppose that there exists a constant $C<\infty$ such that the sequence $(c_{i,j})_{i,j\in\IN}$ satisfies
\begin{equation}
    \label{eq:combinedsummability}
    \forall i\in\IN:\quad \sum_{j\in\IN}c_{i,j}^2d_j^{-2}\leq C^2a_i^2.
\end{equation}
Then for every $\tau>0$
\begin{enumerate}[label=(\roman*)]
    \item\label{item:doublesum1} for all $N\in\IN$ there exists $(m_i)_{i\in\IN}\subseteq\IN_0^{\IN}$ monotonically decreasing such that $\sum_{i\in\IN} m_i\leq N$ and 
    \begin{equation}
        \label{eq:midecrease1}
        \sum_{i\in\IN}\left(\sum_{j>m_i}c_{i,j}^2\right)^{\frac{1}{2}}\lesssim N^{-\min\{\alpha-1,\beta\}+\tau},
    \end{equation}
    \item\label{item:doublesum2} for all $N\in\IN$ there exists $(m_i)_{i\in\IN}\subseteq\IN_0^{\IN}$ monotonically decreasing such that $\sum_{i\in\IN} m_i\leq N$ and 
        \begin{equation}
        \label{eq:midecrease2}
        \left(\sum_{i\in\IN}\sum_{j>m_i}c_{i,j}^2\right)^{\frac{1}{2}}\lesssim N^{-\min\left\{\alpha-\frac{1}{2},\beta\right\}+\tau}.
    \end{equation}
\end{enumerate}
\end{lemma}

In the following, we use Lemma \ref{eq:summabilitylemma} to get a decay property for the Legendre coefficients $c_{\pmb{\nu_i},j}$ with the enumeration $\pmb{\nu_i}$ of $\Lambda_N$ from Theorem \ref{th:algebraicboundonlegendrecoefficients}. The sequence $\pmb{m}=(m_i)_{i\in\IN}$ quantifies which coefficients of the Legendre expansion are ``important'' and are therefore used to define the surrogate $\Gamma_N$.

We first show that Theorem \ref{th:algebraicboundonlegendrecoefficients} yields sufficient decay on the Legendre coefficients $c_{\pmb{\nu_i},j}$ s.t.~the assumptions of Lemma \ref{eq:summabilitylemma} are satisfied.
\begin{lemma}\label{lem:decay}
    Consider the setting of Theorem \ref{th:algebraicboundonlegendrecoefficients}. Let $\tilde{\tau}>0$ such that $1/p>r-\tilde{\tau}/2$. Then the assumptions of Lemma \ref{eq:summabilitylemma} are fulfilled for $\alpha=r-\tilde{\tau}/2$, $\beta=t-\tilde{\tau}/2$, $a_i=a_{\pmb{\nu_i}}$,  $d_j=\theta_j^{t'}$ and $c_{i,j}=\omega_{\pmb{\nu_i}}^{1/2}c_{\pmb{\nu_i},j}$ for $i,j\in\IN$.
\end{lemma}
\begin{proof}
Proposition \ref{prop:cnujbound} with $\tau=\frac{1}{2}$ gives 
\begin{equation}
\label{eq:firstsummation}
    \omega_{\pmb{\nu_i}}^{\frac{1}{2}}\left(\sum_{j\in\IN}\theta_j^{-2t'}c_{\pmb{\nu_i},j}^2\right)^{\frac{1}{2}}=\O(a_{\pmb{\nu_i}})=\O\left(i^{-r+\frac{\tilde{\tau}}{2}}\right).
\end{equation}
The last equality in \eqref{eq:firstsummation} holds because $ia_{\pmb{\nu_i}}^p\leq\sum_{j\in\IN}a_{\pmb{\nu_j}}^p<\infty$ (since $a_{\pmb{\nu_i}}$ is monotonically decreasing) implies $a_{\pmb{\nu_i}}=\O(i^{-1/p})=\O(i^{-r+\tilde{\tau}/2})$.
Since $(\theta_j^{t'})_{j\in\IN}\in l^{1/(t-\tilde{\tau}/2)}$ (see Definition \ref{def:smoothnessscales}) it holds
\begin{equation}
\label{eq:secondsummation}
    \theta_j^{t'}=\O(j^{-t+\tilde{\tau}/2})
\end{equation}
with the same argument.
    \end{proof}
\subsubsection*{Proofs of Theorems \ref{th:aprroximationerrorLinfinity} and \ref{th:aprroximationerrorl2}}
The proof of Theorems \ref{th:aprroximationerrorLinfinity} and \ref{th:aprroximationerrorl2} is similar to \cite[Sections 4.2-4.4, Proofs of Theorems 1,2 and 5]{Herrmann.2024}.
\begin{proof}[Proof of Theorem \ref{th:aprroximationerrorLinfinity}]
Let $(a_{\pnu})_{\pnu\in\F}$ be the enumeration $(\pmb{\nu_i})_{i\in\IN}$ from Theorem \ref{th:algebraicboundonlegendrecoefficients}, where we use the case $\tau=\frac{1}{2}$. Therefore $(a_{\pmb{\nu_i}})_{i\in\IN}$ is monotonically decreasing and belongs to $l^p$ with $p\in(\frac{1}{r},1]$. We further fix $\tilde{\tau} >0$ and demand $\frac{1}{p}> r-\frac{\tilde{\tau}}{2}$. Fix $\tilde{N}\in\IN$ and set $\Lambda_{\tilde{N}}\coloneqq\{\pmb{\nu_j}:\,j\leq \tilde{N}\}\subset\F$, which is downward closed by Theorem \ref{th:algebraicboundonlegendrecoefficients}. Now we approximate the tensorized Legendre polynomials $L_{\pnu}$ on the index set $\Lambda_{\tilde{N}}$. Let $\rho\in (0,\frac{1}{2})$. In the ReLU case, Proposition \ref{prop:tensorlegendreapproximation} gives a NN $f_{\Lambda_{\tilde{N}},\rho}$ with outputs $\{\tilde{L}_{\pnu,\rho}\}_{\pnu\in\Lambda_{\tilde{N}}}$ s.t.
\begin{equation*}
    \sup_{\py\in U}\max_{\pnu\in\Lambda_{\tilde{N}}}\left|L_{\pnu}(\py)-\tilde{L}_{\pnu,\rho}(\py)\right|\leq\rho.
\end{equation*}
Using $|\Lambda_{\tilde{N}}|=\tilde{N}$, \eqref{eq:1normnu} and \eqref{eq:suppnu}, it holds for $\tilde{N}\geq 2$
\begin{align*}
    {\rm depth}(f_{\Lambda_{\tilde{N}},\rho})&=\O\left(\log(\tilde{N})^2\log(\log(\tilde{N}))^2+\log(\tilde{N})\log\left(\rho^{-1}\right)\right),\\
    {\rm width}(f_{\Lambda_{\tilde{N}},\rho})&=\O(\tilde{N}\log(\tilde{N})),
 \\
    {\rm size}(f_{\Lambda_{\tilde{N}},\rho})&=\O\left(\tilde{N}\log(\tilde{N})^2\log(\log(\tilde{N}))+\tilde{N}\log(\tilde{N})\log\left(\rho^{-1}\right)\right) ,\\
    {\rm mpar}(f_{\Lambda_{\tilde{N}},\rho})&=1.
\end{align*}
The constants hidden in $\O(\,.\,)$ are independent of $\tilde{N}$ and $\rho$. For $\tilde{N}\in\IN$, set the accuracy $\rho\coloneqq \tilde{N}^{-\min\{r-\frac{1}{2},t\}}$.
Then it holds
\begin{align*}
    {\rm depth}(f_{\Lambda_{\tilde{N}},\rho})&=\O\left(\log(\tilde{N})^2\log(\log(\tilde{N}))^2\right),
  \\
    {\rm size}(f_{\Lambda_{\tilde{N}},\rho})&=\O\left(\tilde{N}\log(\tilde{N})^2\log(\log(\tilde{N}))\right).
\end{align*}
Proposition \ref{prop:tensorlegendreapproximationsigq} shows that the ReLU bounds also hold for the RePU-case.

By Lemma \ref{lem:decay} the assumptions of Lemma \ref{eq:summabilitylemma} are satisfied.  Applying Lemma \ref{eq:summabilitylemma} \ref{item:doublesum1} with $\alpha\coloneqq r-\tilde{\tau}/2$ and $\beta\coloneqq t-\tilde{\tau}/2$ gives a sequence $(m_i)_{i\in\IN}\subset\IN_0^{\IN}$ such that $\sum_{i\in\IN}m_i\leq \tilde{N}$ and 
\begin{equation}
    \label{eq:thirdsummation}
    \sum_{i\in\IN}\omega_{\pmb{\nu_i}}^{\frac{1}{2}}\left(\sum_{j>m_i}c_{\pmb{\nu_i},j}^2\right)^{\frac{1}{2}}\leq C \tilde{N}^{-\min\{r-1,t\}+\tilde{\tau}}.
\end{equation}
We now define 
\begin{equation}
\label{eq:definegntilde}
\left(\tilde{\gamma}_{{\tilde{N}},j}\right)\coloneqq \sum_{\{i\in\IN: m_i\geq j\}}\tilde{L}_{\pmb{\nu_i},\rho}(\py)c_{\pmb{\nu_i},j}
\end{equation}
for $j\in\IN$, where empty sums are set to zero.
Recall the uniform distribution $\pi$ on $U=[-1,1]^{\IN}$ (Example \ref{ex:pi}). With $\tilde{\gamma}_{\tilde{N}}=(\tilde{\gamma}_{\tilde{N},j})_{j\in\IN}$ it holds
\begin{align}
    &\lVert\G_0\circ\sigma_R^r(\py)-\D_{\Y}\circ\tilde{\gamma}_{\tilde{N}}(\py)\rVert_{\Y}=\left\lVert
    \sum_{i,j\in\IN}c_{\pmb{\nu_i},j}L_{\pmb{\nu_i}}(\py)\eta_j-\sum_{i\in\IN}\sum_{j\leq m_i}c_{\pmb{\nu_i},j}\tilde{L}_{\pmb{\nu_i},\rho}(\py)\eta_j\right\rVert_{\Y}\notag\\
   &\leq\left\lVert\sum_{i\in\IN}L_{\pmb{\nu_i}}(\py)\sum_{j>m_i}c_{\pmb{\nu_i},j}\eta_j\right\rVert_{\Y}+\left\lVert\sum_{i\in\IN}(L_{\pmb{\nu_i}}(\py)-\tilde{L}_{\pmb{\nu_i},\rho}(\py))\sum_{j\leq m_i}c_{\pmb{\nu_i},j}\eta_j\right\rVert_{\Y}\notag\\
    &\leq \Lambda_{\pmb{\Psi}_{\Y}}\sum_{i\in\IN}\underbrace{\lVert L_{\pmb{\nu_i}}\rVert_{\infty,\pi}}_{\leq\omega_{\pmb{\nu_i}}^{\frac{1}{2}}}\left(\sum_{j>m_i}c_{\pmb{\nu_i},j}^2\right)^{\frac{1}{2}}+\Lambda_{\pmb{\Psi}_{\Y}}\rho\sum_{i\in\IN}\left(\sum_{j\leq m_i}c_{\pmb{\nu_i},j}^2\right)^{\frac{1}{2}}\notag\\
    &\leq \tilde{C}\Lambda_{\pmb{\Psi}_{\Y}}\tilde{N}^{-\min\{r-1,t\}+\tilde{\tau}}+\tilde{C}\Lambda_{\pmb{\Psi}_{\Y}}\rho\leq \tilde{C} \tilde{N}^{-\min\{r-1,t\}+\tilde{\tau}}\label{eq:proofapproximationerrorfinalequation}
\end{align}
for all $\py\in U$.
In \eqref{eq:proofapproximationerrorfinalequation} we used the definition of $\D_{\Y}$, \eqref{eq:definegntilde} and \eqref{eq:Legednreexpansionu} at the first equality. Furthermore, we used \eqref{eq:thirdsummation}, the definition of $\rho$ and 
\begin{align}
    \sum_{i\in\IN}\left(\sum_{j\leq m_i}c_{\pmb{\nu_i},j}^2\right)^{\frac{1}{2}}&\leq  \sum_{i\in\IN}\left(\sum_{j\in\IN}c_{\pmb{\nu_i},j}^2\right)^{\frac{1}{2}}\notag\\
    &\leq \tilde{C} \sum_{i\in\IN}\omega_{\pmb{\nu_i}}^{\frac{1}{2}}\left(\sum_{j\in\IN}\theta_j^{-2t'}c_{\pmb{\nu_i},j}^2\right)^{\frac{1}{2}}\notag\\
    &\leq \tilde{C} \sum_{i\in\IN}a_{\pmb{\nu_i}}\leq \tilde{C}\sum_{i\in\IN}i^{-{r+\tilde{\tau}/2}}\leq \tilde{C}\label{eq:cnuijbaschätzung}
\end{align}
at the second-to-last inequality.
We changed the constants $\tilde{C}$ from line to line in \eqref{eq:proofapproximationerrorfinalequation} and \eqref{eq:cnuijbaschätzung}. 
The last line of \eqref{eq:proofapproximationerrorfinalequation} shows why the RePU-case does not improve the approximation property qualitatively. In the RePU-case, Proposition \ref{prop:tensorlegendreapproximationsigq} gives a $\sigma_q$-NN $f_{\Lambda}$ exactly realizing the tensorized Legendre polynomials, i.e.~the case $\rho=0$ from above. Therefore the second summand in the last line of \eqref{eq:proofapproximationerrorfinalequation} vanishes. This does not improve the approximation rate due to the first summand. This part depends on the summability properties of the Legendre coefficients $c_{\pmb{\nu_i},j}$ following Assumption \ref{assump:holomorphicextension} and is therefore independent of the activation function $\sigma$.

Now we argue similar to \cite[Proof of Theorem 1]{Herrmann.2024}. Consider the scaling $S_r$ from \eqref{eq:scaling}. 
It holds
\begin{equation}
\label{eq:scalingproperty}
     S_r\circ\E_{\X}(a)\in U\qquad\forall a \in C_R^r(\X),
\end{equation}
because of \eqref{eq:CrR} and \eqref{eq:scaling}.
We define $\tilde{\Gamma}_{\tilde{N}}\coloneqq\D_{\Y}\circ \tilde{\gamma}_{\tilde{N}}\circ S_r\circ\E_\X$ and calculate
\begin{align}
    \label{eq:prooapproximationscalingargument}
    &\sup_{a\in C_R^r(\X)}\left\lVert \G_0(a)-\tilde{\Gamma}_{\tilde{N}}(a)\right\rVert_{\Y}=\sup_{a\in C_R^r(\X)}\left\lVert \G_0(a)-\D_{\Y}\circ \tilde{\gamma}_{\tilde{N}}\circ  S_r\circ\E_{\X}(a)\right\rVert_{\Y}\notag\\
=&\sup_{\{\py\in U:\,\sigma_R^r(\py)\in C_R^r(\X)\}}\left\lVert\G_0\circ\sigma_R^r(\py)-\D_{\Y}\circ\tilde{\gamma}_{\tilde{N}}\circ  S_r\circ\E_{\X}\circ\sigma_R^r(\py)\right\rVert_{\Y}\notag\\
\leq &\sup_{\py\in U}\left\lVert\G_0\circ\sigma_R^r(\py)-\D_{\Y}\circ\tilde{\gamma}_{\tilde{N}}(\py)\right\rVert_{\Y}\leq \tilde{C}\tilde{N}^{-\min\{r-1,t\}+\tilde{\tau}},
\end{align}
where we used \eqref{eq:scalingproperty} and \eqref{eq:proofapproximationerrorfinalequation}.

In order to finish the proof of Theorem \ref{th:aprroximationerrorLinfinity}, we relate $\tilde{N}$ to $N$ and show $\Gamma_{N}\coloneqq\tilde{\Gamma}_{\tilde N}\in\GNNSP$, i.e.~we show that the approximation networks we constructed have the desired sparse structure. We simultaneously prove the ReLU- and RePU-case.

In order to analyse the NNs $\tilde{\gamma}_{\tilde{N}}$ from \eqref{eq:definegntilde}, we specify its structure.
We set $n_j=|\{m_i\geq j\}|$ and define
\begin{align}
\label{eq:gntildeNN}
    \tilde{\gamma}_{\tilde{N}}=\left(\left\{\Sigma_{n_j}\left(\left\{\Id_{\IR}\circ SM_{c_{\pmb{\nu_i},j}}\circ \tilde{L}_{\pmb{\nu_i},\rho}\right\}_{m_i\geq j}\right)\right\}_{j\in\IN}\right).
\end{align}
The round brackets in \eqref{eq:gntildeNN} denote a parallelization. The networks $SM_{c_{\pmb{\nu_i},j}}$ denote the scalar multiplication networks from Definition \ref{def:scalarmultiplicatoinsimaq}. Furthermore, $\Sigma_{n_j}$ denotes the summation network from Definition \ref{def:summationnetworks} and we use the identity networks $\Id_{\IR}$ from Lemma \ref{def:Identityemulationsigma1} or \ref{def:Identityemulationsigmaq} to synchronize the depth. Using the respective bounds for the summation and scalar multiplication networks and the NN calculus for parallelization and sparse concatenation we get
\begin{align}
    {\rm depth}(\tilde{\gamma}_{\tilde{N}})&\leq 2+\max_{i,j\in\IN^2,\,m_i\geq j}\left({\rm depth}\left(\tilde{L}_{\pmb{\nu_i},\rho}\right)+{\rm depth}\left(SM_{c_{\pmb{\nu_i},j}}\right)\right)+\max_{j\in\IN}{\rm depth}\left(\Sigma_{n_j}\right)\notag\\
    &\leq 3+\O(\log(\tilde{N})^2\log(\log(\tilde{N})))+\max_{i,j\in\IN^2,\,m_i\geq j}C_q\log\left(|c_{\pmb{\nu_i},j}|\right)+0\label{eq:gntildedepth}\notag\\
    &=\O(\log(\tilde{N})^2\log(\log(\tilde{N}))),\\
    {\rm width}(\tilde{\gamma}_{\tilde{N}})&\leq\max\bigggl\{\sum_{j\in\IN}\sum_{m_i\geq j}{\rm width}\left(\tilde{L}_{\pmb{\nu_i},\rho}\right),\sum_{j\in\IN}\sum_{m_i\geq j}{\rm width}\left(SM_{c_{\pmb{\nu_i},j}}\right),\notag\\
    &\qquad\sum_{j\in\IN}\sum_{m_i\geq j}{\rm width}\left(\Id_{\IR}\right),\sum_{j\in\IN} {\rm width}\left(\Sigma_{n_j}\right)\bigggr\}\notag\\
    &\leq C_q\max\left\{O(\tilde{N}\log(\tilde{N})),\sum_{j\in\IN}\sum_{m_i\geq j}C_q,\sum_{j\in\IN}n_j\right\}
    =O(\tilde{N}\log(\tilde{N}))\label{eq:gntildewidth},\\
    {\rm size}(\tilde{\gamma}_{\tilde{N}})&\leq C_q\sum_{j\in\IN}\sum_{m_i\geq j}\left({\rm size}\left(\tilde{L}_{\pmb{\nu_i},\rho}\right)+{\rm size}\left(SM_{c_{\pmb{\nu_i},j}}\right)+{\rm size}\left(\Id_{\IR}\right)\right)+C_q\sum_{j\in\IN}{\rm size}\left(\Sigma_{n_j}\right)\notag\\
    &=\O\left(\tilde{N}\log(\tilde{N})^2\log(\log(\tilde{N}))\right)+
    \sum_{j\in\IN}\left[\left(\,\sum_{m_i\geq j}\,C_q\log\left(|c_{\pmb{\nu_i},j}|\right)\right)+C_qn_j\right]\notag\\
    &\leq \O\left(\tilde{N}\log(\tilde{N})^2\log(\log(\tilde{N}))\right)+C_q\sum_{j\in\IN}\sum_{m_i\geq j}1\notag\\
    &=\O\left(\tilde{N}\log(\tilde{N})^2\log(\log(\tilde{N}))\right)\label{eq:gntildesize},\\
     {\rm mpar}(\tilde{\gamma}_{\tilde{N}})&\leq C_q. \label{eq:gntildeweightbound}
\end{align}
In \eqref{eq:gntildedepth}--\eqref{eq:gntildesize} we used $\log\left(|c_{\pmb{\nu_i},j}|\right)\leq C_q$ for all $i,j\in\IN$ independent of $n$. Furthermore, we used
\begin{equation*}
\sum_{j\in\IN}n_j=\sum_{j\in\IN}\sum_{m_i\geq j}1=\sum_{i\in\IN}\sum_{j\leq m_i}1=\sum_{i\in\IN}m_i\leq \tilde{N}.
\end{equation*}
To get rid of the logarithmic terms, we define $N=N(\tilde{N})\coloneqq \max\{1,\tilde{N}\log(\tilde{N})^3\}$ and obtain a NN $\gamma_N=\tilde{\gamma}_{\tilde N}$ with 
\begin{align}
    {\rm depth}(\gamma_N)&=\O(\log(N)),\notag\\
    {\rm width}(\gamma_N)&=\O(N),\notag\\
    {\rm size}(\gamma_N)&=N,\notag\\
    {\rm mpar}(\gamma_N)&\leq C_q\label{eq:gshatbounds}
\end{align}
and error less than 
\begin{align}\label{eq:error2}
\tilde{C}\tilde{N}^{-\min\{r-1,t\}+\tilde{\tau}}=\tilde{C}\tilde{N}^{-\kappa}
\leq \tilde{C}(3\kappa/\tilde{\tau})^{3\kappa} N^{-\kappa+\tilde{\tau}}\coloneqq CN^{-\min\{r-1,t\}+\tau}.
\end{align}
Per definition ${\rm depth}(\gamma_N)=\O(\log(N))$ and ${\rm width}(\gamma_N)=\O(N)$ yields constants $\tilde{C}_L$, $\tilde{C}_p$ and $N_1,N_2\in\IN$ s.t.
\begin{align*}
    {\rm depth}(\gamma_N)&\leq \tilde{C}_L\max\{1,\log(N)\},\qquad N\geq N_1,\\
    {\rm width}(\gamma_N)&\leq \tilde{C}_p N,\qquad N\geq N_2.
\end{align*}
Setting $C_L=\max\{\tilde{C}_L,\max_{N=2,\dots,N_1-1} {\rm depth}(\gamma_N)/\log(2)\}$ and

$C_p=\max\{\tilde{C}_p,\max_{N=1,\dots,N_2-1}{\rm width}(\gamma_N)\}$ shows
\begin{align*}
    {\rm depth}(\gamma_N)\leq C_L\log(N),\qquad N\in\IN,\\
    {\rm width}(\gamma_N)\leq C_p N,\qquad N\in\IN.
\end{align*}
In order to show $\Gamma_N\coloneqq \D_{\Y}\circ \gamma_{N}\circ  S_r\circ\E_{\X}\in\GNNSP$, we are left to show that the maximum Euclidean norm $\|\cdot\|_2$ of $\gamma_N$ in $U$ is independent of $N$. It holds for all $\py\in U$ that $\D_{\X}\circ S_r^{-1}(\py)\in C_R^r(\X)$. We get 
\begin{align}
     \label{eq:l2normgammas}
     \sup_{\py\in U}\,\lVert\gamma_N(\py)\rVert_{2}&=\sup_{\py\in U}\,\lVert\E_{\Y}\circ\Gamma_N\circ\D_{\X}\circ S_r^{-1}(\py)\rVert_{2}\notag \\
     &\leq\Lambda_{\pmb{\Psi}_{\Y}}\,\sup_{a\in C_R^r(\X)}\,\lVert\Gamma_N(a)\rVert_{\Y}\notag\\
    &=\Lambda_{\pmb{\Psi}_{\Y}}\,\sup_{a\in C_R^r(\X)}\,\lVert\Gamma_N(a)-\G_0(a)+\G_0(a)\rVert_{\Y}\notag\\
     &\leq \Lambda_{\pmb{\Psi}_{\Y}}\,\sup_{a\in C_R^r(\X)}\,\left(\lVert\Gamma_N(a)-\G_0(a)\rVert_{\Y}+c\lVert\G_0(a)\rVert_{\Y^t}\right)\notag\\
     &\leq  \Lambda_{\pmb{\Psi}_{\Y}}(C+cC_{G_0})\eqqcolon B,
\end{align}
where $\Lambda_{\pmb{\Psi}_{\Y}}$ denotes the upper frame bound of $\pmb{\Psi}_{\Y}$ and $c=\theta_0^t$, see Definition \ref{def:smoothnessscales}.
In \eqref{eq:l2normgammas} we used Assumption \ref{assump:holomorphicextension} and the approximation error from \eqref{eq:error2}.
Thus $\Gamma_N\in\GNNSP$ for all $N\in\IN$, where we set $C_s=1$. Using $\supp(\gamma)\subset C_R^r(\X)$ (Assumption \ref{assump:holomorphicextension}) in \eqref{eq:prooapproximationscalingargument}  finalizes the proof of Theorem \ref{th:aprroximationerrorLinfinity}. 
\end{proof}

\begin{proof}[Proof of Theorem \ref{th:aprroximationerrorl2}] 
By Lemma \ref{lem:decay} the assumptions of Lemma \ref{eq:summabilitylemma} are satisfied.  Applying Lemma \ref{eq:summabilitylemma} \ref{item:doublesum2} with $\alpha\coloneqq r-\tilde{\tau}/2$ and $\beta\coloneqq t-\tilde{\tau}/2$ gives a sequence $(m_i)_{i\in\IN}\subset\IN_0^{\IN}$ such that $\sum_{i\in\IN}m_i\leq \tilde{N}$ and 
\begin{equation}
    \label{eq:thirdsummation2}
   \left( \sum_{i\in\IN}\omega_{\pnu_i}\sum_{j>m_i}c_{\pmb{\nu_i},j}^2\right)^{\frac{1}{2}}\leq \tilde C \tilde{N}^{-\min\left\{r-\frac{1}{2},t\right\}+\tilde{\tau}}.
\end{equation}
Define  $\tilde{\gamma}_{\tilde{N}}=(\tilde{\gamma}_{\tilde{N},j})_{j\in\IN}$ for all $y\in U$ with $\tilde{\gamma}_{\tilde{N},j}$ as in \eqref{eq:definegntilde}. Then it holds
\begin{align}
    &\lVert\G_0\circ\sigma_R^r-\D_{\Y}\circ\tilde{\gamma}_{\tilde{N}}\rVert_{L^2(U,\pi;\Y)}\leq\left\lVert
    \sum_{i\in\IN}\sum_{j>m_i}c_{\pmb{\nu_i},j}L_{\pmb{\nu_i}}\eta_j\right\rVert_{L^2(U,\pi;\Y)}\notag\\
    &\quad+\left\lVert\sum_{i\in\IN}\sum_{j\leq m_i}c_{\pmb{\nu_i},j}\eta_j\left(L_{\pmb{\nu_i}}-\tilde{L}_{\pmb{\nu_i},\rho}\right)\right\rVert_{L^2(U,\pi;\Y)}\notag\\
    &\leq\Lambda_{\pmb{\Psi}_{\Y}}\left(\sum_{i\in\IN}\underbrace{\lVert L_{\pmb{\nu_i}}\rVert_{\infty,\pi}^2}_{\leq\omega_{\pmb{\nu_i}}}\sum_{j>m_i}c_{\pmb{\nu_i},j}^2\right)^{\frac{1}{2}}+\Lambda_{\pmb{\Psi}_{\Y}}\rho\left(\sum_{i\in\IN}\sum_{j\leq m_i}c_{\pmb{\nu_i},j}^2\right)^{\frac{1}{2}}\notag\\
    &\leq \tilde{C}\Lambda_{\pmb{\Psi}_{\Y}}\tilde{N}^{-\min\{r-\frac{1}{2},t\}+\tilde{\tau}}+\tilde{C}\Lambda_{\pmb{\Psi}_{\Y}}\rho\leq \tilde{C} \tilde{N}^{-\min\{r-\frac{1}{2},t\}+\tilde{\tau}}.\label{eq:proofapproximationerrorfinalequation2}
\end{align}
In \eqref{eq:proofapproximationerrorfinalequation2} we used the definition of $\D_{\Y}$, \eqref{eq:definegntilde} and \eqref{eq:Legednreexpansionu} at the first inequality. Additionally we used that $(L_{\pnu}\eta_j)_{\pnu,j}$ is a frame of $L^2(U,\pi;\Y)$ at the second inequality. Finally we used \eqref{eq:thirdsummation2}, the definition of $\rho$ and an argument similar to \eqref{eq:cnuijbaschätzung}
at the second-to-last inequality. Note that again we changed the constants $\tilde{C}$ from line to line in \eqref{eq:proofapproximationerrorfinalequation2}.

Since $\pmb{\Psi}_{\X}$ is a Riesz basis, we have (see Section \ref{sec:decoderencoder} and \eqref{eq:scaling})
\begin{align}
\label{eq:L2arguments}
    C_R^r(\X)=\left\{\sigma_R^r(\py),\,\py\in U\right\}\quad{\rm and}\quad \E_{\X}\circ\sigma_R^r(\py)=S_r^{-1}(\py).
\end{align}
With $\tilde{\Gamma}_{\tilde{N}}\coloneqq\D_{\Y}\circ \tilde{\gamma}_{\tilde{N}}\circ  S_r\circ\E_{\X}$ we calculate
\begin{align*}
 \left\lVert\tilde{\Gamma}_{\tilde{N}}-\G_0\right\rVert_{L^2(C_R^r(\X),(\sigma_R^r)_{\#}\pi;\Y)}&=
 \left\lVert\D_{\Y}\circ \tilde{\gamma}_{\tilde{N}}\circ  S_r\circ\E_{\X}-\G_0\right\rVert_{L^2(C_R^r(\X),(\sigma_R^r)_{\#}\pi;\Y)}\notag\\
      &=\left\lVert\D_{\Y}\circ \tilde{\gamma}_{\tilde{N}}\circ  S_r\circ\E_{\X}\circ\sigma_R^r-\G_0\circ\sigma_R^r\right\rVert_{L^2(U,\pi;\Y)}\notag\\
      &=\left\lVert\D_{\Y}\circ \tilde{\gamma}_{\tilde{N}}-\G_0\circ\sigma_R^r\right\rVert_{L^2(U,\pi;\Y)}
      \leq\tilde{C} \tilde{N}^{-\min\left\{r-\frac{1}{2},t\right\}+\tau},
\end{align*}
where we used \eqref{eq:L2arguments} and \eqref{eq:proofapproximationerrorfinalequation2}. Defining $N=N(\tilde{N})\coloneqq \max\{1,\tilde{N}\log(\tilde{N})^3\}$ we can proceed similar to the proof of Theorem \ref{th:aprroximationerrorLinfinity} from \eqref{eq:prooapproximationscalingargument} on. The reason this works is that the NNs $\tilde{\gamma}_{\tilde N}$ are defined in the same way in the $L^2$- and the $L^{\infty}$-case (only the sequence $\pmb{m}$ changes, but not its properties). This shows $\Gamma_N\coloneqq \tilde{\Gamma}_{\tilde N}\in\GNNSP$ for all $N\in\IN$ and thus finishes the proof of Theorem \ref{th:aprroximationerrorl2}.\end{proof}

\subsection{Proof of Lemma \ref{lem:entropybound}}\label{app:proofentropy}
The arguments in the following proof are based on entropy bounds for feedforward neural network classes, first established in \cite[Proof of Lemma 5]{SH20}. 

Define the supremum norm
$\|\cdot\|_{\infty,\infty}$ on $\pg_{{\rm FN}}$ as
\begin{align}\label{eq:NNinftynorm}
   \lVert g\rVert_{\infty,\infty} \coloneqq \sup_{\py\in\IR^{p_0}}\,\lVert g(\py)\rVert_{\infty},\quad g\in\pg_{{\rm FN}},
\end{align}
where $\|\cdot\|_\infty$ denotes the maximum norm in $\IR^n$.
Then \cite[Proposition 3.5]{Petersen.2021} shows that $(\pg_{{\rm FN}},\lVert \cdot \rVert_{\infty,\infty})$ is compact. Since the map $i:\,\pg_{\rm FN}\to\pG_{\rm FN} $, $g\to G=\D_{\Y}\circ g\circ \E_{\X}$ is linear, also $(\pG_{{\rm FN}},\lVert \cdot \rVert_{\infty,\supp(\gamma)})$ and hence $(\pG_{{\rm FN}},\lVert \cdot \rVert_{n})$ is compact.
We now show the entropy bounds for $\pG_{\rm FN}$. 

  {\bf Step 1.} Recall ${\rm depth}(g)\leq L$,
${\rm depth}(g)\leq p$,
${\rm size}(g)\leq s$ and ${\rm mpar}(g)\leq M$ for $g\in\pg_{\rm FN}$.
We first estimate the entropy $H(\pG_{{\rm FN}},\|\cdot\|_{\infty,\supp(\gamma)},\delta)$ against the respective entropy of $\pg_{{\rm FN}}$. For $G,G'\in\pG_{{\rm FN}}$ and $g,g'\in \pG_{{\rm FN}}$ with $G=\D_{\Y}\circ g\circ S_r\circ\E_{\X}$, $G'=\D_{\Y}\circ g'\circ S_r\circ\E_{\X}$, it holds
\begin{align}
    \lVert G-G'\rVert_{\infty,\sigma_R^r(U)}&=\sup_{x\in\sigma_R^r(U)}\lVert \D_{\Y}\circ g\circ S_r\circ\E_{\X}(x)-\D_{\Y}\circ g'\circ S_r\circ\E_{\X}(x)\rVert_{\Y}\notag\\
    &\leq \Lambda_{\pmb{\Psi}_{\Y}}\sup_{\py\in U}\lVert g(\py)-g'(\py)\rVert_{2}\leq\Lambda_{\pmb{\Psi}_{\Y}} \sqrt p\lVert g-g'\rVert_{\infty,\infty}, \label{eq:entropylemma1}
\end{align}
where we used $\sigma_R^r=\D_{\X}\circ S_r^{-1}$
and $\|\cdot\|_{\infty,\infty} $ from \eqref{eq:NNinftynorm}.
Furthermore, we used $\lVert g(u)\rVert_{2}\leq \sqrt{p}\lVert g\rVert_{\infty}$ for all $g\in\pg_{{\rm FN}}$, since NNs $g\in\pg_{{\rm FN}}$ have ${\rm width}(g)\leq p$.
Then \eqref{eq:entropylemma1} yields
\begin{align}\label{eq:gNNGNNrelation}
H(\pG_{{\rm FN}},\|\cdot\|_{\infty,\sigma_R^r(U)},\delta)\leq H\left(\pg_{{\rm FN}},\|\cdot\|_{\infty,\infty} ,\frac{\delta}{\Lambda_{\pmb{\Psi}_{\Y}}\sqrt p}\right).
\end{align}

{\bf Step 2.} It remains to bound
  $H(\pg_{{\rm FN}},\|\cdot\|_{\infty,\infty} ,\delta)=\log(N(\pg_{{\rm FN}},\|\cdot\|_{\infty,\infty} ,\delta))$.
To this end we follow the proof and notation of \cite[Lemma 5]{SH20}.
For $l=1,\dots L+1$, define the matrices $W_l=(w_{i,j}^l)_{i,j}\in \R^{p_{l-1}\times p_{l}}$ and the vectors $B_l=(b_j^l)_j\in \R^{p_{l}}$.
Furthermore, define
\begin{align}\label{eq:sigmabl}
  \sigma^{B_l}&: \IR^{p_l}\to\IR^{p_l},\quad \sigma^{B_l}(x)=\sigma_1(x+B_l)=\max\{0,x+B_l\},\quad l=1,\dots L,\notag\\
\sigma^{B_{L+1}}&: \IR^{p_{L+1}}\to\IR^{p_{L+1}},\quad \sigma^{B_{L+1}}(x)=x+B_{L+1}.
\end{align}
 Then we can write a NN $g\in\pg_{{\rm FN}}$ as a functional composition of $\sigma^{B_l}$ and $W_l$, i.e.
\begin{align*}
    g:\IR^{p_0}\to\IR^{p_{L+1}},\quad g(x)=\sigma^{B_{L+1}}W_{L+1}\sigma^{B_l}\dots W_2\sigma^{B_1}W_1x.
\end{align*}
For $k\in\{1,\dots,L+1\}$ we define the functions
\begin{align}\label{eq:Akp}
    A_k^+g:\IR^{p_0}\to\IR^{p_k},\quad A_k^+g(x)&=\sigma^{B_{k}}W_k\dots\sigma^{B_1}W_1x,\\
    A_k^-g:\IR^{p_{k-1}}\to\IR^{p_{L+1}},\quad A_k^-g(x)&=\sigma^{B_{L+1}}W_{L+1}\dots\sigma^{B_k}W_k x. \label{eq:Akm}
\end{align}
Furthermore, set $A_0^+g=\Id_{\IR^{p_0}}$ and $A_{L+2}^-g=\Id_{\IR^{p_{L+1}}}$.
  For all $1\le l\le L+1$ holds
  \begin{align*}
    \norm[\infty]{\sigma^{B_l}(x)}&\le \norm[\infty]{x}+M\\
    \norm[\infty]{W^l(x)}&\le \norm[\infty]{W^l}\norm[\infty]{x}\le Mp\norm[\infty]{x}.
  \end{align*}
  We claim that for $k\in\{1,\dots,L+1\}$
  \begin{equation*}
    \sup_{x\in [-1,1]^{p_0}}  \lVert \Akp g(x)\rVert_{\infty}\le (M(p+1))^{k}
  \end{equation*}
  and proceed by induction. The case $k=0$ is trivial. To go from $k-1$ to $k$ we compute
  \begin{align}\label{entropyinduct}
    \sup_{x\in [-1,1]^{p_0}}\norm[\infty]{A_k^+x}
    &=\sup_{x\in [-1,1]^{p_0}}\norm[\infty]{\sigma^{B_k}W_k(\sigma^{B_{k-1}}W_{k-1}\cdots \sigma^{B_1}W_1 x)}\notag\\
    &\le\sup_{x\in [-(M(p+1))^{k-1},(M(p+1))^{k-1}]^{p_{k-1}}}\norm[\infty]{\sigma^{B_k}W^k x}\notag\\
    &\le (Mp(M(p+1))^{k-1}+M)\le (M(p+1))^k,
  \end{align}
  as claimed.

Moreover, for $l=1,\dots,L+1$, $W_l:(\R^{p_{l-1}},\norm[\infty]{\cdot})\to(\R^{p_{l}},\norm[\infty]{\cdot})$ is Lipschitz
with constant $Mp$ and $\sigma^{B_l}:(\R^{p_{l}},\norm[\infty]{\cdot})\to(\R^{p_{l}},\norm[\infty]{\cdot})$ is Lipschitz with constant $1$. Thus we can estimate the Lipschitz constant of $\Akm g$ for $k=1,\dots,L+1$. It holds
\begin{align}\label{eq:ReLULipschitz}
\left\lVert\Akm g(x)-\Akm g(y)\right\rVert_{\infty}&=\left\lVert \sigma^{B_{L+1}}W_{L+1}\dots \sigmabk W_kx-\sigma^{B_{L+1}}W_{L+1}\dots \sigmabk W_ky\right\rVert_{\infty}\notag\\
&\leq Mp\left\lVert \sigma^{B_{L}}W_{L}\dots  \sigmabk W_{k}x-\sigma^{B_{L}}W_{L}\dots \sigmabk  W_{k}y\right\rVert_{\infty}\notag\\
&\leq\dots\leq \left(Mp\right)^{L+2-k}\lVert x-y\rVert_{\infty}\quad{\rm for }\,\, x,y\in\IR^{p_{k-1}}.
\end{align}
Now let $g,g^*\in\pg_{{\rm FN}}$ be two NN such that $|w_{i,j}^l-w_{i,j}^{l,*}|<\eps$ and $|b_i^l-b_i^{l,*}|<\eps$ for all $i\le p_{l+1}$, $j\le p_l$, $l\le L+1$. 
Then 
\begin{align}\label{eq:ffstarReLU}
    \left\lVert g-g^*\right\rVert_{{\infty,\infty}}&\leq\sum_{k=1}^{L+1}\left\lVert A_{k+1}^- g\sigmabk W_k A_{k-1}^+ g^*-A_{k+1}^- g\sigma^{B_k^*} W_{k}^*A_{k-1}^+ g^*\right\rVert_{{\infty,\infty}}\notag\\
    &\leq \sum_{k=1}^{L+1} \left(Mp\right)^{L+1-k}\left\lVert \sigmabk W_k A_{k-1}^+ g^*-\sigma^{B_k^*} W_k^*A_{k-1}^+ g^*\right\rVert_{{\infty,\infty}}\notag\\
     &\leq \sum_{k=1}^{L+1} \left(Mp\right)^{L+1-k}\left(\left\lVert(W_k-W_k^*) A_{k-1}^+ g^*\right\rVert_{{\infty,\infty}}+\lVert B_k-B_k^*\rVert_{\infty}\right)\notag\\
    &\leq\eps\sum_{k=1}^{L+1} \left(Mp\right)^{L+1-k}\left(pM^{k-1}(p+1)^{k-1}+1\right)\notag\\
    & <\eps(L+1)M^{L}(p+1)^{L+1},
\end{align}
where we used \eqref{entropyinduct}, \eqref{eq:ReLULipschitz} and $M\geq 1$.
 The total number of weight and biases is less than $(L+1)(p^2+p)$. Therefore there are at most
\begin{align*}
    \begin{pmatrix}
    (L+1)(p^2+p)\\
    s\\
\end{pmatrix}\leq ((L+1)(p^2+p))^s
\end{align*}
combinations to pick $s$ nonzero parameters. Since all parameters are bounded by $M$, we choose $\eps=\delta/((L+1)M^L(p+1)^{L+1})$ and obtain the covering bound for all $\delta>0$
\begin{align}\label{eq:ReLUprooffinal}   N(\pg_{{\rm FN}},\|\cdot\|_{\infty,\infty},\delta)&\leq\max\left\{1,\sum_{s^*=1}^s\left(2M\epsilon^{-1}(L+1)(p^2+p)\right)^{s^*}\right\}\notag\\
&\leq \max\left\{1,\sum_{s^*=1}^s\left(2\delta^{-1}(L+1)M^{L+1}(p+1)^{L+1}(L+1)(p^2+p)\right)^{s^*}\right\}\notag\\
&\leq\max\left\{1,\sum_{s^*=1}^s\left(2\delta^{-1}(L+1)^2M^{L+1}(p+1)^{L+3}\right)^{s^*}\right\}\notag\\
&\leq\left(2^{L+6}L^2M^{L+1}p^{L+3}\max\left\{1,\delta^{-1}\right\}\right)^{s+1},
\end{align}
where we used $L\geq 1$ and $p\geq 1$ at the last inequality. Eqs.\ \eqref{eq:ReLUprooffinal} and \eqref{eq:gNNGNNrelation} show \eqref{eq:Hsig1general}.

Applying \eqref{eq:Hsig1general} to the sparse FrameNet class $\GNNSP[\sigma_1]$ gives  
\begin{align}\label{eq:entropyconstraint}
    &H\bigl(\GNNSP[\sigma_1],\|\cdot\|_{\infty,\sigma_R^r(U)},\delta\bigr)\notag\\
    \leq\,\,& ({\rm size}_N+1)\log\left(2^{{\rm depth}_N+6}\Lambda_{\pmb{\Psi}_{\Y}}{\rm depth}_N^2M^{{\rm depth}_N+1}{\rm width}_N^{{\rm depth}_N+4}\max\left\{1,\delta^{-1}\right\}\right)\notag\\
    \leq\,\,& (C_sN+1)\notag\\
    &\,\,\times \log\left(2^{C_L\log(N)+6}\Lambda_{\pmb{\Psi}_{\Y}}(C_L\log(N))^2M^{C_L\log(N)+1}(C_p N)^{C_L\log(N)+4}\max\left\{1,\delta^{-1}\right\}\right)\notag\\
    \leq\,\,&C_H^\trm{SP} N\left(1+\log(N)^2+\log\left(\max\left\{1,\delta^{-1}\right\}\right)\right),\quad N\in\IN,\quad\delta>0,
\end{align}
where we defined
\begin{align*}
    C_H^\trm{SP}&=2C_s\biggl((C_L+6)\log(2)+\log(\Lambda_{\pmb{\Psi}_{\Y}})+C_L^2+\\
    &\qquad+(C_L+1)\log(M)+(C_L+4)(\log(C_p)+1)\biggr).
\end{align*}

Applying \eqref{eq:Hsig1general} to the fully connected FrameNet class $\GNNFC[\sigma_1]$ gives 
\begin{align}\label{eq:entropyconstraintFC}
    &H\bigl(\GNNFC[\sigma_1],\|\cdot\|_{\infty,\sigma_R^r(U)},\delta\bigr)\notag\\
    \leq\,\,& (s^{\trm{FC}}(N)+1)\log\left(2^{{\rm depth}_N+6}\Lambda_{\pmb{\Psi}_{\Y}}{\rm depth}_N^2M^{{\rm depth}_N+1}{\rm width}_N^{{\rm depth}_N+4}\max\left\{1,\delta^{-1}\right\}\right)\notag\\
    \leq\,\,& \left(\left({\rm depth}_N+1\right)\left({\rm width}_N^2+{\rm width}_N\right)+1\right)\notag\\
    &\quad\times\log\left(2^{C_L\log(N)+6}\Lambda_{\pmb{\Psi}_{\Y}}(C_L\log(N))^2M^{C_L\log(N)+1}(C_p N)^{C_L\log(N)+4}\max\left\{1,\delta^{-1}\right\}\right)\notag\\
    \leq\,\,& \left(\left(C_L\log(N)+1\right)\left(C_p^2N^2+C_pN\right)+1\right)\notag\\
    &\quad\times\log\left(2^{C_L\log(N)+6}\Lambda_{\pmb{\Psi}_{\Y}}(C_L\log(N))^2M^{C_L\log(N)+1}(C_p N)^{C_L\log(N)+4}\max\left\{1,\delta^{-1}\right\}\right)\notag\\ 
    \leq\,\,&C_H^\trm{FC} N^2\left(1+\log(N)^3+\log\left(\max\left\{1,\delta^{-1}\right\}\right)\right),\quad N\in\IN,
\end{align}
where we defined 
\begin{align*}
C_H^\trm{FC}&=8C_LC_p^2\biggl((C_L+6)\log(2)+\log(\Lambda_{\pmb{\Psi}_{\Y}})+C_L^2+\\
    &\qquad+(C_L+1)\log(M)+(C_L+4)(\log(C_p)+1)\biggr).
\end{align*}
Equations \eqref{eq:entropyconstraint} and \eqref{eq:entropyconstraintFC} finish the proof of Lemma \ref{lem:entropybound}.

\subsection{Proof of Lemma \ref{lem:entropyboundRePU}}\label{app:proofentropyRePU}
The following proof is a modification of \cite[Proof of Lemma 5]{SH20} to the case where the activation function is not globally, but only locally Lipschitz continuous. The compactness of $(\pG,\|\cdot\|_{\infty,\supp(\gamma)})$ follows similarly to the ReLU case since \cite[Proposition 3.5]{Petersen.2021} holds for any continuous activation function.

Let $q\in\IN$, $q\geq 2$ and let $\sigma_q:\IR\to\IR$, $\sigma_q(x)=\max\{0,x\}^q$ denote the RePU activation function. Recall ${\rm depth}(g)\leq L$,
${\rm depth}(g)\leq p$,
${\rm size}(g)\leq s$ and ${\rm mpar}(g)\leq M$ for $g\in\pg_{\rm FN}$. We argue analogously to the ReLU-case in Lemma \ref{lem:entropybound} and bound the entropy of the NN class $\pg_{{\rm FN}}(\sigma_q,L,p,s,M,B)$. Recall the definitions of $\sigma^{B_l}$, $\Akp g$ and $\Akm g$ from \eqref{eq:sigmabl}-\eqref{eq:Akm}. Similar to \eqref{entropyinduct} it holds that
\begin{align}\label{eq:RePUinfinityboiund}
    \lVert \Akp g\rVert_{{\infty,\infty}}&=\sup_{x\in [-1,1]^{p_0}}  \lVert \Akp g(x)\rVert_{\infty}\notag\\
   &\leq\sup_{x\in [-M(p+1),M(p+1)]^{p_1}}\left\lVert\sigmabk W_k\sigma^{B_{k-1}}\dots W_2\sigma_{q}x\right\rVert_\infty\notag\\
    &\leq\sup_{x\in [-M^q(p+1)^q,M^q(p+1)^q]^{p_1}}\left\lVert\sigmabk W_k\sigma^{B_{k-1}}\dots W_2x\right\rVert_\infty\notag\\
&\leq\dots\leq \left(M(p+1))\right)^{\sum_{j=1}^k q^j}\leq \left(M(p+1)\right)^{q^{k+1}},
\end{align}
where we used $M\geq 1$.

In the RePU-case, $\Akm g$ is only locally Lipschitz: Since $|\sigma_q(x)'|\le q|x|^{q-1}$ it holds
  \begin{equation*}
    |\sigma_q(x)-\sigma_q(y)|\le q\max\{|x|,|y|\}^{q-1}|x-y|\qquad\forall x,y\in\R.
  \end{equation*}

Therefore for $k=1,\dots,L+1$ and $x,y\in\IR^{p_{k-1}}$, $\lVert x\rVert_{\infty}, \lVert y\rVert_{\infty}\leq C$, we get
\begin{align*}
&\left\lVert\Akm g(x)-\Akm g(y)\right\rVert_{\infty}\notag\\
=\,\,&\left\lVert \sigma^{B_{L+1}}W_{L+1}\sigma^{B_L}\dots  W_kx-\sigma^{B_{L+1}}W_{L+1}\sigma^{B_L}\dots  W_ky\right\rVert_{\infty}\notag\\
\leq\,\,& Mp\left\lVert \sigma^{B_L}W_L\sigma^{B_{L-1}}\dots  W_kx-\sigma^{B_L}W_L\sigma^{B_{L-1}}\dots  W_ky\right\rVert_{\infty}\notag\\
\leq\,\,& Mpq\left(\sup_{\lVert x\rVert_{\infty}\leq C}\left\lVert W_{L}\sigma^{B_{L-1}}\dots  W_kx\right\rVert_{\infty}\right)^{q-1}\notag\\
&\qquad \,\times\left\lVert W_{L}\sigma^{B_{L-1}}\dots W_kx-W_L\sigma^{B_{L-1}}\dots W_ky\right\rVert_{\infty}\notag\\
\leq\,\,& (Mpq)^{L+2-k}\left(\sup_{\lVert x\rVert_{\infty}\leq C}\left\lVert W_{L}\sigma^{B_{L-1}}\dots W_kx\right\rVert_{\infty}\right)^{q-1}\notag\\
&\,\,\,  \times\left(\sup_{\lVert x\rVert_{\infty}\leq C}\left\lVert W_{L-1}\sigma^{B_{L-2}}\dots  W_kx\right\rVert_{\infty}\right)^{q-1}\times \dots \times\left(\sup_{\lVert x\rVert_{\infty}\leq C}\left\lVert W_kx\right\rVert_{\infty}\right)^{q-1} \lVert x-y\rVert_{\infty}.
\end{align*}
Using 
\begin{align*}
    \sup_{\lVert x\rVert_{\infty}\leq C}\left\lVert W_{j}\sigma^{B_{j-1}}\dots W_kx\right\rVert_{\infty}&\leq     \sup_{\lVert x\rVert_{\infty}\leq \left(M(p+1)C\right)^q}\left\lVert W_{j}\sigma^{B_{j-1}}\dots W_{k+1}x\right\rVert_{\infty}\\
    &\leq \sup_{\lVert x\rVert_{\infty}\leq\left(M(p+1)\right)^{q+q^2}C^{q^2}}\left\lVert W_{j}\sigma^{B_{j-1}}\dots W_{k+2}x\right\rVert_{\infty}\\
    &\leq \dots\leq \sup_{\lVert x\rVert_{\infty}\leq\left(M(p+1)\right)^{\sum_{l=1}^{j-k}q^l}C^{q^{j-k}}}\left\lVert W_{j}x\right\rVert_{\infty}\\
    &\leq\left(M(p+1)\right)^{\sum_{l=0}^{j-k}q^l} C^{q^{j-k}}\leq \left(M(p+1)C\right)^{q^{j-k+1}}
\end{align*}
for $j=k,\dots,L$ and $C,M\geq 1$, we get
\begin{align}\label{eq:RePULipschitz}
    \left\lVert\Akm g(x)-\Akm g(y)\right\rVert_{\infty}&\leq (Mpq)^{L+2-k}\prod_{j=k}^{L}\left(M(p+1)\hat C\right)^{q^{j-k+1}}\lVert x-y\rVert_{\infty}\notag\\
    &\leq(Mpq)^{L+2-k}\left(M(p+1)\hat C\right)^{q^{L+2-k}}\lVert x-y\rVert_{\infty},\quad x,y\in\IR^{p_{k-1}}.
\end{align}
 Now we proceed similar to \eqref{eq:ffstarReLU}. Let $g,g^*\in\pg_{{\rm FN}}$ be two NN such that $|w_{i,j}^l-w_{i,j}^{l,*}|<\eps$ and $|b_i^l-b_i^{l,*}|<\eps$ for all $i\le p_{l+1}$, $j\le p_l$, $l\le L+1$.  Then with $A_0^+g=\Id_{\IR^{p_0}}$ and $A_{L+2}^-g=\Id_{\IR^{p_{L+1}}}$ we estimate 
 \begin{align}\label{eq:ffstarRePU}
    &\left\lVert g-g^*\right\rVert_{{\infty,\infty}}\notag\\
    \leq\,\,&\sum_{k=1}^{L+1}\left\lVert A_{k+1}^- g\sigmabk W_k A_{k-1}^+ g^*-A_{k+1}^- g\sigma^{B_k^*} W_k^*A_{k-1}^+ g^*\right\rVert_{{\infty,\infty}}\notag\\
   \leq\,\,&\sum_{k=1}^{L+1} \left(Mpq\right)^{L+1-k}\left(M(p+1)\left(M(p+1)\right)^{q^{k+1}}\right)^{q^{L+1-k}}\notag\\
   &\qquad \left\lVert\sigmabk W_k A_{k-1}^+ g^*-\sigma^{B_k^*} W_k^*A_{k-1}^+ g^*\right\rVert_{{\infty,\infty}}\notag\\
    \leq\,\,& \sum_{k=1}^{L+1}  \left(Mpq\right)^{L+1-k}\left(M(p+1)\left(M(p+1)\right)^{q^{k+1}}\right)^{q^{L+1-k}}\notag\\
    &\qquad \left(\left\lVert(W_k-W_k^*) A_{k-1}^+ g^*\right\rVert_{{\infty,\infty}}+\lVert B_k-B_k^*\rVert_{{\infty,\infty}}\right)q\left(M(p+1)\left\lVert A_{k-1}^+ g^*\right\rVert_{{\infty,\infty}}\right)^{q-1}\notag\\
    \leq\,\,& 2\eps\sum_{k=1}^{L+1}  \left(Mpq\right)^{L+2-k}\left(M(p+1)\left(M(p+1)\right)^{q^{k+1}}\right)^{q^{L+1-k}}\left\lVert A_{k-1}^+ g^*\right\rVert_{{\infty,\infty}}\notag\\
    &\qquad \left(M(p+1)\left\lVert A_{k-1}^+ g^*\right\rVert_{{\infty,\infty}}\right)^{q-1}\notag\\
    \leq\,\,& 2\eps(L+1)  \left(M(p+1)q\right)^{L+q}\left(M(p+1)\left(M(p+1)\right)^{q^{L+2}}\right)^{q^{L}}\left(\left(M(p+1)\right)^{q^{L+1}}\right)^q\notag\\
    <\,\,&  \eps L q^{L+q}\left(2pM\right)^{4q^{2L+2}}\left(2M\sqrt{p}(p^2+p)(L+1)\right)^{-1}.
\end{align}
In \eqref{eq:ffstarRePU} we used the Lipschitz bound \eqref{eq:RePULipschitz} with 
\begin{align*} C=\max\left\{1,\lVert\sigmabk W_k A_{k-1}^+ g^*\rVert_{{\infty,\infty}},\lVert\sigma^{B_k^*} W_k^* A_{k-1}^+ g^*\rVert_{{\infty,\infty}} \right\}\leq  \left(M(p+1)\right)^{q^{k+1}},
\end{align*}
and $p\geq1$, $L\geq 1$, $q\geq 2$ at the last inequality.

As in the proof of Lemma \ref{lem:entropybound}, there are 
\begin{align*}
    \begin{pmatrix}
    (L+1)(p^2+p)\\
    s\\
\end{pmatrix}\leq ((L+1)(p^2+p))^s
\end{align*}
combinations to pick $s$ nonzero weights and biases. Since all parameters are bounded by $M$, we choose 
\begin{align*}
     \eps=\frac{2M\sqrt{p}(p^2+p)(L+1)\delta}{L q^{L+q}\left(2pM\right)^{4q^{2L+2}}}
\end{align*}
and obtain the covering bound
\begin{align}  \label{eq:Entropyrepuprooffinal}
N(\pg_{{\rm FN}},\|\cdot\|_{\infty,\infty},\delta)&\leq\max\left\{1,\sum_{s^*=1}^{s}\left(2M\epsilon^{-1}(L+1)(p^2+p)\right)^{s^*}\right\}\notag\\
&\leq\max\left\{1,\sum_{s^*=1}^{s}\left(L q^{L+q}\left(2pM\right)^{4q^{2L+2}}(\sqrt p\delta)^{-1}\right)^{s^*}\right\}\notag\\
&\leq\left(L q^{L+q}\left(2pM\right)^{4q^{2L+2}}\sqrt p^{-1}\max\left\{1,\delta^{-1}\right\}\right)^{s+1}.
\end{align}
Eqs.~\eqref{eq:Entropyrepuprooffinal} and \eqref{eq:gNNGNNrelation} show \eqref{eq:Hsigqgeneral}.

Applying \eqref{eq:Hsigqgeneral} to the sparse FrameNet class $\GNNSP[\sigma_q]$ gives  
 \begin{align}\label{eq:entropyconstraintsigq}
    &H\bigl(\GNNSP,\|\cdot\|_{\infty,\sigma_R^r(U)},\delta\bigr)\notag\\
    \leq\,\,& (s^{SP}(N)+1)\log\left(\Lambda_{\pmb{\Psi}_{\Y}} {\rm depth}_N q^{{\rm depth}_N+q}\left(2{\rm width}_NM\right)^{4q^{2{\rm depth}_N+2}}\max\left\{1,\delta^{-1}\right\}\right)\notag\\
    \leq\,\,& (C_sN+1)\log\left(\Lambda_{\pmb{\Psi}_{\Y}}C_L\log(N)  q^{C_L\log(N)+q}\left(2C_pNM\right)^{4q^{2C_L\log(N)+2}}\max\left\{1,\delta^{-1}\right\}\right)\notag\\
    \leq\,\,&C_H^\trm{SP}N^{1+2C_L\log(q)}\left(1+\log(N)+\log\left(\max\left\{1,\delta^{-1}\right\}\right)\right),
\end{align}
where we set 
\begin{align*}
C_H^\trm{SP}&=2C_s\left(\log(\Lambda_{\pmb{\Psi}_{\Y}})+C_L+(C_L+q)\log(q)+4q^2\left(\log(2C_pM)+1\right)\right).
\end{align*}

Applying \eqref{eq:Hsigqgeneral} to the fully connected FrameNet class $\GNNFC[\sigma_q]$ gives the entropy bound
 \begin{align}\label{eq:entropyconstraintFCsigq}
    &H\bigl(\GNNFC,\|\cdot\|_{\infty,\sigma_R^r(U)},\delta\bigr)\notag\\
   &\qquad
       \leq (s^{FC}(N)+1)\log\left(\Lambda_{\pmb{\Psi}_{\Y}}{\rm depth}_N  q^{{\rm depth}_N+q}\left(2{\rm width}_NM\right)^{4q^{2{\rm depth}_N+2}}\max\left\{1,\delta^{-1}\right\}\right)\notag\\
    &\qquad \leq \left(\left({\rm depth}_N+1\right)\left({\rm width}_N^2+{\rm width}_N\right)+1\right)\notag\\
    &\qquad \quad\times\log\left(\Lambda_{\pmb{\Psi}_{\Y}} {\rm depth}_N q^{{\rm depth}_N+q}\left(2{\rm width}_NM\right)^{4q^{2{\rm depth}_N+2}}\max\left\{1,\delta^{-1}\right\}\right)\notag\\
&\qquad \leq\left(\left(C_L\log(N)+1\right)\left(C_p^2N^2+C_pN\right)+1\right)\notag\\
    &\qquad\quad\times\log\left(\Lambda_{\pmb{\Psi}_{\Y}}C_L\log(N)  q^{C_L\log(N)+q}\left(2C_pNM\right)^{4q^{2C_L\log(N)+2}}\max\left\{1,\delta^{-1}\right\}\right)\notag\\
   &\qquad\leq C_H^\trm{FC}N^{2+2C_L\log(q)}\left(1+\log(N)^2+\log\left(\max\left\{1,\delta^{-1}\right\}\right)\right),
 \end{align}
 where we set 
\begin{align*} C_H^\trm{FC}=&8C_sC_p^2\biggl(\log(\Lambda_{\pmb{\Psi}_{\Y}})+C_L
+(C_L+q)\log(q)+4q^2\left(\log(2C_pM)+1\right)\biggr).
\end{align*}
Equations \eqref{eq:entropyconstraintsigq} and \eqref{eq:entropyconstraintFCsigq} finish the proof of Lemma \ref{lem:entropyboundRePU}.

\section{Proofs of Section \ref{sec:applications}}
\subsection{Proof of Theorem \ref{thm:torus}}\label{app:prooftorus}

In \cite[Proof of Proposition 3, Step
  1]{Herrmann.2024}, the holomorphy in Assumption
  \ref{assump:holomorphicextension} is verified for $\X$, $\Y$ in
  \eqref{eq:cXscYt_torus} with $r_0>d/2$ and $t\in[0,(1+r_0-d/2-t_0)/d)$. Moreover, $\gamma=(\sigma_R^r)_{\#}\pi$ in particular shows $\supp(\gamma)\subseteq C_R^r(\X)$ and hence verifies the second part of Assumption \ref{assump:holomorphicextension}. Substituting $\mathfrak{s}=r_0+rd$, i.e.~$r=\frac{\mathfrak{s}-r_0}{d}$,
and taking $t=(1+r_0-d/2-t_0)/d-\tau$ with some small $\tau$, Theorem \ref{th:RgnG0boundWN} \ref{item:sparseReLU} then gives 
 \begin{align*}
   \IE_{\G_0}[\lVert \hat{G}_n-\G_0\rVert_{L^2(\gamma)}^2]   
   &\leq C n^{-\frac{\kappa}{\kappa+1}+\tau},
 \end{align*}
 where
 \begin{align*}
\kappa=2\min\biggl\{\frac{\mathfrak{s}-r_0}{d}-\frac{1}{2},\frac{1+r_0-\frac{d}{2}-t_0}{d}\biggr\}-\tau
 \end{align*}
 for all $r_0>d/2$ and $t_0\in[0,1]$.

 From here on the proof is essentially the same as \cite[Proof of Proposition 3, Step 2]{Herrmann.2024}; the only difference is that while \cite{Herrmann.2024} uses the uniform bound in Theorem \ref{th:approximationNN} \ref{item:Linfapprox}, we require the $L^2$-bound in Theorem \ref{th:approximationNN} \ref{item:Ltwoapprox}. For completeness, we repeat the argument. 
The constraint $r>1$ implies $\mathfrak{s}>r_0+d$ on
$\mathfrak{s}$. We now choose $r_0>\frac{d}{2}$ in order to maximize the convergence rate.
Solving
\begin{equation*}
  \frac{\mathfrak{s}-r_0}{d}-\frac{1}{2} = \frac{1+r_0-\frac{d}{2}-t_0}{d}
\end{equation*}
for $r_0$ gives
\begin{equation}\label{eq:s0}
  r_0 = \frac{\mathfrak{s}+t_0-1}{2}.
\end{equation}
The constraint $r_0>\frac{d}{2}$ implies the constraint
$\mathfrak{s}>d+1-t_0$.

We look at two cases separately. First, if
$\mathfrak{s}\in (\frac{3d}{2},2d+1-t_0]$, we set
$r_0:=\frac{d}{2}+\tau_2$, where we choose $\tau_2>0$ 
s.t.~$\tau_2<\mathfrak{s}-3d/2$ which guarantees $\mathfrak{s}>r_0+d$. For $\tau<\tau_2/d$, we obtain the convergence rate
\begin{equation*}
  \kappa=2\min\Big\{\frac{\mathfrak{s}-\frac{d}{2}-\tau_2}{d}-\frac{1}{2},\frac{1+\frac{d}{2}+\tau_2-\frac{d}{2}-t_0}{d}\Big\}-\tau
  \geq 2\min\Big\{\frac{\mathfrak{s}}{d}-1-\frac{2\tau_2}{d},\frac{1-t_0}{d}\Big\}.
\end{equation*}

In the case $\mathfrak{s}>2d+1-t_0$, define $r_0$ as in \eqref{eq:s0}.
The constraint $\mathfrak{s}>r_0+d$ amounts to
\begin{equation*}
  \mathfrak{s}>\frac{\mathfrak{s}+t_0-1}{2}+d
  \qquad\Leftrightarrow\qquad
  \mathfrak{s}>2d+t_0-1,
\end{equation*}
which holds since $\mathfrak{s}> 2d+1-t_0
\ge 2d+t_0-1$ for all $t_0\in [0,1]$.
In this case we get the convergence rate
\begin{equation*}
  \kappa=2\frac{\mathfrak{s}-r_0}{d}-1-\tau = \frac{\mathfrak{s}+1-t_0}{d}-1-\tau.
\end{equation*}
Choosing $\tau_1>8\tau_2/d>8\tau$ shows \eqref{eq:Rtorus} and finishes the proof of Theorem \ref{thm:torus}.
\end{appendices}
\end{document}